\documentclass[11pt]{article}

\usepackage{array}
\usepackage{amsfonts}
\usepackage{amscd}
\usepackage{amssymb}
\usepackage{amsthm}
\usepackage{amsmath}
\usepackage{stmaryrd}
\usepackage{graphicx}
\usepackage{color}
\usepackage{verbatim}
\usepackage{mathrsfs}
\usepackage{pstricks,tikz}
\usetikzlibrary{calc}
\usepackage{caption}
\usepackage[neveradjust]{paralist}
\usepackage[a4paper, total={15cm, 25cm},centering]{geometry}
\usepackage{float}
\usepackage[absolute]{textpos}
\usepackage{subcaption}
\usepackage{ytableau}
\usepackage{hyperref}

\let\OLDthebibliography\thebibliography
\renewcommand\thebibliography[1]{
  \OLDthebibliography{#1}
  \setlength{\parskip}{0pt}
  \setlength{\itemsep}{0pt plus 0.3ex}
}

\definecolor{green1}{RGB}{153,216,201}
\definecolor{green2}{RGB}{44,190,95}

\definecolor{blue1}{RGB}{158,202,225}
\definecolor{blue2}{RGB}{49,130,189}

\definecolor{RedOrange}{cmyk}{0,0.77,0.87,0} 
\definecolor{Mahogany}{cmyk}{0,0.85,0.87,0.35} 
\definecolor{Maroon}{cmyk}{0,0.87,0.68,0.32} 
\definecolor{BrickRed}{cmyk}{0,0.89,0.94,0.28} 
\definecolor{Red}{cmyk}{0,1.,1.,0} 
\definecolor{OrangeRed}{cmyk}{0,1.,0.50,0} 

\definecolor{purple}{rgb}{0.8,0.12,0.8}
\definecolor{orange}{rgb}{1.0,0.7,0.0}
\definecolor{pink}{rgb}{1,0.5,0.8}
\definecolor{blackg}{rgb}{0.1,0.25,0.1}
\definecolor{ForestGreen}{cmyk}{0.91,0,0.88,0.42}
\definecolor{Turquoise}{cmyk}{0.85,0,0.20,0}

 \theoremstyle{plain}

\newtheorem{thm}{Theorem}[section]
\newtheorem{lemma}[thm]{Lemma}
\newtheorem{prop}[thm]{Proposition}
\newtheorem{cor}[thm]{Corollary}
\newtheorem{conjecture}[thm]{Conjecture}

\theoremstyle{definition}
\newtheorem{defn}[thm]{Definition}

\newtheorem{remark}[thm]{Remark}
\newtheorem{example}[thm]{Example}

\newtheorem{convention}[thm]{Convention}
\numberwithin{equation}{section}

\newcommand{\Wfin}{W_0}
\newcommand{\Waff}{W}
\newcommand{\Wext}{\widetilde{W}}
\newcommand{\Haff}{\mathcal{H}}
\newcommand{\Hext}{\widetilde{\mathcal{H}}}
\newcommand{\Ifin}{I}
\newcommand{\Iaff}{\{0\}\cup I}
\newcommand{\WJ}{\mathbb{W}^J}
\newcommand{\TJ}{\mathbb{T}_J}
\newcommand{\zJ}{\zeta_J}
\newcommand{\WJaff}{W_J^{\mathrm{aff}}}
\newcommand{\WKaff}{W_K^{\mathrm{aff}}}
\newcommand{\HJaff}{\mathcal{H}_J^{\mathrm{aff}}}
\newcommand{\HKaff}{\mathcal{H}_K^{\mathrm{aff}}}
\newcommand{\Jaff}{J^{\mathrm{aff}}}
\newcommand{\st}{\tau}

\newcommand{\Mat}{\mathrm{Mat}}

\def\sA{\mathsf{A}}
\def\sB{\mathsf{B}}
\def\sC{\mathsf{C}}

\def\sF{\mathsf{F}}
\def\sG{\mathsf{G}}

\def\sBC{\mathsf{BC}}

\def\cA{\mathcal{A}}

\def\cH{\mathcal{H}}

\def\cK{\mathcal{K}}
\def\cL{\mathcal{L}}

\def \cP{\mathcal{P}}
\def\cQ{\mathcal{Q}}

\def\nZ{\mathbb{Z}}
\def\nN{\mathbb{N}}

\def\ssA{\mathsf{A}}

\def\ssC{\mathsf{C}}
\def\ssBC{\mathsf{BC}}

\newcommand{\la}{\lambda}

\newcommand{\Ga}{\Gamma}

\def\CC{\mathbb{C}}

\def\NN{\mathbb{N}}

\def\RR{\mathbb{R}}

\def\ZZ{\mathbb{Z}}

\DeclareMathOperator\height{\mathrm{ht}}

\def\la{\lambda}

\def\<{\langle}
\def\>{\rangle}

\def\la{\lambda}

\def\sR{\mathsf{R}}

\def\bm{\mathbf{m}}

\newcommand{\fc}{\mathfrak{c}}

\newcommand{\sw}{\mathsf{w}}
\newcommand{\sy}{\mathsf{y}}

\newcommand{\sv}{\mathsf{v}}

\newcommand{\wt}{\mathrm{wt}}

\newcommand{\ba}{\mathbf{a}}

\newcommand{\sq}{\mathsf{q}}

\makeatletter
\renewcommand{\@makefnmark}{\mbox{\textsuperscript{}}}
\makeatother

\newcommand{\bla}[1]{\textcolor{gray}{ #1}}

\title{On $J$-folded alcove paths and combinatorial representations of affine Hecke algebras}
\author{J\'er\'emie Guilhot, Eloise Little, James Parkinson}
\date{\today}

\begin{document}

\maketitle

\vspace{-0.5cm}

\begin{abstract}
We introduce the combinatorial model of $J$-folded alcove paths in an affine Weyl group and construct representations of affine Hecke algebras using this model. We study boundedness of these representations, and we state conjectures linking our combinatorial formulae to Kazhdan-Lusztig theory and Opdam's Plancherel Theorem.
\end{abstract}

\tableofcontents

\section{Introduction}

Path models have a long history in combinatorial representation theory, starting with the celebrated Littelmann path model for computing multiplicities in the representation theory of symmetrisable Kac-Moody algebras~\cite{Lit:94,Lit:95}. In \cite{GL:05} Gaussent and Littelmann adapted the path model to study Mirkovic-Vilonen cycles in the affine Grassmannian. The work of Ram~\cite{Ram:06} and  C. Schwer~\cite{Sch:06} on positively folded alcove paths connects path models to the representation theory of affine Hecke algebras, and the work of Parkinson, Ram and C. Schwer~\cite{PRS:09} uses labelled positively folded alcove paths to index points in the affine Grassmannian. In \cite{GP:19,GP:19b} Guilhot and Parkinson used alcove paths as a primary tool to prove Lusztig's conjectures P1--P15 for $\tilde{\sG}_2$ and $\tilde{\sC}_2$ with arbitrary parameters. In \cite{MST:19} Mili\'cevi\'c, P. Schwer and Thomas~\cite{MST:19} apply alcove paths to the study of affine Diligne-Lusztig varieties, and in~\cite{MNST:22} Mili\'cevi\'c, Naqvi, P. Schwer and Thomas develop a path model for certain double coset intersections, and relate these to chimney retractions in affine buildings.

In this paper we introduce the combinatorial model of \textit{$J$-folded alcove paths} in an affine Weyl group, and use this model to construct and analyse representations of affine Hecke algebras with arbitrary parameters. This gives rise to large class of representations of affine Hecke algebras admitting a combinatorial description in terms of positively $J$-folded alcove paths. On the one hand this generalises known formulae for the principal series representation and Macdonald spherical functions in terms of (classical) positively folded alcove paths (see \cite{Ram:06,Sch:06}), and on the other hand these combinatorial representations have recently found applications in Kazhdan-Lusztig theory (see \cite{CGLP:24,GP:19,GP:19b}). Indeed, a primary motivation of the the present paper is to extend the rank~$2$ analysis from \cite{GP:19,GP:19b} to arbitrary rank, and to provide conjectural connections between the combinatorics of $J$-folded alcove paths and Kazhdan-Lusztig theory.

Let $\Wext$ be an irreducible extended affine Weyl group with spherical Weyl group~$\Wfin$ and spherical root system~$\Phi$ in an $n$-dimensional Euclidean vector space~$V$ with simple roots $\alpha_1,\ldots,\alpha_n$. For the purpose of this introduction we will assume that $\Phi$ is reduced (that is, if $\alpha,k\alpha\in\Phi$ then $k=\pm 1$). This simplifying assumption excludes the $\sBC_n$ case, however the main results of the paper also hold for this important case.  Let  $\Hext$ be the associated extended affine Hecke algebra defined over the ring 
$\sR=\ZZ[\sq_0^{\pm 1},\ldots,\sq_n^{\pm 1}]$ where $\sq_0,\ldots,\sq_n$ are commuting invertible indeterminates subject to the constraint $\sq_i=\sq_j$ whenever $s_i$ and $s_j$ are conjugate in~$\Wext$ (where $s_0,s_1,\ldots,s_n$ are the Coxeter generators of the non-extended affine Weyl group). 

The combinatorial representations $\pi_{J,\sv}$ of~$\Hext$ constructed in this paper depend on two pieces of data: a subset $J\subseteq\Ifin$ (with $\Ifin$ an indexing set for the generators of $W_0$) and a \textit{$J$-parameter system}~$\sv$. The latter is a  family $\sv=(\sv_{\alpha})_{\alpha\in\Phi_J}$ such that if $j\in J$ then $\sv_{\alpha_j}\in\{\sq_j,-\sq_j^{-1}\}$, and $\sv_{\alpha}=\sv_{\beta}$ whenever $\alpha,\beta\in\Phi_J$ with $\beta\in W_J\alpha$ (here $W_J$ is the $J$-parabolic subgroup of~$\Wfin$ and $\Phi_J$ the associated sub-root system of~$\Phi$).  As mentioned above, our construction relies on the notion of positively $J$-folded alcove paths: these are folded alcove paths (in the classical sense of~\cite{Ram:06}) satisfying a certain positivity condition that are confined to remain in the \textit{fundamental $J$-alcove}
$$
\cA_J=\{x\in V\mid 0\leq \langle x,\alpha\rangle\leq 1\text{ for all }\alpha\in\Phi_J^+\}
$$
(see Figure~\ref{fig:G2example0} where $\cA_J$ in type $\tilde{\sG}_2$ with $J=\{1\}$ is shaded green, and a positively folded $J$-alcove path is illustrated). This constraint forces the path to `bounce' on the walls of the fundamental $J$-alcove, as illustrated in the fourth, fourteenth, and twenty first steps of the path in Figure~\ref{fig:G2example0}. These bounces play a different role in the theory to the folds (illustrated in steps eight and twenty four in the figure), and are a new feature of our model.

\begin{figure}[H]
\centering
\begin{tikzpicture}[scale=1.2]
\clip ({-4*0.866+0.433},{-2}) rectangle + ({6*0.866},{5.5});
\path [fill=green2] ({-4*0.866},-3)--({-2*0.866},-3)--({3*0.866},4.5)--({1*0.866},4.5)--({-4*0.866},-3);
\node at (0.15,0.6) {$e$};
\draw[line width=2pt,-latex](0,0)--({-3*0.866},1.5);
\draw[line width=2pt,-latex](0,0)--({2*0.866},0);
\node at ({-3*0.866+0.5},1.5) {$\alpha_1^{\vee}$};
\node at ({2*0.866},{0.5}) {$\alpha_2^{\vee}$};
\draw(-4.33,4.5)--(4.33,4.5);
\draw(-4.33,3)--(4.33,3);
\draw(-4.33,1.5)--(4.33,1.5);
\draw(-4.33,0)--(4.33,0);
\draw(-4.33,-1.5)--(4.33,-1.5);
\draw(-4.33,-3)--(4.33,-3);
\draw(-4.33,-3)--(-4.33,4.5);
\draw(-3.464,-3)--(-3.464,4.5);
\draw(-2.598,-3)--(-2.598,4.5);
\draw(-1.732,-3)--(-1.732,4.5);
\draw(-.866,-3)--(-.866,4.5);
\draw(0,-3)--(0,4.5);
\draw(.866,-3)--(.866,4.5);
\draw(1.732,-3)--(1.732,4.5);
\draw(2.598,-3)--(2.598,4.5);
\draw(3.464,-3)--(3.464,4.5);
\draw(4.33,-3)--(4.33,4.5);
\draw(-4.33,3.5)--({-3*0.866},4.5);
\draw(-4.33,2.5)--({-1*0.866},4.5);
\draw(-4.33,1.5)--({1*0.866},4.5);
\draw(-4.33,.5)--({3*0.866},4.5);
\draw(-4.33,-.5)--(4.33,4.5);
\draw(-4.33,-1.5)--(4.33,3.5);
\draw(-4.33,-2.5)--(4.33,2.5);
\draw(-3.464,-3)--(4.33,1.5);
\draw(-1.732,-3)--(4.33,.5);
\draw(0,-3)--(4.33,-.5);
\draw(1.732,-3)--(4.33,-1.5);
\draw(3.464,-3)--(4.33,-2.5);
\draw(4.33,3.5)--({3*0.866},4.5);
\draw(4.33,2.5)--({1*0.866},4.5);
\draw(4.33,1.5)--({-1*0.866},4.5);
\draw(4.33,.5)--({-3*0.866},4.5);
\draw(4.33,-.5)--(-4.33,4.5);
\draw(4.33,-1.5)--(-4.33,3.5);
\draw(4.33,-2.5)--(-4.33,2.5);
\draw(3.464,-3)--(-4.33,1.5);
\draw(1.732,-3)--(-4.33,.5);
\draw(0,-3)--(-4.33,-.5);
\draw(-1.732,-3)--(-4.33,-1.5);
\draw(-3.464,-3)--(-4.33,-2.5);
\draw(-4.33,-1.5)--(-3.464,-3);
\draw(-4.33,1.5)--(-1.732,-3);
\draw(-4.33,4.5)--(0,-3);
\draw({-3*0.866},4.5)--(1.732,-3);
\draw({-1*0.866},4.5)--(3.464,-3);
\draw({1*0.866},4.5)--(4.33,-1.5);
\draw({3*0.866},4.5)--(4.33,1.5);
\draw(4.33,-1.5)--(3.464,-3);
\draw(4.33,1.5)--(1.732,-3);
\draw(4.33,4.5)--(0,-3);
\draw({3*0.866},4.5)--(-1.732,-3);
\draw({1*0.866},4.5)--(-3.464,-3);
\draw({-1*0.866},4.5)--(-4.33,-1.5);
\draw({-3*0.866},4.5)--(-4.33,1.5);
\draw[line width=2pt]({-5*0.866},1.5)--({-3*0.866},4.5);
\draw[line width=2pt]({-5*0.866},-1.5)--({-1*0.866},4.5);
\draw[line width=2pt]({-4*0.866},-3)--({1*0.866},4.5);
\draw[line width=2pt]({-2*0.866},-3)--({3*0.866},4.5);
\draw[line width=2pt]({0*0.866},-3)--({5*0.866},4.5);
\draw[line width=2pt]({2*0.866},-3)--({5*0.866},1.5);
\draw[line width=2pt]({4*0.866},-3)--({5*0.866},-1.5);
\draw[line width=1pt,-latex](-0.47,0.45)--(-0.68,0.15);
\draw[line width=1pt,-latex](-0.68,0.15)--(-0.68,-0.2);
\draw[line width=1pt,-latex](-0.68,-0.2)--(-0.47,-0.45);
\draw[line width=1pt,-latex] plot[smooth] coordinates {(-0.47,-0.45) (-0.38,-0.52) (-0.42,-0.6) (-0.55,-0.54)};
\draw[line width=1pt,-latex](-0.55,-0.54)--(-0.75,-0.9);
\draw[line width=1pt,-latex](-0.75,-0.9)--(-1.02,-0.9);
\draw[line width=1pt,-latex](-1.02,-0.9)--(-1.3,-1.1);
\draw[line width=1pt,-latex] plot[smooth] coordinates {(-1.3,-1.1) (-1.36,-1.17) (-1.45,-1.12) (-1.35,-1)};
\draw[line width=1pt,-latex](-1.35,-1)--(-1.02,-0.78);
\draw[line width=1pt,-latex](-1.02,-0.78)--(-1.25,-0.45);
\draw[line width=1pt,-latex](-1.25,-0.45)--(-1.05,-0.15);
\draw[line width=1pt,-latex](-1.05,-0.15)--(-1.05,0.15);
\draw[line width=1pt,-latex](-1.05,0.15)--(-1.3,0.4);
\draw[line width=1pt,-latex] plot[smooth] coordinates {(-1.3,0.4) (-1.4,0.45) (-1.36,0.55) (-1.2,0.45)};
\draw[line width=1pt,-latex](-1.2,0.45)--(-1,0.8);
\draw[line width=1pt,-latex](-1,0.8)--(-0.7,0.8);
\draw[line width=1pt,-latex](-0.7,0.8)--(-0.35,1);
\draw[line width=1pt,-latex](-0.35,1)--(-0.15,1.3);
\draw[line width=1pt,-latex](-0.15,1.3)--(0.22,1.3);
\draw[line width=1pt,-latex](0.22,1.3)--(0.38,1.02);
\draw[line width=1pt,-latex] plot[smooth] coordinates {(0.38,1.02)(0.48,0.96)(0.55,1.05)(0.42,1.125)};
\draw[line width=1pt,-latex](0.44,1.12)--(0.3,1.35);
\draw[line width=1pt,-latex](0.3,1.35)--(0.3,1.65);
\draw[line width=1pt,-latex] plot[smooth] coordinates {(0.3,1.65)(0.05,1.65)(0.05,1.72)(0.28,1.72)};
\draw[line width=1pt,-latex](0.28,1.72)--(0.48,1.91);
\draw[line width=1pt,-latex](0.48,1.91)--(0.75,2.1);
\draw[line width=1pt,-latex](0.75,2.1)--(0.5,2.5);
\draw[line width=1pt,-latex](0.5,2.5)--(0.68,2.79);
\node at ({-4*0.866},{2*1.5}) {\Large{$\bullet$}};
\node at ({-2*0.866},{2*1.5}) {\Large{$\bullet$}};
\node at ({0*0.866},{2*1.5}) {\Large{$\bullet$}};
\node at ({2*0.866},{2*1.5}) {\Large{$\bullet$}};
\node at ({4*0.866},{2*1.5}) {\Large{$\bullet$}};
\node at ({-3*0.866},{1.5}) {\Large{$\bullet$}};
\node at ({-1*0.866},{1.5}) {\Large{$\bullet$}};
\node at ({1*0.866},{1.5}) {\Large{$\bullet$}};
\node at ({3*0.866},{1.5}) {\Large{$\bullet$}};
\node at ({-4*0.866},0) {\Large{$\bullet$}};
\node at ({-2*0.866},0) {\Large{$\bullet$}};
\node at ({0*0.866},0) {\Large{$\bullet$}};
\node at ({2*0.866},0) {\Large{$\bullet$}};
\node at ({4*0.866},0) {\Large{$\bullet$}};
\node at ({-3*0.866},{-1*1.5}) {\Large{$\bullet$}};
\node at ({-1*0.866},{-1*1.5}) {\Large{$\bullet$}};
\node at ({1*0.866},{-1*1.5}) {\Large{$\bullet$}};
\node at ({3*0.866},{-1*1.5}) {\Large{$\bullet$}};
\end{tikzpicture}
\caption{The fundamental $J$-alcove and a positively $J$-folded alcove path, with $J=\{1\}$}\label{fig:G2example0}
\end{figure}
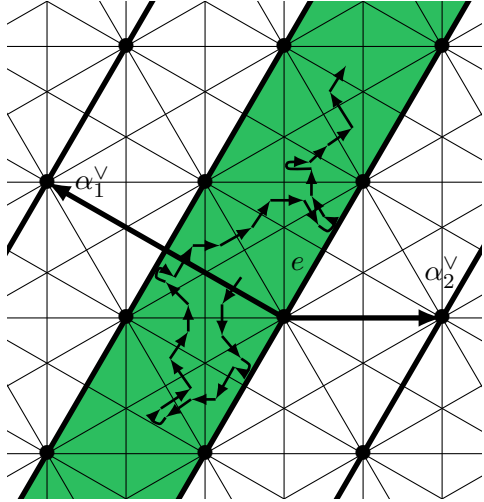

Our representations $\pi_{J,\sv}$ are constructed in Theorem~\ref{thm:module}, and are inspired by the work of Deodhar~\cite{Deo:87,Deo:90,Deo:91,Deo:97} (see in particular \cite[\S2]{Deo:87}) and Lusztig~\cite[Lemma~4.7]{Lus:97}, and in Theorem~\ref{thm:mainpath1} we obtain an explicit formula for the matrices $\pi_{J,\sv}(T_w)$ in terms of the positively $J$-folded alcove paths described above (here $T_w$ denotes the standard basis of $\Hext$).

Parabolic induction plays a central role in the representation theory of Hecke algebras, and in the case of affine Hecke algebras this parabolic induction takes the form of induction of finite dimensional representations of Levi subalgebras. Given a subset $J\subseteq I$ and a $J$-parameter system~$\sv$, one can define a generic 1-dimensional representation of the associated Levi subalgebra. Our second main result is that the set of representations obtained by inducing these 1-dimensional representations is exactly the set of representations~$\pi_{J,\sv}$.

Our combinatorial constructions are connected to Kazhdan-Lusztig theory via the notion of \textit{bounded} representations, inspired by the work of Geck~\cite{Geck:11}. To define this concept, we assume that there are integers $a_0,\ldots,a_n>0$ such that $\sq_i=\sq^{a_i}$. Thus $\sR=\ZZ[\sq,\sq^{-1}]$, and one may talk of degree (in $\sq$) of elements of $\sR$. The matrix representation $\pi_{J,\sv}$ is called \textit{bounded} if the degree of the matrix entries of $\pi_{J,\sv}(T_w)$ are uniformly bounded, for all $w\in \Wext$. The \textit{bound} $\ba_{J,\sv}$ of a bounded representation $\pi_{J,\sv}$ is the maximum of the degrees of the matrix entries of all $\pi_{J,\sv}(T_w)$, $w\in\Wext$. Our third main result is the complete classification of bounded representations $\pi_{J,\sv}$.

There is strong evidence from \cite{CGLP:24,GP:19,GP:19b} that bounded representations have deep connections with Kazhdan-Lusztig theory and Opdam's Plancherel Theorem, and we discuss these connections in Section~\ref{sec:conj}. In particular, we give a general conjectural formula for the bound $\ba_{J,\sv}$ in terms of Macdonald's $c$-function~\cite{Mac:71} and Opdam's Plancherel Theorem~\cite{Opd:04} and make conjectural links between the bound $\ba_{J,\sv}$ and Kazhdan-Lusztig theory.

To summarise, the main results of this paper are as follows.

 \begin{compactenum}
\item[\raisebox{0.35ex}{\tiny$\bullet$}] For each subset $J \subseteq I$  and each $J$-parameter system $\sv$ we construct a finite dimensional representation $\pi_{J,\sv}$ of the extended affine Hecke algebra $\Hext$ together with a combinatorial description in terms of  $J$-folded alcove paths (see Theorem \ref{thm:module} and Theorem~\ref{thm:mainpath1})
\item[\raisebox{0.35ex}{\tiny$\bullet$}] We show that our representations $\pi_{J,\sv}$ are isomorphic to induced representations from generic 1-dimensional representations of Levi subalgebras, and that all such induced representations arise from our combinatorial construction (see Theorem \ref{thm:induced}).
\item[\raisebox{0.35ex}{\tiny$\bullet$}] We classify all the subsets $J \subseteq I$ and all the $J$-parameter systems $\sv$ for which the representation $\pi_{J,\sv}$ is bounded (see Theorem \ref{thm:bound} and Proposition \ref{prop:classifybounded}). 
\end{compactenum}
\smallskip

We now give a brief overview of the structure of the paper. In Section~\ref{sec:background} we give background material on root systems, Weyl groups, positively folded alcove paths, and affine Hecke algebras. In Section~\ref{sec:Jalcoves} we study the geometry of the fundamental $J$-alcove in preparation for Section~\ref{sec:Jpaths} where we introduce and develop the combinatorial model of $J$-folded alcove paths. In Section~\ref{sec:TheModule} we construct our $\Hext$-modules $M_{J,\sv}$ (Theorem~\ref{thm:module}), and develop a combinatorial formula for the matrix entries of this representation in terms of positively $J$-folded alcove paths (Theorems~\ref{thm:mainpath1} and~\ref{thm:mainpath2}), and prove that these modules are induced from Levi subalgebras (Theorem~\ref{thm:induced}). In Section~\ref{sec:boundedreps} we recall the notion of bounded representations from \cite{GP:19,GP:19b} for weighted affine Hecke algebras, and we classify the subsets $J\subseteq\Ifin$ and $J$-parameter systems $\sv$ for which $M_{J,\sv}$ is bounded (Theorem~\ref{thm:bound}). We apply the theory of $J$-folded alcove paths to study the bound $\ba_{J,\sv}$ of $M_{J,\sv}$, and state conjectural formulae for $\ba_{J,\sv}$ and conjectural links with Lusztig's $\ba$-function and Kazhdan-Lusztig cells (see Conjectures~\ref{conj:strongconjecture} and~\ref{conj:cells}). We verify these conjectures in various cases (including $2$-dimensional affine Hecke algebras with general parameters). In Section~\ref{sec:An} we consider the $\tilde{\sA}_n$ case with $J=\{1,2,\ldots,n-1\}$, giving both an example of the theory developed in the paper, and verifying our conjectures for this case.

Finally, we note that throughout the paper we work in the general setting of multiparameter affine Hecke algebras, with the non-reduced root systems of type $\sBC_n$ being employed to deal with the $3$-parameter affine Hecke algebras with affine Weyl group of type $\tilde{\sC}_n$ (and the $2$-parameter $\tilde{\sA}_1$ Hecke algebras, see Convention~\ref{conv:parameters}). At times this level of generality leads to more complicated formulae. However since this work is primarily directed towards understanding Kazhdan-Lusztig theory for Hecke algebras with unequal parameters (where the deep geometric interpretations of Kazhdan and Lusztig~\cite{KL:79} and the associated positivity of Elias and Williamson~\cite{EW:14} for the equal parameter case typically do not hold), and since the $3$-parameter affine Hecke algebras are the most extreme cases of such algebras, we believe that this level of generality is warranted and valuable. As a general guide to the reader unacquainted with the non-reduced setting one may wish to assume the reduced case on first reading, in which case all symbols like $\sv_{2\alpha}$ or $\sq_{2\alpha}$ may be read as being~$1$.

\section{Background}\label{sec:background}

This section contains background on root systems, Weyl groups, positively folded alcove paths, affine Hecke algebras and the Bernstein-Lusztig presentation. Our main references are~\cite{Bou:02} (for root systems and Weyl groups, \cite{PRS:09,Ram:06} (for positively folded alcove paths), and \cite{Lus:89,NR:03,Ram:06} (for the affine Hecke algebra and Bernstein-Lusztig presentation).

\subsection{Root systems}\label{sec:root}

Let $\Ifin=\{1,2,\ldots,n\}$. Let $\Phi$ be an irreducible, not necessarily reduced, crystallographic root system of rank~$n$ in a real vector space $V$ with bilinear form $\langle\cdot,\cdot\rangle$. For $\alpha\in V\backslash\{0\}$ let $\alpha^{\vee}=2\alpha/\langle\alpha,\alpha\rangle$, and let $\Phi^{\vee}=\{\alpha^{\vee}\mid \alpha\in\Phi\}$ be the dual root system. Let $\{\alpha_i\mid i\in \Ifin\}$ be a system of simple roots, and let $\Phi^+$ denote the corresponding set of positive roots.  We will adopt Bourbaki conventions~\cite{Bou:02} when labelling the simple roots. 
If $\alpha=\sum_{i\in \Ifin}a_i\alpha_i\in\Phi$ let $\height(\alpha)=\sum_{i\in\Ifin}a_i$ denote the \textit{height} of~$\alpha$. Let $\varphi\in\Phi$ denote the \textit{highest root} of~$\Phi$, and define integers $m_i\geq 1$ by $\varphi=m_1\alpha_1+\cdots+m_n\alpha_n$. 

The \textit{coroot lattice} of $\Phi$ is the $\nZ$-lattice $Q$ spanned by $\Phi^{\vee}$. The \textit{fundamental coweights} of $\Phi$ are the elements $\omega_i\in V$, $i\in\Ifin$, with $\langle \omega_i,\alpha_j\rangle=\delta_{i,j}$. The \textit{coweight lattice} of $\Phi$ is the $\nZ$-lattice~$P$ spanned by the fundamental coweights, that is $P=\nZ\omega_1+\cdots+\nZ\omega_n$. The set of positive coweights is $P^+=\nN\omega_1+\cdots+\nN\omega_n$. Note that $Q\subseteq P$.

\begin{remark}\label{rem:reduced}
A root system $\Phi$ is called \textit{reduced} if $\alpha\in\Phi$ and $k\alpha\in\Phi$ implies that $k\in\{-1,1\}$. In any irreducible reduced root system there are at most two root lengths (the \textit{long roots} and the \textit{short roots}, with all roots considered long if there is only one root length). For each $n\geq 1$ there is a unique non-reduced irreducible crystallographic root system $\Phi$ of rank~$n$ up to isomorphism, denoted $\sBC_n$. Explicitly we can take $V=\mathbb{R}^n$ with standard basis $e_1,\ldots,e_n$, simple roots $\alpha_j=e_j-e_{j+1}$ for $1\leq j\leq n-1$ and $\alpha_n=e_n$. Then
$$
\Phi^+=\{e_i-e_j,e_i+e_j,e_k,2e_k\mid 1\leq i<j\leq n,\,1\leq k\leq n\}.
$$
Note that there are three root lengths in $\Phi$ (if $n>1$). The roots $2e_k$ are the \textit{long roots}. Note that $P=Q$ for the $\ssBC_n$ root system, and that the highest root is $\varphi=2e_1=2(\alpha_1+\cdots+\alpha_n)$. See Figure~\ref{fig:rootsystem}.
\end{remark}

To each root system $\Phi$ we associate reduced root systems $\Phi_0$ and $\Phi_1$ by
$$
\Phi_0=\{\alpha\in\Phi\mid \alpha/2\notin\Phi\}\quad\text{and}\quad 
\Phi_1=\{\alpha\in\Phi\mid 2\alpha\notin\Phi\}.
$$
In particular if $\Phi$ is reduced then $\Phi_0=\Phi_1=\Phi$, and if $\Phi$ is of type $\sBC_n$ then $\Phi_0$ (respectively $\Phi_1$) is a reduced root system of type~$\sB_n$ (respectively $\sC_n$).

\subsection{Weyl groups, affine Weyl groups, and extended affine Weyl groups}\label{sec:affineWeyl}

Let $\Phi$ be as above. For each $\alpha\in \Phi$ let $s_{\alpha}$ be the orthogonal reflection in the hyperplane $H_{\alpha}=\{x\in V\mid \langle x,\alpha\rangle=0\}$ orthogonal to $\alpha$, thus $s_{\alpha}(x)=x-\langle x,\alpha\rangle\alpha^{\vee}$ for $x\in V$. Write $s_i=s_{\alpha_i}$ for $i\in\Ifin$. The \textit{Weyl group} of $\Phi$ is the subgroup $W_0$ of $GL(V)$ generated by the reflections $s_1,\ldots,s_n$. The \textit{inversion set} of $w\in W_0$ is
$
\Phi(w)=\{\alpha\in\Phi_0^+\mid w^{-1}\alpha\in-\Phi_0^+\}
$
(note the convention that $\Phi(w)\subseteq\Phi_0$).

For each $\alpha\in\Phi$ and $k\in\nZ$ let $H_{\alpha,k}=\{x\in V\mid \langle x,\alpha\rangle=k\}$, and let $s_{\alpha,k}$ be the orthogonal reflection in the affine hyperplane $H_{\alpha,k}$. Explicitly, $s_{\alpha,k}(x)=x-(\langle x,\alpha\rangle-k)\alpha^{\vee}$. The \textit{affine Weyl group} is $\Waff=\langle\{s_{\alpha,k}\mid \alpha\in\Phi,\,k\in\mathbb{Z}\}\rangle$ (a subgroup of $\mathrm{Aff}(V)$). We have $\Waff=Q\rtimes W_0$, where we identify $\lambda\in V$ with the translation $t_{\lambda}(x)=x+\lambda$. The \textit{extended affine Weyl group} is $\Wext=P\rtimes W_0$. Since $Q\subseteq P$ we have $\Waff\leq\Wext$. If $w\in\Wext$ we define the \textit{linear part} $\theta(w)\in W_0$ and the \textit{translation coweight} $\wt(w)\in P$ by the equation
\begin{align}\label{eq:weightandfinaldirection}
w=t_{\wt(w)}\theta(w).
\end{align}

The Weyl group $W_0$ is a Coxeter group with Coxeter generators $\{s_i\mid i\in \Ifin\}$. The affine Weyl group $W$ is a Coxeter group with Coxeter generators $S=\{s_i\mid i\in \Iaff\}$, where $s_0=s_{\varphi,1}=t_{\varphi^{\vee}}s_{\varphi}.$
 If $\Phi=\ssBC_n$ then $\Waff=\Wext$ is a Coxeter group of type $\tilde{\ssC}_n$.

The extended affine Weyl group is typically not a Coxeter group. Writing $\Sigma=P/Q$ we have
$
\Wext\cong W\rtimes \Sigma.
$
Each $\sigma\in\Sigma$ induces a permutation (also denoted by $\sigma$) of $\Iaff$ by $\sigma s_i\sigma^{-1}=s_{\sigma(i)}$. In this way we can identify $\Sigma$ with a subgroup of the group of automorphisms of the extended Dynkin diagram. For example, in type $\ssA_n$ the group $\Sigma$ is generated by the permutation $i\mapsto i+1\mod n+1$ of $\Iaff$.

Let $\ell:W\to \nN$ denote the length function on the Coxeter group $W$. We extend this length function to the extended affine Weyl group $\Wext$ by setting $\ell(w\sigma)=\ell(w)$ for all $w\in W$ and $\sigma\in\Sigma$. Thus $\Sigma=\{w\in \Wext\mid \ell(w)=0\}$. By a \textit{reduced expression} for $w\in \Wext$ we shall mean a decomposition $w=s_{i_1}\cdots s_{i_{\ell}}\sigma$ with $\ell=\ell(w)$ and $\sigma\in\Sigma$.

The closures of the open connected components of $V\setminus\left(\bigcup_{\alpha,k}H_{\alpha,k}\right)$ are called \textit{alcoves}. Let $\mathbb{A}$ denote the set of all alcoves. The \textit{fundamental alcove} is given by
$$
A_0=\{x\in V\mid 0\leq\langle x,\alpha\rangle\leq 1\text{ for all $\alpha\in\Phi^+$}\}.
$$
The hyperplanes bounding $A_0$ are called the \textit{walls} of $A_0$. Explicitly these walls are $H_{\alpha_i,0}$ with $i\in\Ifin$ and $H_{\varphi,1}$. We say that a \textit{panel} of $A_0$ (that is, a codimension $1$ facet) has \textit{type} $i$ for $i\in\Ifin$ if it lies on the wall $H_{\alpha_i,0}$, and type $0$ if it lies on the wall $H_{\varphi,1}$.

Note that if $\Phi$ is not reduced, and if $\alpha\in\Phi^+$ with $2\alpha\in\Phi$, then $H_{\alpha,2k}=H_{2\alpha,k}$. Thus there are two distinct positive roots associated to this hyperplane. However note that every hyperplane $H$ can be expressed uniquely as $H=H_{\alpha,k}$ for some $\alpha\in\Phi_1^+$ and $k\in\ZZ$.

The (non-extended) affine Weyl group $\Waff$ acts simply transitively on $\mathbb{A}$. We use the action of $W$ to transfer the notions of walls, panels, and types of panels to arbitrary alcoves. Alcoves $A$ and $A'$ are called \textit{$i$-adjacent} (written $A\sim_{i}A'$) if $A\neq A'$ and $A$ and $A'$ share a common type~$i$ panel (with $i\in \Iaff$). Thus the alcoves $wA_0$ and $ws_iA_0$ are $i$-adjacent for all $w\in W$ and $i\in\Iaff$.

The extended affine Weyl group $\Wext$ acts transitively on $\mathbb{A}$, and the stabiliser of $A_0$ is $\Sigma$. The vertex set of $A_0$ is $\{x_i\mid i\in \Iaff\}$ where $x_0=0$ and $x_i=\omega_i/m_i$ for $i\in I$ (with $m_i$ as in Section~\ref{sec:root}), and the action of $\sigma\in\Sigma$ on this set of vertices is given by $\sigma(x_i)=x_{\sigma(i)}$.

Each affine hyperplane $H_{\alpha,k}$ with $\alpha\in\Phi^+$ and $k\in\nZ$ divides $V$ into two half-spaces, denoted
$$
H_{\alpha,k}^+=\{\alpha\in V\mid \langle x,\alpha\rangle\geq k\}\quad\text{and}\quad H_{\alpha,k}^-=\{x\in V\mid \langle x,\alpha\rangle\leq k\}.
$$
This ``orientation'' of the hyperplanes is called the \textit{periodic orientation} (see Figure~\ref{fig:rootsystem} for an illustration in the non-reduced $\sBC_2$ case).

\begin{figure}[H]
\begin{center}
\begin{tikzpicture}[scale=1]
\path [fill=gray!90] (0,0) -- (1,0) -- (1,1) -- (0,0);
\draw (-4,-4)--(4,-4);
\draw (-4,-3)--(4.5,-3);
\draw (-4,-2)--(4.5,-2);
\draw (-4,-1)--(4.5,-1);
\draw (-4,0)--(4.5,0);
\draw (-4,1)--(4.5,1);
\draw (-4,2)--(4.5,2);
\draw (-4,3)--(4.5,3);
\draw (-4,4)--(4,4);
\draw (-4,-4)--(-4,4);
\draw (-3,-4)--(-3,4.5);
\draw (-2,-4)--(-2,4.5);
\draw (-1,-4)--(-1,4.5);
\draw (0,-4)--(0,4.5);
\draw (1,-4)--(1,4.5);
\draw (2,-4)--(2,4.5);
\draw (3,-4)--(3,4.5);
\draw (4,-4)--(4,4);
\draw (-4,-4)--(4,4);
\draw (-4.6,-2.25)--(-4.25,-2.25);
\draw (-4.25,-2.25)--(2,4);
\draw (-4.6,-0.25)--(-4.25,-0.25);
\draw (-4.25,-0.25)--(0,4);
\draw (-4.6,1.75)--(-4.25,1.75);
\draw (-4.25,1.75)--(-2,4);
\draw (-2,-4)--(4,2);
\draw (0,-4)--(4,0);
\draw (2,-4)--(4,-2);
\draw (-4,-2)--(-1.75,-4.25);
\draw (-1.75,-4.25)--(-1.75,-4.6);
\draw (-4,0)--(0.25,-4.25);
\draw (0.25,-4.25)--(0.25,-4.6);
\draw (-4,2)--(2.25,-4.25);
\draw (2.25,-4.25)--(2.25,-4.6);
\draw (-4,4)--(4,-4);
\draw (-2,4)--(4,-2);
\draw (0,4)--(4,0);
\draw (2,4)--(4,2);
\node at (-2.75,4.3) {\small{$+$}};
\node at (-3.25,4.3) {\small{$-$}};
\node at (-1.75,4.3) {\small{$+$}};
\node at (-2.25,4.3) {\small{$-$}};
\node at (-0.75,4.3) {\small{$+$}};
\node at (-1.25,4.3) {\small{$-$}};
\node at (0.25,4.3) {\small{$+$}};
\node at (-0.25,4.3) {\small{$-$}};
\node at (1.25,4.3) {\small{$+$}};
\node at (0.75,4.3) {\small{$-$}};
\node at (2.25,4.3) {\small{$+$}};
\node at (1.75,4.3) {\small{$-$}};
\node at (3.25,4.3) {\small{$+$}};
\node at (2.75,4.3) {\small{$-$}};
\node at (4.3,3.25) {\small{$+$}};
\node at (4.3,2.75) {\small{$-$}};
\node at (4.3,2.25) {\small{$+$}};
\node at (4.3,1.75) {\small{$-$}};
\node at (4.3,1.25) {\small{$+$}};
\node at (4.3,0.75) {\small{$-$}};
\node at (4.3,0.25) {\small{$+$}};
\node at (4.3,-0.25) {\small{$-$}};
\node at (4.3,-0.75) {\small{$+$}};
\node at (4.3,-1.25) {\small{$-$}};
\node at (4.3,-1.75) {\small{$+$}};
\node at (4.3,-2.25) {\small{$-$}};
\node at (4.3,-2.75) {\small{$+$}};
\node at (4.3,-3.25) {\small{$-$}};
\node at (-4.4,-2) {\small{$-$}};
\node at (-4.4,-2.5) {\small{$+$}};
\node at (-4.4,0) {\small{$-$}};
\node at (-4.4,-0.5) {\small{$+$}};
\node at (-4.4,2) {\small{$-$}};
\node at (-4.4,1.5) {\small{$+$}};
\node at (-1.5,-4.4) {\small{$+$}};
\node at (-2,-4.4) {\small{$-$}};
\node at (0.5,-4.4) {\small{$+$}};
\node at (0,-4.4) {\small{$-$}};
\node at (2.5,-4.4) {\small{$+$}};
\node at (2,-4.4) {\small{$-$}};
\draw [latex-latex,line width=1pt] (-4,0)--(4,0);
\draw [latex-latex,line width=1pt] (0,-4)--(0,4);
\draw [latex-latex,line width=1pt] (-2,0)--(2,0);
\draw [latex-latex,line width=1pt] (0,-2)--(0,2);
\draw [latex-latex,line width=1pt] (-2,-2)--(2,2);
\draw [latex-latex,line width=1pt] (-2,2)--(2,-2);
\node at (2.6,-1.7) {\small{$\alpha_1=\alpha_1^{\vee}$}};
\node at (2.5,2) {\small{$\omega_2$}};
\node at (1,2.25) {\small{$(2\alpha_2)^{\vee}=\alpha_2^{\vee}/2$}};
\node at (0.4,3.8) {\small{$2\alpha_2$}};
\node at (2.25,0.25) {\small{$\omega_1$}};
\node at (3.75,0.25) {\small{$\varphi$}};
\node at (0.6,0.3) {\small{$e$}};
\node at (1.4,0.3) {\small{$s_0$}};
\node at (0.6,-0.3) {\small{$s_2$}};
\node at (0.3,0.65) {\small{$s_1$}};
\end{tikzpicture}
\caption{Root system of type $\ssBC_2$}
\label{fig:rootsystem}
\end{center}
\end{figure}
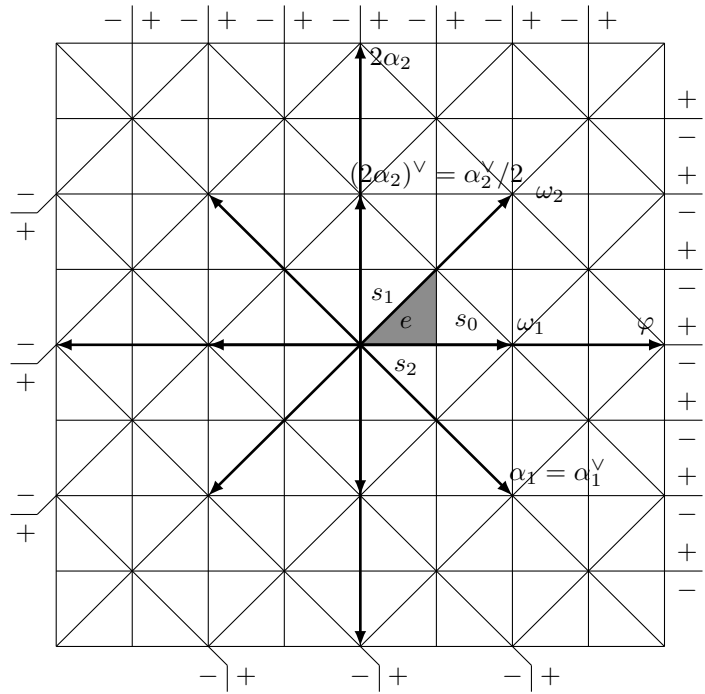

For $w\in\Wext$ and $i\in\{0\}\cup\Ifin$ we shall use the notation $$wA_0\,{^-}|^+\, ws_iA_0$$ to indicate that $ws_iA_0$ is on the positive side of the hyperplane separating $wA_0$ and $ws_iA_0$.  We will also say that $w\to ws_i$ is a \textit{positive crossing} (or simply is \textit{positive}), and similarly if $wA_0\,{^+}|^-\, ws_iA_0$ we say that $w\to ws_i$ is negative (see also Section~\ref{sec:alc}). 

\subsection{Affine root system}\label{sec:affrtsys}

It is convenient to have the notion of the \textit{affine root system} 
$$
\widetilde{\Phi}=\{\alpha+k\delta\mid \alpha\in\Phi,\,k\in\ZZ\}
$$
in the space $V\oplus\RR\delta$. Identifying $V$ with its dual, one may regard $\delta$ as the (nonlinear) constant function $\delta:V\to \RR$ with $\delta(v)=1$ for all $v\in V$. Writing $\langle \lambda,\alpha+k\delta\rangle=\langle \lambda,\alpha\rangle+k$ (however note this is no longer bilinear) we have 
$
H_{\alpha,k}=\{x\in V\mid \langle x,\alpha-k\delta\rangle=0\},
$
and so the hyperplane $H_{\alpha,k}$ corresponds to the affine roots $\pm(\alpha-k\delta)$. 

The simple roots of the affine root system are $\alpha_0=-\varphi+\delta$ and $\alpha_i+0\delta$ with $i\in\Ifin$. These choices give the set of positive affine roots as
$$
\widetilde{\Phi}^+=(\Phi^++\ZZ_{\geq 0}\delta)\cup(-\Phi^++\ZZ_{>0}\delta).
$$
Let $\widetilde{\Phi}_0=\Phi_0+\ZZ\delta$ (and so $\widetilde{\Phi}_0=\widetilde{\Phi}$ if $\Phi$ is reduced), and let $\widetilde{\Phi}_0^+=\widetilde{\Phi}^+\cap\widetilde{\Phi}_0$. 

The action of $\Wext$ on the affine root system (given by the action on half spaces) is given by
\begin{align*}
w(\alpha+k\delta)=w\alpha+k\delta\quad\text{and}\quad t_{\lambda}(\alpha+k\delta)=\alpha+(k-\langle\lambda,\alpha\rangle)\delta \quad \text{for $w\in W_0$ and $\lambda\in P$}.
\end{align*}
The (affine) inversion set of $w\in\Wext$ is 
$$
\widetilde{\Phi}(w)=\{\alpha+k\delta\in\widetilde{\Phi}_0^+\mid w^{-1}(\alpha+k\delta)\in-\widetilde{\Phi}_0^+\}.
$$

In the notation of the previous section, for $w\in\Wext$ and $i\in\{0\}\cup\Ifin$ we have
\begin{align}\label{eq:affineroots}
wA_0\,{^-}|^+\, ws_iA_0\quad\text{if and only if}\quad w\alpha_i\in -\Phi^++\ZZ\delta. 
\end{align}

\subsection{Parabolic subgroups}\label{sec:parabolics}

Let $J\subseteq I$. The \textit{$J$-parabolic subgroup} of $W_0$ is the subgroup $W_J=\langle \{s_j\mid j\in J\}\rangle$. Since $W_J$ is finite there exists a unique longest element of $W_J$, denoted~$\sw_J$, and we have $\ell(\sw_Jw)=\ell(w\sw_J)=\ell(\sw_J)-\ell(w)$ for all $w\in W_J$. We write $\sw_0=\sw_I$. It is well known (see, for example \cite[Proposition~2.20]{AB:08}) that each coset $W_Jw$ with $w\in W_0$ contains a unique representative of minimal length. Let $W^J$ be the transversal of these minimal length coset representatives. Then each $w\in W_0$ has a unique decomposition
\begin{align}
\label{eq:WJdecomposition}w=yu\quad\text{with $y\in W_J$, $u\in W^J$},
\end{align}
and moreover whenever $y\in W_J$ and $u\in W^J$ we have $\ell(yu)=\ell(y)+\ell(u)$.

The \textit{support} of a root $\alpha\in\Phi$ is $\mathrm{supp}(\alpha)=\{i\in I\mid c_i\neq 0\}$, where $\alpha=\sum_{i\in I}c_i\alpha_i$.  For $J\subseteq I$ let 
$
\Phi_J=\{\alpha\in\Phi\mid\mathrm{supp}(\alpha)\subseteq J\},
$
and for $w\in W_0$ write $\Phi_J(w)=\Phi(w)\cap\Phi_J$. The following lemma is well known (see, for example \cite[Corollary~2.13]{HNW:16}), however we provide a proof for completeness. 

\begin{lemma}\label{lem:decomposition}
Let $J\subseteq I$. If $w=yu$ with $y\in W_J$ and $u\in W^J$ then $\Phi(y)=\Phi(w)\cap \Phi_J$. In particular, we have
$
W^J=\{u\in W\mid \Phi_J(u)=\emptyset\}.
$
\end{lemma}

\begin{proof}
Suppose there exists $\beta\in\Phi_J(w)\backslash\Phi(y)$. Since $\beta\notin\Phi(y)$ we have $\ell(s_{\beta}y)>\ell(y)$, and since $\beta\in\Phi_J(w)$ we have $\ell(s_{\beta}w)<\ell(w)$. Since $\beta\in\Phi_J$ we have $s_{\beta}\in W_J$, but then the element $y'=s_{\beta}y\in W_J$ satisfies
$
\ell(y'u)=\ell(s_{\beta}w)<\ell(w)=\ell(y)+\ell(u)<\ell(s_{\beta}y)+\ell(u)=\ell(y')+\ell(u),
$
contradicting the fact that $u\in W^J$. 
\end{proof}

\begin{defn}\label{defn:theta}
Let $w\in \Wext$ and recall the definition of $\theta(w)$ from~(\ref{eq:weightandfinaldirection}). Define $\theta_J(w)\in W_J$ and $\theta^J(w)\in W^J$ by the equation 
$
\theta(w)=\theta_J(w)\theta^J(w).
$
\end{defn}

For $J\subseteq\Ifin$ let
$$
V_J=\sum_{j\in J}\mathbb{R}\alpha_j^{\vee}\quad\text{and}\quad V^J=\sum_{i\in I\backslash J}\mathbb{R}\omega_i.
$$
Then $V=V_J\oplus V^J$ (orthogonal direct sum). 

Let $\lambda\mapsto\lambda^J$ denote the orthogonal projection $V\to V^J$, and let
$$
P^J=\{\lambda^J\mid \lambda\in P\}\subseteq V^J.
$$
Note that in general the $\ZZ$-lattice $P^J$ is not a subset of $P$. For example if $\Phi$ is of type $\sA_2$ and $J=\{1\}$ then $P^J\ni \omega_1^J=\frac{1}{2}\omega_2\notin P$ (see Figure~\ref{fig:A2}, where $V^J$ is denoted as a solid line, and $P^J$ is indicated by heavy dots). Let 
$\{\overline{\omega}_i\mid i\in I\backslash J\}$ be a choice of $\ZZ$-basis of $P^J$ (in the example of Figure~\ref{fig:A2} we may take $\overline{\omega}_2=\omega_2/2$).

\begin{figure}[H]
\centering
\begin{tikzpicture}[scale=1.5]
\clip ({-5*0.5},{-2.5*0.866}) rectangle + ({5},{5*0.866});
    \path [fill=gray!90] (0,0) -- (-0.5,0.866) -- (0.5,0.866) -- (0,0);
    \draw (2.5, {-3*0.866})--( 3.5, {-1*0.866} );
    \draw (1.5, {-3*0.866})--( 3.5, {1*0.866} );
    \draw  (0.5, {-3*0.866})--( 3.5, {3*0.866} );
    \draw  (-0.5, {-3*0.866})--( 3, {4*0.866} );
    \draw  [line width=2pt](-1.5, {-3*0.866})--( 2, {4*0.866} );
  \draw  (-2.5, {-3*0.866})--( 1, {4*0.866} );
    \draw  (-3.5, {-3*0.866})--(0, {4*0.866} );
    \draw  (-3.5, {-1*0.866})--(-1, {4*0.866} );
    \draw  (-3.5, {1*0.866})--(-2, {4*0.866} );
    \draw  (-3.5, {3*0.866})--(-3, {4*0.866} );
   \draw (-2.5, {-3*0.866})--( -3.5, {-1*0.866} );
    \draw (-1.5, {-3*0.866})--( -3.5, {1*0.866} );
    \draw  (-0.5, {-3*0.866})--( -3.5, {3*0.866} );
    \draw  (0.5, {-3*0.866})--( -3, {4*0.866} );
  \draw (1.5, {-3*0.866})--( -2, {4*0.866} );
    \draw  (2.5, {-3*0.866})--( -1, {4*0.866} );
    \draw  (3.5, {-3*0.866})--(0, {4*0.866} );
    \draw  (3.5, {-1*0.866})--(1, {4*0.866} );
    \draw  (3.5, {1*0.866})--(2, {4*0.866} );
    \draw  (3.5, {3*0.866})--(3, {4*0.866} );
    \draw (-3.5, -2.598)--( 3.5, -2.598);
    \draw (-3.5, -1.732)--( 3.5, -1.732);
    \draw (-3.5, -0.866)--( 3.5, -0.866);
    \draw (-3.5, 0)--( 3.5, 0);
    \draw (-3.5, 3.464)--( 3.5, 3.464 );
    \draw (-3.5, 2.598)--( 3.5, 2.598);
    \draw (-3.5, 1.732)--( 3.5, 1.732);
        \draw (-3.5, 0.866)--( 3.5, 0.866);
    \draw [latex-latex, line width=1pt] (-1.5,-0.866)--(1.5,0.866);
    \draw [latex-latex, line width=1pt] (1.5,-0.866)--(-1.5,0.866);
    \draw [latex-latex, line width=1pt] (0,-1.732)--(0,1.732);
    \node at (-2,1.1) {\small{$\alpha_1^{\vee}$}};
    \node at (2,1.1) {\small{$\alpha_2^{\vee}$}};
    \node at (-0.8,1.05) {\small{$\omega_1$}};
    \node at (0.9,1.05) {\small{$\omega_2$}};
    \node at ({0*0.5},{0*0.866}) {\Large{$\bullet$}};
     \node at ({0.5*0.5},{0.5*0.866}) {\Large{$\bullet$}};
     \node at ({1*0.5},{1*0.866}) {\Large{$\bullet$}};
      \node at ({2*0.5},{2*0.866}) {\Large{$\bullet$}};
       \node at ({1.5*0.5},{1.5*0.866}) {\Large{$\bullet$}};
       \node at ({2.5*0.5},{2.5*0.866}) {\Large{$\bullet$}};
        \node at ({-0.5*0.5},{-0.5*0.866}) {\Large{$\bullet$}};
     \node at ({-1*0.5},{-1*0.866}) {\Large{$\bullet$}};
      \node at ({-2*0.5},{-2*0.866}) {\Large{$\bullet$}};
       \node at ({-1.5*0.5},{-1.5*0.866}) {\Large{$\bullet$}};
       \node at ({-2.5*0.5},{-2.5*0.866}) {\Large{$\bullet$}};
       \draw [line width=1pt,dotted] (-0.5,0.866)--(0.25,0.433);
\end{tikzpicture}
\caption{The set $P^J$ for type $\sA_2$ with $J=\{1\}$}\label{fig:A2}
\end{figure}
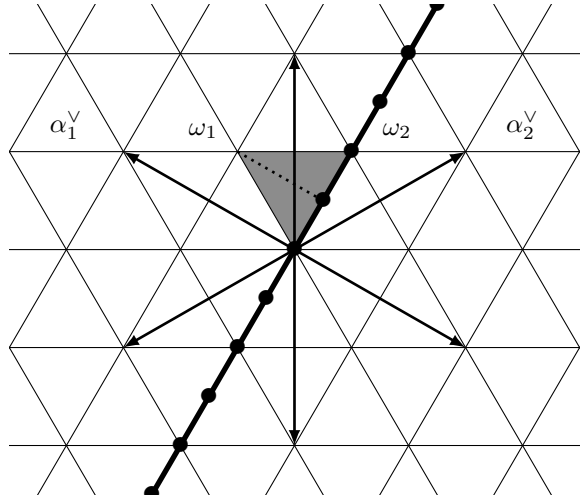

For $J\subseteq\Ifin$ let $\Phi_{J,l}$ and $\Phi_{J,s}$ be the long and short roots of $\Phi_J\cap \Phi_0$, respectively (with $\Phi_{J,s}=\emptyset$ if $\Phi$ is simply laced). Define
\begin{align*}
 \rho_J=\frac{1}{2}\sum_{\alpha\in\Phi_{J,l}^+}\alpha\quad\text{and}\quad\rho_J'=\frac{1}{2}\sum_{\alpha\in\Phi_{J,s}^+}\alpha,
\end{align*}
and let $\{\tilde{\omega}_j\mid j\in J\}$ be the basis of $V_J$ dual to the basis $\{\alpha_j^{\vee}\mid j\in J\}$ (that is, $\langle\alpha_j^{\vee},\tilde{\omega}_i\rangle=\delta_{i,j}$ for all $i,j\in J$). 

\begin{lemma}\label{lem:orthogonal}
We have
$$
\rho_J=\sum_{\{j\in J\,\mid\, \alpha_j\in\Phi_{J,l}\}}\tilde{\omega}_j\quad\text{and}\quad \rho_J'=\sum_{\{j\in J\,\mid\, \alpha_j\in\Phi_{J,s}\}}\tilde{\omega}_j.
$$
In particular $\rho_J$ (respectively $\rho_J'$) is orthogonal to all short (respectively long) simple roots of~$\Phi_J$.
\end{lemma}

\begin{proof}
In the simply laced case see \cite[VI, \S1, Proposition 29]{Bou:02}. In the non-simply laced case the claim is readily checked from the classification (it is sufficient to check for irreducible~$J$). 
\end{proof}

\subsection{Positively folded alcove paths}\label{sec:alc}

Since the extended affine Weyl group~$\Wext$ does not act freely on the set of alcoves, it is convenient for our purposes to consider ``alcove paths'' as sequences of elements of $\Wext$ rather than sequences of alcoves. Thus we make the following definition (see~\cite{Ram:06}).

\begin{defn}\label{defn:posfolded}
Let $\vec w=s_{i_1}s_{i_2}\cdots s_{i_{\ell}}\sigma$ be an expression for $w\in \Wext$ (not necessarily reduced) with $\sigma\in \Sigma$, and let~$v\in \Wext$. A \textit{folded alcove path of type~$\vec w$ starting at $v$, and ending at $\mathrm{end}(p)=v_{\ell}\sigma$} is a sequence $$p=(v_0,v_1,\ldots,v_{\ell},v_{\ell}\sigma)$$ with $v_0,\ldots,v_{\ell}\in \Wext$ such that $v_0=v$ and $v_k\in\{v_{k-1},v_{k-1}s_{i_k}\}$ for $1\leq k\leq \ell$. A \textit{positively folded alcove path} is a folded alcove path $p=(v_0,v_1,\ldots,v_{\ell},v_{\ell}\sigma)$ such that:
\begin{center}
if $v_{k-1}=v_k$ then $v_{k-1}A_0\, {^+}|^-\,v_{k-1}s_{i_k}A_0$. 
\end{center}
\end{defn}

For $v\in \Wext$ let
$$
\mathcal{P}(\vec{w},v)=\{\text{all positively folded alcove paths of type $\vec{w}$ starting at $v$}\}.
$$ 

Let $\vec w=s_{i_1}s_{i_2}\cdots s_{i_{\ell}}\sigma$ and let $p=(v_0,\ldots,v_{\ell},v_{\ell}\sigma)\in\mathcal{P}(\vec{w},v)$. The index~$k\in\{1,2,\ldots,\ell\}$ is called:
\begin{compactenum}[$(1)$]
\item a \textit{positive (respectively, negative) $i_k$-crossing} if $v_k=v_{k-1}s_{i_k}$ and $v_{k}A_0$ is on the positive (respectively, negative) side of the hyperplane separating the alcoves $v_{k-1}A_0$ and $v_kA_0$;
\item an \textit{$i_k$-fold} if $v_k=v_{k-1}$ (in which case $v_{k-1}A_0$ is necessarily on the positive side of the hyperplane separating $v_{k-1}A_0$ and $v_{k-1}s_{i_k}A_0$).
\end{compactenum}
If $p$ has no folds we say that $p$ is \textit{straight}. Less formally, the above crossings and folds can be visualised as follows (where $x=v_{k-1}$):

\begin{center}
\begin{tikzpicture}[xscale=0.7, yscale=0.6]
\draw (-8.5,-1)--(-8.5,1);
\draw[line width=0.5pt,-latex](-9,0)--(-8,0);
\node at (-9.5,1){{ $-$}};
\node at (-9.5,.2){{$x$}};
\node at (-7.5,.2){{$xs_{i_k}$}};
\node at (-7.9,1){{ $+$}};
\node at (-8.5,-1.5){{ \text{(positive $i_k$-crossing)}}};
\draw(-2,-1)--(-2,1);
\node at (-3,1){{ $-$}};
\node at (-3,.2){{ $xs_{i_k}$}};
\node at (-1,.2){{ $x$}};
\node at (-1.4,1){{ $+$}};
\draw[line width=0.5pt,-latex] plot[smooth] coordinates {(-1.35,0)(-1.85,0)(-1.85,-.15)(-1.35,-.15)};
\node at (-2,-1.5){{ \text{($i_k$-fold)}}};
%
\draw(4.5,-1)--(4.5,1);
\draw[line width=0.5pt,-latex](5,0)--(4,0);
\node at (5,1){{ $+$}};
\node at (5.5,.2){{$x$}};
\node at (3.25,.2){{ $xs_{i_k}$}};
\node at (3.6,1){{ $-$}};
\node at (4.5,-1.5){{ \text{(negative $i_k$-crossing)}}};
\end{tikzpicture}
\end{center}

 If $p$ is a positively folded alcove path, then for each $i\in \Iaff$ we define
\begin{align*}
f_i(p)&=\#\textrm{($i$-folds in $p$)}\quad\text{and}\quad f(p)=\#\textrm{(folds in $p$)}=\sum_{i=0}^nf_i(p).
\end{align*}
The \textit{coweight} and \textit{final direction} of a positively folded alcove path $p$ are (c.f. (\ref{eq:weightandfinaldirection}))
\begin{align}\label{eq:weighttheta}
\wt(p)=\wt(\mathrm{end}(p))\quad\text{and}\quad \theta(p)=\theta(\mathrm{end}(p)),
\end{align}
and we write $\theta_J(p)=\theta_J(\mathrm{end}(p))$ and $\theta^J(p)=\theta^J(\mathrm{end}(p))$.

\subsection{Affine Hecke algebras and the Bernstein-Lusztig presentation}\label{sec:hecke}

Let $(\sq_i)_{i\in \Iaff}$ be a family of commuting invertible indeterminates with the property that $\sq_i=\sq_{j}$ whenever $s_i$ and $s_j$ are conjugate in~$\Wext$. Let $\sR=\mathbb{Z}[(\sq_i^{\pm 1})_{i\in \Iaff}]$. The \textit{extended affine Hecke algebra} is the $\sR$-algebra $\Hext$ with basis $\{T_w\mid w\in \Wext\}$ and multiplication given by (for $w,v\in \Wext$ and $i\in\Iaff$)
\begin{align}\label{eq:relations}
\begin{aligned}
T_wT_v&=T_{wv}&&\text{if $\ell(wv)=\ell(w)+\ell(v)$}\\
T_wT_{s_i}&=T_{ws_i}+(\sq_i-\sq_i^{-1})T_w&&\text{if $\ell(ws_i)=\ell(w)-1$}.
\end{aligned}
\end{align}
Note that each $T_{s_i}$ is invertible with $T_{s_i}^{-1}=T_{s_i}-(\sq_i-\sq_i^{-1})$, and that each $T_{\sigma}$ (with $\sigma\in\Sigma$) is invertible with $T_{\sigma}^{-1}=T_{\sigma^{-1}}$. It follows that each $T_w$ with $w\in\Wext$ is invertible. The (non-extended) affine Hecke algebra is the subalgebra $\Haff$ spanned by $\{T_w\mid w\in \Waff\}$.

We often write $T_i$ in place of $T_{s_i}$. For $w\in\Wext$ we write
$
\sq_w=\sq_{i_1}\cdots \sq_{i_k}
$
whenever $w=s_{i_1}\cdots s_{i_k}\sigma$ is a reduced expression of $w$ (this can easily be seen to be independent of the choice of reduced expression using Tits' solution to the Word Problem). In particular, note that $\sq_{\sigma}=1$ for all $\sigma\in\Sigma$.

Let $w\in \Wext$ and choose any expression $w=s_{i_1}\cdots s_{i_{\ell}}\sigma$ (not necessarily reduced). Let $v_0=e$ and $v_k=s_{i_1}\cdots s_{i_k}$ for $1\leq k\leq \ell$ (thus $(v_0,v_1,\ldots,v_{\ell})$ is the straight alcove path of type $s_{i_1}\cdots s_{i_{\ell}}$ starting at $e$). Let $A_k=v_kA_0$, and let $\epsilon_1,\ldots,\epsilon_{\ell}$ be the sequence of signs of the crossings, defined by (see (\ref{eq:affineroots})):
\begin{align*}
\epsilon_k=\begin{cases}
+1&\text{if $A_{k-1}\,{^-}|^+\,A_k$ (that is, $v_{k-1}\alpha_{i_k}\in-\Phi^++\ZZ\delta$)}\\
-1&\text{if $A_{k-1}\,{^+}|^-\,A_k$ (that is, $v_{k-1}\alpha_{i_k}\in\Phi^++\ZZ\delta$)}.
\end{cases}
\end{align*}
Then the element 
$$
X_w=T_{i_1}^{\epsilon_1}\cdots T_{i_k}^{\epsilon_k}T_{\sigma}\in\Hext
$$
does not depend on the particular expression $w=s_{i_1}\cdots s_{i_k}\sigma$ chosen (see \cite{Goe:07}). 

From the defining relations~(\ref{eq:relations}) it follows that $X_w-T_w$ is a linear combination of terms $T_v$ with $v<w$ (in extended Bruhat order), and hence $\{X_w\mid w\in\Wext\}$ is a basis of $\Hext$.

If $\lambda\in P$ we write 
$$
X^{\lambda}=X_{t_{\lambda}}.
$$
Then $X^{\lambda}X^{\mu}=X^{\lambda+\mu}=X^{\mu}X^{\lambda}$ for all $\lambda,\mu\in P$, and for $w\in\Wext$ we have
\begin{align}\label{eq:splitting}
X_w=X_{t_{\lambda}u}=X^{\lambda}X_{u}=X^{\lambda}T_{u^{-1}}^{-1}
\end{align}
where $\lambda=\wt(w)$ and $u=\theta(w)$ (the second equality follows since $t_{\lambda}$ is on the positive side of every hyperplane through $\lambda$, and the third equality follows since $X_u=T_{u^{-1}}^{-1}$ for all $u\in W_0$). Thus the set $\{X^{\lambda}T_{u^{-1}}^{-1}\mid \lambda\in P,\,u\in W_0\}$ is a basis of~$\Hext$ (called the \textit{Bernstein-Lusztig basis}).

Since $s_0=t_{\varphi^{\vee}}s_{\varphi}$ equation~(\ref{eq:splitting}) gives
\begin{align}\label{eq:T0}
T_0=X^{\varphi^{\vee}}T_{s_{\varphi}}^{-1}.
\end{align}

The combinatorics of positively folded alcove paths encode the change of basis from the standard basis $(T_w)_{w\in \Wext}$ of $\Hext$ to the Bernstein-Lusztig basis $(X_w)_{w\in\Wext}$. This is seen by taking $u=e$ in the following proposition.

\begin{prop}{\cite[Theorem~3.3]{Ram:06}}\label{prop:basischange}
Let $u,w\in\Wext$ and let $\vec{w}$ be any reduced expression for~$w$. Then 
$$
X_uT_w=\sum_{p\in \cP(\vec w,u)}\cQ(p)X_{\mathrm{end}(p)}\quad\text{where}\quad \cQ(p)=\prod_{i\in \Iaff}(\sq_i-\sq_i^{-1})^{f_i(p)}.
$$
\end{prop}

\begin{convention}\label{conv:parameters}
It is convenient to make the following convention, and we will do so henceforth: If $\Phi=\sA_1$ then $\sq_0=\sq_1$, and if $\Phi=\sC_n$ ($n\geq 2$) then $\sq_0=\sq_n$ (note that there is no loss of generality as the case $\sq_0\neq \sq_n$ is covered by the non-reduced system~$\sBC_n$ for $n\geq 1$). Thus if $\sigma\in \Sigma$ we have $\sq_{\sigma(i)}=\sq_i$ for all $i\in\{0\}\cup\Ifin$ (recall that $\Sigma$ is trivial if $\Phi=\sBC_n$). 
\end{convention}

With Convention~\ref{conv:parameters} in force, a complete set of relations of $\Hext$ in the Bernstein-Lusztig basis is given by (for $i,j\in \Ifin$ with $i\neq j$ and $\lambda,\mu\in P$):
\begin{align*}
T_i^2&=I+(\sq_i-\sq_i^{-1})T_i\\
T_iT_{j}T_i\cdots&=T_{j}T_iT_{j}\cdots\quad\text{($m_{ij}$ terms on each side)}\\
X^{\lambda}X^{\mu}&=X^{\lambda+\mu}\\
T_iX^{\lambda}&=\begin{cases}X^{s_i\lambda}T_i+(\sq_i-\sq_i^{-1})\frac{X^{\lambda}-X^{s_i\lambda}}{1-X^{-\alpha_i^{\vee}}}&\text{if $(\Phi,i)\neq(\sBC_n,n)$}\\
X^{s_n\lambda}T_n+\big[\sq_n-\sq_n^{-1}+(\sq_0-\sq_0^{-1})X^{-\alpha_n^{\vee}/2}\big]\frac{X^{\lambda}-X^{s_n\lambda}}{1-X^{-\alpha_n^{\vee}}}&\text{if $(\Phi,i)=(\ssBC_n,n)$.}
\end{cases}
\end{align*}
The final relation is known as the Bernstein-Lusztig relation (see \cite[Proposition~3.6]{Lus:89}). Note that since $s_i\lambda=\lambda-\langle\lambda,\alpha_i\rangle\alpha_i^{\vee}$ and $\langle\lambda,\alpha_i\rangle\in\mathbb{Z}$ the right hand side of the Bernstein-Lusztig relation is in~$\Hext$.

\section{The fundamental $J$-alcove $\cA_J$}\label{sec:Jalcoves}

In this section we extend the definition of ``strips'' in \cite{GP:19,GP:19b} to arbitrary affine type, and study the geometry of these subsets of $V$ (following Lusztig \cite[\S2]{Lus:97} we call these generalisations \textit{$J$-alcoves} for reasons that will be explained below). It turns out that the  geometry of the fundamental $J$-alcove $\cA_J$ will be crucial to our combinatorial formulae, and so we develop some of the basic properties of $\cA_J$ and its symmetries here.

\subsection{$J$-alcoves and $J$-affine Weyl groups}\label{sec:JAffineWeyl}

\begin{defn}
Let $J\subseteq \Ifin$. The \textit{fundamental $J$-alcove} is the set 
$$
\cA_J=\{x\in V\mid 0\leq \langle x,\alpha\rangle \leq 1\text{ for all }\alpha\in\Phi_J^+\}.
$$
\end{defn}

See Example~\ref{ex:Other} for an illustration of a fundamental $J$-alcove. Let us explain this terminology. To begin with, in the case $J=I$ we have $\cA_I=A_0$, the fundamental alcove. For general $J\subseteq I$ let $\mathbb{H}_J$ be the set of all hyperplanes $H_{\alpha,k}$ with $\alpha\in\Phi_J^+$ and $k\in\ZZ$. Let $\mathbb{A}_J$ denote the set of ``alcoves'' of this hyperplane arrangement: by definition these are the closures of the open connected components of $V\backslash \left(\bigcup_{H\in\mathbb{H}_J}H\right)$. We shall call elements of $\mathbb{A}_J$ \textit{$J$-alcoves}, and it is clear that $\cA_J$ is indeed a $J$-alcove.

Let 
$$
Q_J=\sum_{\alpha\in\Phi_J}\mathbb{Z}\alpha^{\vee}.
$$
Note that if $\Phi_J$ is not reduced (hence $\Phi$ is of type $\sBC_n$ and $n\in J$) then $\alpha_n^{\vee}/2=(2\alpha_n)^{\vee}\in Q_J$. The \textit{$J$-affine Weyl group} is 
$$
\WJaff=\langle s_{\alpha,k}\mid \alpha\in\Phi_J^+,\,k\in\mathbb{Z}\rangle=Q_J\rtimes W_J.
$$
By the general theory of Section~\ref{sec:affineWeyl} (c.f. \cite[\S V.3]{Bou:02}) the group $\WJaff$ acts simply transitively on the set $\mathbb{A}_J$, and $\cA_J$ is a fundamental domain for the action of $\WJaff$ on~$V$, see \cite[\S V.3]{Bou:02}.

\begin{remark}
The notion of $J$-alcoves can be traced back to Lusztig~\cite[\S2]{Lus:97}. There is a connection between $J$-alcoves and the \textit{$J$-sectors} of Mili\'cevi\'c, Naqvi, Schwer, and Thomas \cite[\S3.1]{MNST:22}, which in turn are special cases of the  \textit{chimneys} of Rousseau \cite[\S3.1]{Rou:01}. Indeed, the $J$-sectors of \cite{MNST:22} are precisely the intersection of the fundamental $J$-alcove with a set of the form
$$
\bigcap_{\beta\in\Phi^+\backslash \Phi_J}H_{\beta,k_{\beta}}^-
$$
for some choice of integers $k_{\beta}$, $\beta\in\Phi^+\backslash \Phi_J$. See \cite[Remark~3.3]{MNST:22} for a discussion of the history of these concepts. 
\end{remark}

Let $\cK(J)$ denote the set of connected components of $J$ (that is, the connected components of the Coxeter graph of the Coxeter system $(W_J,J)$). For example, in type $\tilde{\sA}_6$ if $J=\{1,3,4,6\}$ then $\cK(J)=\{\{1\},\{3,4\},\{6\}\}$. For connected subsets $K\subseteq \Ifin$ let $\varphi_K$ be the highest root of $\Phi_K$. Note that 
\begin{align}\label{eq:disjoint}
\Phi_J=\bigsqcup_{K\in\cK(J)}\Phi_K.
\end{align}
Note that if $\Phi_J$ is not reduced, then $\Phi$ is of type $\sBC_n$, and $J$ contains a connected component $K=\{k+1,k+2,\ldots,n\}$ (of type $\sBC_{n-k}$).

\begin{lemma}\label{lem:boundingwalls}
Let $J\subseteq \Ifin$. The walls of the fundamental $J$-alcove $\cA_J$ are $H_{\alpha_j,0}$ and $H_{\varphi_K,1}$ with $j\in J$ and $K\in\cK(J)$. That is,
$$
\cA_J=\{x\in V\mid \langle x,\alpha_j\rangle\geq 0\text{ for all $j\in J$ and }\langle x,\varphi_K\rangle\leq 1\text{ for all $K\in\cK(J)$}\}.
$$
\end{lemma}

\begin{proof}
Write $\cA_J'$ for the right hand side of the displayed equation in the statement of the lemma. Clearly $\cA_J\subseteq\cA_J'$, and so it suffices to show that $\cA_J'\subseteq \cA_J$. Thus suppose that $x\in \cA_J'$. Let $\alpha\in \Phi_J$. Writing $\alpha=\sum_{j\in J}a_j\alpha_j$ we have
$
\langle x,\alpha\rangle=\sum_{j\in J}a_j\langle x,\alpha_j\rangle\geq 0.
$
Since $\alpha\in\Phi_J$ we have $\alpha\in\Phi_K$ for some $K\in\cK(J)$, by~(\ref{eq:disjoint}). Since $\langle x,\alpha_j\rangle\geq 0$ for all $j\in J$ and $\varphi_K-\alpha$ is a nonnegative linear combination of roots $\alpha_k$ with $k\in K\subseteq J$ we have
$
\langle x,\varphi_K-\alpha\rangle\geq 0,
$
and so $\langle x,\alpha\rangle\leq \langle x,\varphi_K\rangle\leq 1$, hence $x\in \cA_J$. 
\end{proof}

The reflections in the walls of $\cA_J$ generate the $J$-affine Coxeter group, and we have
\begin{align}\label{eq:Jaffine}
\WJaff=\prod_{K\in\cK(J)}\WKaff.
\end{align}
Let $s_j'=s_j$ for $j\in J$, and let $s_{0_K}'=s_{\varphi_K,1}$ for $K\in\cK(J)$ (where for each $K\in\cK(J)$ we introduce a symbol $0_K$). Then $\{s_{0_K}'\}\cup\{s_k'\mid k\in K\}$ are the Coxeter generators of $\WKaff$. Write $\Jaff=\{0_K\mid K\in\cK(J)\}\cup J$, and so $\{s_j'\mid j\in \Jaff\}$ is the set of Coxeter generators of $\WJaff$.

\subsection{The set $P^{(J)}=\cA_J\cap P$}

In this section we determine the set $P^{(J)}=\cA_J\cap P$ of coweights contained in the fundamental $J$-alcove. 

\begin{lemma}\label{lem:projection}
Let $J\subseteq \Ifin$ and let $\lambda\in P$. There exists a unique $\lambda^*\in (\lambda+Q_J)\cap \cA_J$. 
\end{lemma}

\begin{proof}
Let $\lambda\in P$. Then $\lambda\in \cA$ for some $J$-alcove $\cA\in \mathbb{A}_J$. Since $\WJaff$ acts transitively on $\mathbb{A}_J$, and since $\cA_J\in \mathbb{A}_J$, there is $w\in \WJaff$ such that $w\cA=\cA_J$ (and in particular, $w\lambda\in\cA_J$). Since $w$ is a product of reflections in hyperplanes $H_{\alpha,k}$ with $\alpha\in\Phi_J^+$ and $k\in\mathbb{Z}$, and since $s_{\alpha,k}(\lambda)=\lambda-(\langle \lambda,\alpha\rangle-k)\alpha^{\vee}$ it follows that $w\lambda\in \lambda+Q_J$, proving the existence of~$\lambda^*$. Uniqueness follows from the fact that $\cA_J$ is a fundamental domain for the action of $\WJaff$ on~$V$.
\end{proof}

\begin{defn}\label{defn:projections}
Let $\lambda^{(J)}$ denote the unique element $\la^*$ of $(\lambda+Q_J)\cap \cA_J$ (c.f. Lemma~\ref{lem:projection}). Note that $\lambda^{(J)}$ and $\lambda^J$ are, in general, distinct. See Example~\ref{ex:Other}.
\end{defn}

The following lemma gives a useful characterisation of $P^{(J)}$. For $\lambda\in P$ let 
\begin{align}\label{eq:Jlambda}
J_{\lambda}=\{j\in J\mid \langle \lambda,\alpha_j\rangle\neq 0\}.
\end{align}

\begin{lemma}\label{lem:restrictingstrip}
Let $J\subseteq\Ifin$ and $\lambda\in P$. Then $\lambda\in\cA_J$ if and only if the set
$
J_{\lambda}
$
has the following properties:
\begin{compactenum}[$(1)$]
\item if $j\in J_{\lambda}$ then $\langle\lambda,\alpha_j\rangle=1$;
\item for each $K\in\cK(J)$ we have $|J_{\lambda}\cap K|\leq 1$;
\item if $j\in J_{\lambda}$ then $\langle\omega_j,\varphi_K\rangle=1$ where $K\in\cK(J)$ is the connected component of $J$ containing~$j$.
\end{compactenum}
\end{lemma}

\begin{proof}
Write $\lambda=\sum_{i\in I}c_i\omega_i$. If $\lambda\in P\cap \cA_J$ then $c_j\in\{0,1\}$ for all $j\in J$. Moreover, for each $K\in\cK(J)$ there is at most one $j\in K$ with $c_j=1$, for otherwise $\langle\lambda,\varphi_K\rangle\geq 2$. Moreover, if $j\in K$ with $c_j=1$ then for $\alpha\in\Phi_K^+$ we have $\langle \lambda,\alpha\rangle=\langle\omega_j,\alpha\rangle$. In particular, taking $\alpha=\varphi_K\in\Phi_J$ we have $\langle\omega_j,\varphi_K\rangle=1$. The converse is clear, using $\langle \omega_j,\alpha\rangle\leq\langle\omega_j,\varphi_K\rangle$ for all $\alpha\in\Phi_K^+$. 
\end{proof}

\begin{cor}\label{cor:lambdadecomp} The set $P^{(J)}$ consists precisely of the elements 
\begin{align*}
\lambda=\sum_{i\in I\backslash J}a_i\omega_j+\sum_{j\in J'}\omega_j
\end{align*}
with $a_i\in \ZZ$ and where $J'\subseteq J$ is a set satisfying $|J'\cap K|\leq 1$ for all $K\in\cK(J)$ and if $j\in J'\cap K$ then the coefficient of $\alpha_j$ in $\varphi_K$ is~$1$ (then $J'=J_{\lambda}$). 
\end{cor}

\begin{proof}
This follows immediately from Lemma~\ref{lem:restrictingstrip}.  
\end{proof}

\subsection{The set $\WJ$ and the elements $\sy_{\lambda}$ and $\st_{\lambda}$}

Define a subset $\WJ\subseteq\Wext$ by
\begin{align*}
\WJ=\{w\in \Wext\mid wA_0\subseteq \cA_J\}.
\end{align*}
Thus $\WJ\cap \Waff$ is in bijection with the (classical) alcoves contained in the fundamental $J$-alcove. 

We note, in passing, that $\WJ$ has the following characterisations (in particular, the third characterisation below shows that $\WJ$ may be regarded as an affine analogue of $W^J$, see Lemma~\ref{lem:decomposition}, and recall the definition of $\widetilde{\Phi}(w)$ from Section~\ref{sec:affrtsys}).

\begin{lemma} We have
\begin{align*}
\WJ&=\{w\in \Wext\mid \ell(s_{\beta,k}w)>\ell(w)\text{ for all $\beta+k\delta\in\Phi_J+ \ZZ\delta$}\}\\
&=\{w\in \Wext\mid \ell(s_j'w)>\ell(w)\text{ for all $j\in \Jaff$}\}\\
&=\{w\in\Wext\mid \widetilde{\Phi}(w)\cap(\Phi_J+\ZZ\delta)=\emptyset\}.
\end{align*}
\end{lemma}

\begin{proof}
We have $w\in \WJ$ if and only if the alcove $wA_0$ lies on the same side of $H_{\beta,k}$ as $A_0$ for all $(\beta,k)\in\Phi_J\times \ZZ$, if and only if $\ell(s_{\beta,k}w)>\ell(w)$. Hence the first equality. For the second equality, we similarly have $w\in \WJ$ if and only if the alcove $wA_0$ lies on the same side of $H_{\alpha_j,0}$ and $H_{\varphi_K,1}$ as $A_0$ for all $j\in J$ and $K\in\cK(J)$ (see Lemma~\ref{lem:boundingwalls}). The third characterisation follows from the first characterisation and the fact that if $\beta-k\delta\in\widetilde{\Phi}^+$ then $\ell(s_{\beta,k}w)<\ell(w)$ if and only if $\beta+k\delta\in\widetilde{\Phi}(w)$. 
\end{proof}

\begin{defn}\label{defn:utau}
For $\lambda\in P^{(J)}$ let 
$$
\sy_{\lambda}=\sw_{J\backslash J_{\lambda}}\sw_J\quad\text{and}\quad \st_{\lambda}=t_{\lambda}\sy_{\lambda},
$$
where $J_{\lambda}$ is as in~(\ref{eq:Jlambda}).
\end{defn}

Note that both $\sy_{\lambda}$ and $\st_{\lambda}$ depend on the subset $J$, however this dependence is suppressed in the notation.

\begin{lemma}\label{lem:utbasics1}
For $\lambda\in P^{(J)}$ we have
$
\Phi(\sy_{\lambda})=\Phi_J^+\backslash\Phi_{J\backslash J_{\lambda}}^+=\{\alpha\in\Phi_J^+\mid \langle\lambda,\alpha\rangle=1\}.
$
\end{lemma}

\begin{proof}
It is clear from the definition that $\Phi(\sy_{\lambda})=\Phi_J^+\backslash\Phi_{J\backslash J_{\lambda}}^+$. A root $\alpha\in\Phi_J^+$ does not lie in $\Phi_{J\backslash J_{\lambda}}^+$ if and only there exists $j\in J_{\lambda}$ such that the coefficient of $\alpha_j$ in $\alpha$ is strictly positive. By Lemma~\ref{lem:restrictingstrip} this occurs if and only if the coefficient of $\alpha_j$ in $\alpha$ is $1$, and hence the result. 
\end{proof}

The following theorem gives an explicit decomposition of $\WJ$ (see Example~\ref{ex:Other} for an illustration). 

\begin{thm}\label{thm:WJ}
We have $\WJ=\{\st_{\lambda}u\mid \lambda\in P^{(J)}\text{ and }u\in W^J\}.$
\end{thm}

\begin{proof}
We first show that $\st_{\lambda}\in\WJ$ for all $\lambda\in P^{(J)}$. The hyperplane in the parallelism class of $\alpha\in\Phi_J$ passing through $\lambda$ is $H_{\alpha,\langle\lambda,\alpha\rangle}$, and we have $\langle\lambda,\alpha\rangle\in\{0,1\}$ (as $\lambda\in P^{(J)}$). Thus the hyperplanes in the parallelism classes of $\Phi_J$ separating the alcove $t_{\lambda}A_0$ from $\cA_J$ are precisely the hyperplanes $H_{\alpha,1}$ with $\alpha\in\Phi_J^+$ such that $\langle\lambda,\alpha\rangle=1$. By Lemma~\ref{lem:utbasics1} these are precisely the hyperplanes separating $t_{\lambda}A_0$ from $t_{\lambda}\sy_{\lambda}A_0$ (because the corresponding linear hyperplanes $H_{\alpha,0}$ are the hyperplanes separating $e$ from $\sy_{\lambda}$, and translation by $\lambda$ preserves orientation and shifts these hyperplanes). Hence $\st_{\lambda}A_0\subseteq \cA_J$ and so $\st_{\lambda}\in\WJ$.

Let $u\in W^J$ and suppose that $\st_{\lambda}u\notin\WJ$. That is, $\st_{\lambda}uA_0\not\subseteq\cA_J$. Since $\st_{\lambda}A_0\subseteq \cA_J$ (from the previous paragraph) it follows that there is a hyperplane $H_{\alpha,k}$ separating $t_{\lambda}\sy_{\lambda}A_0$ from $t_{\lambda}\sy_{\lambda}uA_0$ with $\alpha\in\Phi_J^+$ and $k\in\{0,1\}$. Translating by $t_{-\lambda}$ implies that the hyperplane $H_{\alpha,0}$ separates $\sy_{\lambda}$ from $\sy_{\lambda}u$, contradicting Lemma~\ref{lem:decomposition}. Thus $\{\st_{\lambda}u\mid \lambda\in P^{(J)}\text{ and }u\in W^J\}\subseteq \WJ$. 

For the reverse containment, suppose that $w\in\WJ$. Write $w=t_{\lambda}u_1u_2$ with $\lambda=\wt(w)$, $u_1\in W_J$, and $u_2\in W^J$. It is clear that $\lambda\in P^{(J)}$ (as $w(0)=\lambda$ and $0\in A_0$). Since $\Phi_J(u_2)=\emptyset$ (by Lemma~\ref{lem:decomposition}) there are no walls of $\cA_J$ separating $t_{\lambda}u_1$ from $t_{\lambda}u_1u_2$. Thus $t_{\lambda}u_1A_0\subseteq\cA_J$. The argument of the first paragraph, combined with the fact that $\Phi_J(u_2)=\emptyset$, shows that the hyperplanes separating $t_{\lambda}$ from $t_{\lambda}u_1$ are precisely the hyperplanes $H_{\alpha,1}$ with $\alpha\in\Phi_J^+$ with $\langle\lambda,\alpha\rangle=1$, and it follows that $u_1=\sy_{\lambda}$. So $w=\st_{\lambda}u_2$ with $u_2\in W^J$.
\end{proof}

\begin{remark}
The first paragraph of the proof of Theorem~\ref{thm:WJ} shows that if $\lambda\in P^{(J)}$ then $\sy_{\lambda}$ may be characterised as the unique element of $W_J$ such that $t_{\lambda}\sy_{\lambda}\in \WJ$.
\end{remark}

The following immediate consequence of Theorem~\ref{thm:WJ} will play a role later. Recall the definition of $\theta_J(w)$ and $\theta^J(w)$ from Definition~\ref{defn:theta}. 

\begin{cor}\label{cor:splitting2}
If $w\in\WJ$ then $\theta_J(w)=\sy_{\wt(w)}$, and so $w=\st_{\wt(w)}\theta^J(w)$.  
\end{cor}

\begin{proof}
By Theorem~\ref{thm:WJ} $w=\st_{\lambda}u=t_{\lambda}\sy_{\lambda}u$ with $\lambda=\wt(w)$ and $u\in W^J$, hence the result. 
\end{proof}

 \subsection{The $J$-translation group~$\TJ$}

We have seen in Lemma~\ref{lem:projection} that $ P^{(J)}$ is in bijection with $P/Q_J$ and hence the former inherits a group structure of the later via $\lambda\mapsto \lambda+Q_J$. In this section we introduce the \textit{$J$-translation group}
$$
\TJ=\{\st_{\lambda}\mid \lambda\in P^{(J)}\}
$$
and show that this subset of $\Wext$ is a group realising the above group structure (see Corollary~\ref{cor:Tiso}). This group interpolates between $P$ (when $J=\emptyset$) and $\Sigma$ (when $J=I$). The following lemma gives some important basic properties of the elements $\sy_{\lambda}$ and $\st_{\lambda}$.

\begin{lemma}\label{lem:utbasics}
Let $\lambda,\mu\in P^{(J)}$. Then
\begin{compactenum}[$(1)$]
\item $(\lambda+\mu)^{(J)}=\lambda+\sy_{\lambda}\mu$ and $(-\lambda)^{(J)}=-\sy_{\lambda}^{-1}\lambda$;
\item $\sy_{\lambda}\sy_{\mu}=\sy_{(\lambda+\mu)^{(J)}}=\sy_{\mu}\sy_{\lambda}$ and $\sy_{\lambda}^{-1}=\sy_{(-\lambda)^{(J)}}$;
\item $\st_{\lambda}\st_{\mu}=\st_{(\lambda+\mu)^{(J)}}=\st_{\mu}\st_{\lambda}$ and $\st_{\lambda}^{-1}=\st_{(-\lambda)^{(J)}}$.
\end{compactenum}
\end{lemma}

\begin{proof}
(1) Since $\sy_{\lambda}\in W_J$ we have $\lambda+\sy_{\lambda}\mu\in \lambda+\mu+Q_J$, and so to prove that $(\lambda+\mu)^{(J)}=\lambda+\sy_{\lambda}\mu$ it suffices, by the uniqueness in Lemma~\ref{lem:projection}, to show that $\lambda+\sy_{\lambda}\mu\in  P^{(J)}$. To do this, let $\alpha\in\Phi_J^+$, and write $\beta=\sy_{\lambda}^{-1}\alpha$. Then 
$
\langle\lambda+\sy_{\lambda}\mu,\alpha\rangle=\langle\lambda,\alpha\rangle+\langle\mu,\beta\rangle.
$
Since $\lambda\in P^{(J)}$ we have $\langle\lambda,\alpha\rangle\in \{0,1\}$. If $\langle\lambda,\alpha\rangle=0$ then $\alpha\notin\Phi(\sy_{\lambda})$ (by Lemma~\ref{lem:utbasics1}) and so $\beta=\sy_{\lambda}^{-1}\alpha\in\Phi_J^+$, and so $\langle\lambda+\sy_{\lambda}\mu,\alpha\rangle=\langle\mu,\beta\rangle$, giving 
$0\leq \langle\lambda+\sy_{\lambda}\mu,\beta\rangle\leq 1$ (as $\beta\in\Phi_J^+$ and $\mu\in\cA_J$). If $\langle\lambda,\alpha\rangle=1$ then $\alpha\in\Phi(\sy_{\lambda})$ and so $\beta\in-\Phi_J^+$. Thus $\langle\lambda+\sy_{\lambda}\mu,\alpha\rangle=1+\langle\mu,\beta\rangle$ and so $0\leq \langle\lambda+\sy_{\lambda}\mu,\alpha\rangle\leq 1$ (as $-1\leq \langle\mu,\beta\rangle\leq 0$ as $\beta\in-\Phi_J^+$ and $\mu\in\cA_J$). Hence $\lambda+\sy_{\lambda}\mu\in  P^{(J)}$. 

To show that $(-\lambda)^{(J)}=-\sy_{\lambda}^{-1}\lambda$ one shows, in a similar way, that $-\sy_{\lambda}^{-1}\lambda\in  P^{(J)}$.

(2) By (1) it suffices to show that $\sy_{\lambda}\sy_{\mu}=\sy_{\lambda+\sy_{\lambda}\mu}$ and $\sy_{\lambda}^{-1}=\sy_{-\sy_{\lambda}^{-1}\lambda}$. To prove the first statement we shall show that $\Phi(\sy_{\lambda+\sy_{\lambda}\mu})=\Phi(\sy_{\lambda}\sy_{\mu})$. It follows from Lemma~\ref{lem:utbasics1} that
$$
\Phi(\sy_{\lambda+\sy_{\lambda}\mu})=\{\alpha\in\Phi_J^+\mid \langle\lambda,\alpha\rangle=0\text{ and }\langle \mu,\sy_{\lambda}^{-1}\alpha\rangle=1,\text{ or } \langle\lambda,\alpha\rangle=1\text{ and }\langle \mu,\sy_{\lambda}^{-1}\alpha\rangle=0\}.
$$
Suppose that $\alpha\in\Phi(\sy_{\lambda+\sy_{\lambda}\mu})$. If $\langle\lambda,\alpha\rangle=1$ and $\langle \mu,\sy_{\lambda}^{-1}\alpha\rangle=0$ then $\alpha\in\Phi(\sy_{\lambda})$ and $-\sy_{\lambda}^{-1}\alpha\notin\Phi(\sy_{\mu})$, and so $\alpha\in\Phi(\sy_{\lambda}\sy_{\mu})$. If $\langle\lambda,\alpha\rangle=0$ and $\langle\mu,\sy_{\lambda}^{-1}\alpha\rangle=1$ then $\alpha\in\Phi_J^+\backslash\Phi(\sy_{\lambda})$ and $\sy_{\lambda}^{-1}\alpha\in\Phi(\sy_{\mu})$, giving $\alpha\in\Phi(\sy_{\lambda}\sy_{\mu})$. Conversely, suppose that $\alpha\in\Phi(\sy_{\lambda}\sy_{\mu})$. Then $\alpha\in\Phi_J^+$ with $\sy_{\mu}^{-1}\sy_{\lambda}^{-1}\alpha<0$, and there are two cases. If $\sy_{\lambda}^{-1}\alpha>0$ and $\sy_{\mu}^{-1}(\sy_{\lambda}^{-1}\alpha)<0$ then $\langle\lambda,\alpha\rangle=0$ (as $\alpha\notin\Phi(\sy_{\lambda})$) and $\langle \mu,\sy_{\lambda}^{-1}\alpha\rangle=1$, so $\alpha\in\Phi(\sy_{\lambda+\sy_{\lambda}\mu})$. If $\sy_{\lambda}^{-1}\alpha<0$ and $\sy_{\mu}^{-1}(-\sy_{\lambda}^{-1}\alpha)>0$ then $\langle\lambda,\alpha\rangle=1$ and $\langle\mu,-\sy_{\lambda}^{-1}\alpha\rangle=0$, and so again $\alpha\in\Phi(\sy_{\lambda+\sy_{\lambda}\mu})$. Hence $\sy_{\lambda}\sy_{\mu}=\sy_{\lambda+\sy_{\lambda}\mu}=\sy_{(\lambda+\mu)^{(J)}}=\sy_{\mu}\sy_{\lambda}$.

The statement $\sy_{\lambda}^{-1}=\sy_{-\sy_{\lambda}^{-1}\lambda}$ follows from the general formula
$
\Phi(w^{-1})=\{-w\alpha\mid \alpha\in\Phi(w)\}.
$

(3) By (1) and (2) we have
$$
\st_{\lambda}\st_{\mu}=t_{\lambda}\sy_{\lambda}t_{\mu}\sy_{\mu}=t_{\lambda+\sy_{\lambda}\mu}\sy_{\lambda}\sy_{\mu}=t_{(\lambda+\mu)^{(J)}}\sy_{(\lambda+\mu)^{(J)}}=\st_{(\lambda+\mu)^{(J)}}.$$ 
Similarly $\st_{\lambda}^{-1}=\st_{(-\lambda)^{(J)}}$.
\end{proof}

\begin{cor}\label{cor:Tiso}
We have $\TJ\cong P/Q_J$. In particular, $\TJ$ is an abelian group of rank~$|I\backslash J|$. Moreover $\TJ$ acts freely on $\WJ$ with fundamental domain $W^J$. 
\end{cor}

\begin{proof}
Lemma~\ref{lem:utbasics}(3) shows that $\TJ$ is a group, and that the map $f:P\to \TJ$ given by $f(\lambda)=\st_{\lambda^{(J)}}$ is a group homomorphism. This homomorphism is surjective (as $\lambda^{(J)}=\lambda$ for $\lambda\in  P^{(J)}$) and $\mathrm{ker}(f)=Q_J$.

If $\lambda\in P^{(J)}$ and $w\in \WJ$ then by Theorem~\ref{thm:WJ} we have $w=\st_{\mu}u$ with $\mu=\wt(w)\in  P^{(J)}$ and $u\in W^J$. By Lemma~\ref{lem:utbasics}(3) we have $\st_{\lambda}\cdot w=\st_{\lambda}\st_{\mu}u=\st_{(\lambda+\mu)^{(J)}}u$, and hence $\st_{\lambda}\cdot w\in \WJ$ by Theorem~\ref{thm:WJ}. Thus $\TJ$ acts on $\WJ$. It is clear that this action is free, and that $W^J$ is a fundamental domain. 
\end{proof}

\begin{remark}
By Corollary~\ref{cor:Tiso} the map $\TJ\to P/Q_J$ with $\st_{\lambda}\to \lambda^{(J)}+Q_J$ is an isomorphism. In contrast, we note that the map $\TJ\to P^J$ with $\st_{\lambda}\mapsto\lambda^J$ is a surjective, not necessarily injective, group homomorphism (the fact that the map is a surjective homomorphism follows from Lemma~\ref{lem:utbasics}(3) and the obvious fact that $(\lambda^{(J)})^J=\lambda^J$, and to see that the map is not necessarily injective consider the extreme case $J=I$ where $P^J=\{0\}$ while $\TJ=P/Q=\Sigma$ may be nontrivial). 
\end{remark}

Recall the classical formula $\ell(t_{\la})=\sum_{\alpha\in\Phi_1^+}|\langle\la,\alpha\rangle|$ (obtained by counting hyperplanes crossed). The following proposition gives an analogous formula for the elements~$\tau_{\la}$. 

\begin{prop}\label{prop:tlength}
For $\la\in P^{(J)}$ we have 
$$
\ell(\tau_{\la})=\sum_{\alpha\in\Phi_1^+\backslash \Phi_{J}}|\langle\la,\alpha\rangle|.
$$
\end{prop}

\begin{proof}
By \cite[(2.4.1)]{Mac:03} we have (for $\la\in P$ and $w\in \Wfin$)
$$
\ell(t_{\la}w)=\sum_{\alpha\in\Phi_1^+}|\langle\la,\alpha\rangle-\chi^-(w^{-1}\alpha)|,
$$
where $\chi^{-}(\cdot)$ is the characteristic function of $-\Phi^+$. Now suppose that $\la\in P^{(J)}$ and consider the contribution
$|\langle\la,\alpha\rangle-\chi^-(\sy_{\la}^{-1}\alpha)|$ to $\ell(\tau_{\la})$ in the above sum from $\alpha\in\Phi_1^+$. If $\alpha\in \Phi_{J}^+$ then either $\langle \la,\alpha\rangle=0$ in which case $\chi^-(\sy_{\la}^{-1}\alpha)=0$, or $\langle\la,\alpha\rangle=1$ in which case $\chi^-(\sy_{\la}^{-1}\alpha)=1$ (in both cases using Lemma~\ref{lem:utbasics1}). Thus if $\alpha\in\Phi_{J}^+$ then the contribution to the sum is~$0$. If $\alpha\in\Phi_1^+\backslash\Phi_{J}$ then $\chi^-(\sy_{\la}^{-1}\alpha)=0$, hence the result. 
\end{proof}

\begin{example}\label{ex:Other}
Figure~\ref{fig:G2TJ} illustrates the decomposition of $\Wext$ (which equals $\Waff$ in this case) into $J$-alcoves for type $\sG_2$ with $J=\{1\}$. The $J$ alcoves are shaded blue and green, with the dark green region denoting the fundamental $J$-alcove $\cA_J$ (the grey alcoves are also part of~$\cA_J$). Moreover the elements $\st_{\lambda}$ with $\lambda\in  P^{(J)}$ are shaded grey, and the decomposition of~$\cA_J$ given by Theorem~\ref{thm:WJ} and Corollary~\ref{cor:Tiso} is illustrated by dotted lines. Note that $\la^J\neq\la^{(J)}$ in general (c.f. Definition~\ref{defn:projections}), for example $\omega_1^{(J)}=\omega_1$ while $\omega_1^J=\frac{3}{2}\omega_2$. We have $\omega_1-\omega_2=\frac{1}{2}\omega_2$ and so $P^J=\mathbb{Z}\overline{\omega}_2$ where $\overline{\omega}_2=\frac{1}{2}\omega_2$ (this lattice is denoted as the union of the solid and open circles on the hyperplane $H_{\alpha_1,0}$). We have $P_J/Q_J=\ZZ/2\ZZ$.

\begin{figure}[H]
\centering
\begin{tikzpicture}[scale=1.5]
\clip ({-5*0.866+0.433},{-2}) rectangle + ({8*0.866},{5.5});
\path [fill=green2] ({-4*0.866},-3)--({-2*0.866},-3)--({3*0.866},4.5)--({1*0.866},4.5)--({-4*0.866},-3);
\path [fill=green1] ({-5*0.866},1.5)--({-3*0.866},4.5)--({-5*0.866},4.5)--({-5*0.866},1.5);
\path [fill=blue1] ({-5*0.866},-1.5)--({-1*0.866},4.5)--({-3*0.866},4.5)--({-5*0.866},1.5)--({-5*0.866},-1.5);
\path [fill=green1] ({-5*0.866},-1.5)--({-5*0.866},-3)--({-4*0.866},-3)--({1*0.866},4.5)--({-1*0.866},4.5)--({-5*0.866},-1.5);
\path [fill=green1] ({0*0.866},-3)--({5*0.866},4.5)--({3*0.866},4.5)--({-2*0.866},-3)--({0*0.866},-3);
\path [fill=blue1] ({2*0.866},-3)--({5*0.866},1.5)--({5*0.866},4.5)--({0*0.866},-3)--({2*0.866},-3);
\path [fill=green1] ({2*0.866},-3)--({5*0.866},1.5)--({5*0.866},-1.5)--({4*0.866},-3)--({2*0.866},-3);
\path [fill=blue1] ({4*0.866},-3)--({5*0.866},-1.5)--({5*0.866},-3)--({4*0.866},-3);
\path [fill=black!70] (0,0) -- (0.433,0.75) -- (0,1) -- (0,0);
\path [fill=black!50] ({0+2*0.433},{0+1.5}) -- ({0.433+2*0.433},{0.75+1.5}) -- ({0+2*0.433},{1+1.5}) -- ({0+2*0.433},{0+1.5});
\path [fill=black!50] ({0+4*0.433},{0+2*1.5}) -- ({0.433+4*0.433},{0.75+2*1.5}) -- ({0+4*0.433},{1+2*1.5}) -- ({0+4*0.433},{0+2*1.5});
\path [fill=black!50] ({0-2*0.433},{0-1.5}) -- ({0.433-2*0.433},{0.75-1.5}) -- ({0-2*0.433},{1-1.5}) -- ({0-2*0.433},{0-1.5});
\path [fill=black!50] ({0-4*0.433},{0-2*1.5}) -- ({0.433-4*0.433},{0.75-2*1.5}) -- ({0-4*0.433},{1-2*1.5}) -- ({0-4*0.433},{0-2*1.5});
\path [fill=black!50] ({0-6*0.433},{0-1*1.5}) -- ({0-6*0.433+0.433},{0-1*1.5+0.75}) -- ({0-6*0.433+2*0.433},{0-1*1.5+0.5}) -- ({0-6*0.433},{0-1*1.5});
\path [fill=black!50] ({0-4*0.433},{0-0*1.5}) -- ({0-4*0.433+0.433},{0-0*1.5+0.75}) -- ({0-4*0.433+2*0.433},{0-0*1.5+0.5}) -- ({0-4*0.433},{0-0*1.5});
\path [fill=black!50] ({0-2*0.433},{0+1*1.5}) -- ({0-2*0.433+0.433},{0+1*1.5+0.75}) -- ({0-2*0.433+2*0.433},{0+1*1.5+0.5}) -- ({0-2*0.433},{0+1*1.5});
\path [fill=black!50] ({0-0*0.433},{0+2*1.5}) -- ({0-0*0.433+0.433},{0+2*1.5+0.75}) -- ({0-0*0.433+2*0.433},{0+2*1.5+0.5}) -- ({0-0*0.433},{0+2*1.5});
\draw [line width=2pt,dotted] (0.433,0.75)--(0,1)--(-2*0.433,0.5)--(-2*0.433,-0.5)--(-0.433,-0.75);
\draw [line width=2pt,dotted] (-2*0.433,-0.5)--(-4*0.433,-1)--(-4*0.433,-2);
\draw [line width=2pt,dotted] (0,1)--(0,2)--(2*0.433,2.5)--(3*0.433,2.25);
\draw [line width=2pt,dotted] (2*0.433,2.5)--(2*0.433,3.5);
\draw [line width=2pt,dotted] (0,2)--(-0.433,2.25);
\draw [line width=2pt,dotted] (0-2*0.433,2-1.5)--(-0.433-2*0.433,2.25-1.5);
\draw [line width=2pt,dotted] (0-4*0.433,2-2*1.5)--(-0.433-4*0.433,2.25-2*1.5);
\draw(-4.33,4.5)--(4.33,4.5);
\draw(-4.33,3)--(4.33,3);
\draw(-4.33,1.5)--(4.33,1.5);
\draw(-4.33,0)--(4.33,0);
\draw(-4.33,-1.5)--(4.33,-1.5);
\draw(-4.33,-3)--(4.33,-3);
\draw(-4.33,-3)--(-4.33,4.5);
\draw(-3.464,-3)--(-3.464,4.5);
\draw(-2.598,-3)--(-2.598,4.5);
\draw(-1.732,-3)--(-1.732,4.5);
\draw(-.866,-3)--(-.866,4.5);
\draw(0,-3)--(0,4.5);
\draw(.866,-3)--(.866,4.5);
\draw(1.732,-3)--(1.732,4.5);
\draw(2.598,-3)--(2.598,4.5);
\draw(3.464,-3)--(3.464,4.5);
\draw(4.33,-3)--(4.33,4.5);
\draw(-4.33,3.5)--({-3*0.866},4.5);
\draw(-4.33,2.5)--({-1*0.866},4.5);
\draw(-4.33,1.5)--({1*0.866},4.5);
\draw(-4.33,.5)--({3*0.866},4.5);
\draw(-4.33,-.5)--(4.33,4.5);
\draw(-4.33,-1.5)--(4.33,3.5);
\draw(-4.33,-2.5)--(4.33,2.5);
\draw(-3.464,-3)--(4.33,1.5);
\draw(-1.732,-3)--(4.33,.5);
\draw(0,-3)--(4.33,-.5);
\draw(1.732,-3)--(4.33,-1.5);
\draw(3.464,-3)--(4.33,-2.5);
\draw(4.33,3.5)--({3*0.866},4.5);
\draw(4.33,2.5)--({1*0.866},4.5);
\draw(4.33,1.5)--({-1*0.866},4.5);
\draw(4.33,.5)--({-3*0.866},4.5);
\draw(4.33,-.5)--(-4.33,4.5);
\draw(4.33,-1.5)--(-4.33,3.5);
\draw(4.33,-2.5)--(-4.33,2.5);
\draw(3.464,-3)--(-4.33,1.5);
\draw(1.732,-3)--(-4.33,.5);
\draw(0,-3)--(-4.33,-.5);
\draw(-1.732,-3)--(-4.33,-1.5);
\draw(-3.464,-3)--(-4.33,-2.5);
\draw(-4.33,-1.5)--(-3.464,-3);
\draw(-4.33,1.5)--(-1.732,-3);
\draw(-4.33,4.5)--(0,-3);
\draw({-3*0.866},4.5)--(1.732,-3);
\draw({-1*0.866},4.5)--(3.464,-3);
\draw({1*0.866},4.5)--(4.33,-1.5);
\draw({3*0.866},4.5)--(4.33,1.5);
\draw(4.33,-1.5)--(3.464,-3);
\draw(4.33,1.5)--(1.732,-3);
\draw(4.33,4.5)--(0,-3);
\draw({3*0.866},4.5)--(-1.732,-3);
\draw({1*0.866},4.5)--(-3.464,-3);
\draw({-1*0.866},4.5)--(-4.33,-1.5);
\draw({-3*0.866},4.5)--(-4.33,1.5);
\draw[line width=2pt]({-5*0.866},1.5)--({-3*0.866},4.5);
\draw[line width=2pt]({-5*0.866},-1.5)--({-1*0.866},4.5);
\draw[line width=2pt]({-4*0.866},-3)--({1*0.866},4.5);
\draw[line width=2pt]({-2*0.866},-3)--({3*0.866},4.5);
\draw[line width=2pt]({0*0.866},-3)--({5*0.866},4.5);
\draw[line width=2pt]({2*0.866},-3)--({5*0.866},1.5);
\draw[line width=2pt]({4*0.866},-3)--({5*0.866},-1.5);
\node at ({-4*0.866},{2*1.5}) {\Large{$\bullet$}};
\node at ({-2*0.866},{2*1.5}) {\Large{$\bullet$}};
\node at ({0*0.866},{2*1.5}) {\Large{$\bullet$}};
\node at ({2*0.866},{2*1.5}) {\Large{$\bullet$}};
\node at ({4*0.866},{2*1.5}) {\Large{$\bullet$}};
\node at ({-3*0.866},{1.5}) {\Large{$\bullet$}};
\node at ({-1*0.866},{1.5}) {\Large{$\bullet$}};
\node at ({1*0.866},{1.5}) {\Large{$\bullet$}};
\node at ({3*0.866},{1.5}) {\Large{$\bullet$}};
\node at ({-4*0.866},0) {\Large{$\bullet$}};
\node at ({-2*0.866},0) {\Large{$\bullet$}};
\node at ({0*0.866},0) {\Large{$\bullet$}};
\node at ({2*0.866},0) {\Large{$\bullet$}};
\node at ({4*0.866},0) {\Large{$\bullet$}};
\node at ({-3*0.866},{-1*1.5}) {\Large{$\bullet$}};
\node at ({-1*0.866},{-1*1.5}) {\Large{$\bullet$}};
\node at ({1*0.866},{-1*1.5}) {\Large{$\bullet$}};
\node at ({3*0.866},{-1*1.5}) {\Large{$\bullet$}};
\node at ({0.433},{0.75}) {\huge{$\circ$}};
\node at ({3*0.433},{3*0.75}) {\huge{$\circ$}};
\node at ({-1*0.433},{-1*0.75}) {\huge{$\circ$}};
\node at (-6*0.433-0.3,1.5+0.3) {$\alpha_1^{\vee}$};
\node at (4*0.433+0.4,0*1.5+0.2) {$\alpha_2^{\vee}$};
\node at (-0.4,3) {$\omega_1$};
\node at (2*0.433+0.3,1.5) {$\omega_2$};
\draw[line width=2pt,-latex](0,0)--(-6*0.433,1.5);
\draw[line width=2pt,-latex](0,0)--(4*0.433,0);
\end{tikzpicture}
\caption{$J$-alcoves and the group $\TJ$ for $\Phi=\sG_2$ and $J=\{1\}$}\label{fig:G2TJ}
\end{figure}
\end{example}

\subsection{Symmetries of the fundamental $J$-alcove}

In this section we determine the subgroup of $\Wext$ stabilising the fundamental $J$-alcove~$\cA_J$.

\begin{defn}
Let $G_{J}=\{g\in\Wfin\mid g\cA_{J}=\cA_{J}\}$ be the subgroup of $\Wfin$ stabilising $\cA_{J}$. 
\end{defn}

\begin{thm}\label{thm:mainG}
We have $G_{J}=\{g\in \Wfin\mid g\Phi_{J}^+=\Phi_{J}^+\}$. Moreover, for $g\in G_{J}$ and $\alpha\in \Phi_{J}^+$ we have $\height(g\alpha)=\height(\alpha)$. In particular $g$ maps the simple roots of $\Phi_{J}$ to the simple roots of $\Phi_{J}$, and hence induces a permutation of~$J$. This permutation maps connected components of $J$ to connected components of~$J$. 
\end{thm}

\begin{proof}
Let $g\in G_{J}$. Then for all $\alpha\in\Phi_{J}^+$ and all $\la\in \cA_{J}$ we have $0\leq \langle g\la,\alpha\rangle\leq 1$. We claim that this forces $g^{-1}\Phi_{J}^+=\Phi_{J}^+$. For if $g^{-1}\alpha\notin\Phi_{J}^+$ then either $g^{-1}\alpha\in -\Phi_{J}^+$, or there is $i\in \mathrm{supp}(g^{-1}\alpha)$ with $i\notin J$. In the former case, choose any $j\in\mathrm{supp}(g^{-1}\alpha)$ (and so $j\in J$) and let $m_j^K$ be the coefficient of $\alpha_j$ in $\varphi_K$, where $K$ is the connected component of $J$ containing $j$. Then $\la=\omega_j/m_j^K\in\cA_J$ (for if $\beta\in\Phi_J^+\backslash\Phi_K$ then $\langle\la,\beta\rangle=0$ and if $\beta\in\Phi_K^+$ then $0\leq \langle\la,\beta\rangle\leq \langle\la,\varphi_K\rangle=1$), and $\langle \la,g^{-1}\alpha\rangle<0$ gives a contradiction. In the latter case, taking $\la=2\omega_i\in \cA_{J}$ gives $|\langle g\la,\alpha\rangle|=|\langle \la,g^{-1}\alpha\rangle|\geq 2$, again a  contradiction. Thus $g\Phi_{J}^+=\Phi_{J}^+$. On the other hand, if $g\Phi_{J}^+=\Phi_{J}^+$ and $\la\in\cA_{J}$ then for $\alpha\in\Phi_{J}^+$ we have $\langle g\la,\alpha\rangle=\langle \la,g^{-1}\alpha\rangle$, which lies between $0$ and $1$ (since $g^{-1}\alpha\in\Phi_{J}^+$ and $\la\in\cA_{J}$), and so $g\cA_{J}=\cA_{J}$.

Let $g\in G_{J}$. Since $g\Phi_{J}^+=\Phi_{J}^+$ we have $\height(g\alpha)\geq \mathrm{ht}(\alpha)$ (because each $g\alpha_j$ has height at least~$1$ and $g$ is linear). Since $g^{-1}\in G_{J}$ we also have $\mathrm{ht}(g^{-1}\beta)\geq \mathrm{ht}(\beta)$ for all $\beta\in\Phi_{J}^+$, and by hypothesis we have $\beta=g\alpha$ for some $\alpha\in\Phi_{J}^+$, and so $\mathrm{ht}(\alpha)\geq \mathrm{ht}(g\alpha)$. Thus $\mathrm{ht}(g\alpha)=\mathrm{ht}(\alpha)$ for all $\alpha\in\Phi_{J}^+$. In particular each $g\in G_J$ maps simple roots of $\Phi_{J}$ to simple roots of $\Phi_{J}$, and hence induces a permutation of~$J$ by $\alpha_{g(j)}=g\alpha_j$ for $j\in J$.  

We now show that for each $g\in G_{J}$ the permutation $g:J\to J$ maps connected components to connected components. Let $K$ be a connected component of $J$, and let $\varphi_K$ be the highest root of $\Phi_K$. We claim that $g\varphi_K=\varphi_{K'}$ is the highest root of some connected component $K'\subseteq J$. For if not, then there is $i\in J$ such that $\alpha=g\varphi_K+\alpha_i$ is a root of $\Phi_{J}$. But then $g^{-1}\alpha=\varphi_K+\alpha_{g^{-1}(i)}$ is a root, necessarily of $\Phi_K$, of height exceeding $\mathrm{ht}(\varphi_K)$, a contradiction. Thus $g$ maps highest roots to highest roots, and so $g$ maps $K$ to $K'$. 
\end{proof}

\begin{lemma} If $\la\in P^{(J)}$ and $g\in G_{J}$ then 
$$
g\sy_{\la}g^{-1}=\sy_{g\la},\quad g\tau_{\la}g^{-1}=\tau_{g\la},\quad\text{and}\quad \ell(\tau_{g\la})=\ell(\tau_{\la}).
$$ 
\end{lemma}

\begin{proof} 
To prove that $g\sy_{\la}g^{-1}=\sy_{g\la}$ it is sufficient to show that $\Phi(g\sy_{\la}g^{-1})=\Phi(\sy_{g\la})$. If $\alpha\in\Phi(\sy_{g\la})$ then $\alpha\in\Phi_{J}^+$ with $\langle g\la,\alpha\rangle=1$. Thus $\langle\la,g^{-1}\alpha\rangle=1$ and so $g^{-1}\alpha\in\Phi(\sy_{\la})$. Hence $g\sy_{\la}g^{-1}\alpha\in-\Phi_{J}^+$, showing that $\Phi(\sy_{g\la})\subseteq \Phi(g\sy_{\la}g^{-1})$. On the other hand, suppose that $\alpha\in\Phi(g\sy_{\la}g^{-1})$. Since $g$ maps simple roots of $\Phi_{J}$ to simple roots of $\Phi_{J}$ we have $gs_jg^{-1}=s_{g(j)}\in W_J$ for all $j\in J$, and so $g\sy_{\la}g^{-1}\in W_{J}$. Thus $\alpha\in\Phi_{J}$, and so $g^{-1}\alpha\in\Phi_{J}^+$. If $g^{-1}\alpha\notin\Phi(\sy_{\la})$ then $g\sy_{\la}g^{-1}\alpha\in\Phi_{J}^+$, a contradiction, and so $g^{-1}\alpha\in\Phi(\sy_{\la})$. Thus $\langle \la,g^{-1}\alpha\rangle=1$, and so $\langle g\la,\alpha\rangle=1$, giving $\alpha\in\Phi(\sy_{g\la})$ as required. 

It then follows that $g\tau_{\la}g^{-1}=t_{g\la}(g\sy_{\la}g^{-1})=t_{g\la}\sy_{g\la}=\tau_{g\la}$ for all $\la\in P^{(J)}$. By Proposition~\ref{prop:tlength} we have $\ell(\tau_{g\la})=\sum_{\alpha\in\Phi_1^+\backslash \Phi_{J}}|\langle\la,g^{-1}\alpha\rangle|=\sum_{\alpha\in\Phi_1^+\backslash \Phi_{J}}|\langle\la,\alpha\rangle|=\ell(\tau_{\la})$.
\end{proof}

\begin{cor}
The subgroup of $\Wext$ stabilising $\cA_{J}$ is $\TJ\rtimes G_{J}$. 
\end{cor}

\begin{proof}
Let $w\in \Wext$ and suppose that $w\cA_{J}=\cA_{J}$. Let $\la=\wt(w)$, and so $w(0)=\la\in P^{(J)}$. Then $\tau_{\la}^{-1}w(0)=0$, and so $g=\tau_{\la}^{-1}w\in \Wfin$ with $g\cA_{J}=\cA_{J}$, and so $g\in G_{J}$. Thus $w=\tau_{\la}g\in \TJ G_{J}$. 
\end{proof}

Note that the group $\TJ\rtimes G_J$ plays the role of the ``extended affine Weyl group'' of $\cA_J$ in the sense that if $J=\emptyset$ we have $\TJ\rtimes G_{J}=\Wext$. 

\goodbreak

\section{$J$-folded alcove paths and $J$-parameter systems}\label{sec:Jpaths}

In this section we introduce positively $J$-folded alcove paths, generalising the positively folded alcove paths of Ram~\cite{Ram:06} and the $2$-dimensional theory from~\cite{GP:19,GP:19b}. We also introduce the notion of a \textit{$J$-parameter system}, a combinatorial object that will be useful in indexing later objects in this paper.

\subsection{$J$-folded alcove paths}

We introduce the following definition, giving a $J$-relative version of positively folded alcove~paths. 

\begin{defn}\label{defn:Jfold}
Let $\vec w=s_{i_1}s_{i_2}\cdots s_{i_{\ell}}\sigma$ be an expression for $w\in \Wext$ (not necessarily reduced) with $\sigma\in \Sigma$, and let~$v\in \WJ$. A \textit{$J$-folded alcove path of type~$\vec w$ starting at $v$} is a folded alcove path 
$p=(v_0,v_1,\ldots,v_{\ell},v_{\ell}\sigma)$, as in Definition~\ref{defn:posfolded}, with $v_0,\ldots,v_{\ell}\in \WJ$. A \textit{positively $J$-folded alcove path} is a $J$-folded alcove path $p=(v_0,v_1,\ldots,v_{\ell},v_{\ell}\sigma)$ such that:
\begin{center} if $v_{k-1}=v_k$ with $v_{k-1}s_{i_k}A_0\subseteq \cA_J$ then $v_{k-1}A_0\, {^+}|^-\,v_{k-1}s_{i_k}A_0$.
\end{center}
The \textit{end} of the $J$-folded alcove path $p=(v_0,\ldots,v_{\ell},v_{\ell}\sigma)$ is $\mathrm{end}(p)=v_{\ell}\sigma$. 
\end{defn}

Thus $J$-folded alcove paths are folded alcove paths confined to the fundamental $J$-alcove. However note that positively $J$-folded alcove paths are not necessarily positively folded alcove paths, because if $v_{k-1}=v_k$ with $v_{k-1}s_{i_k}A_0\not\subseteq \cA_J$ then there is no requirement that $v_{k-1}A_0\, {^+}|^-\,v_{k-1}s_{i_k}A_0$. See Example~\ref{ex:G2fold} for an illustration. Positively $\emptyset$-folded alcove paths are the same as positively folded alcove paths.

For $v\in \WJ$ let
$$
\mathcal{P}_J(\vec{w},v)=\{\text{all positively $J$-folded alcove paths of type $\vec{w}$ starting at $v$}\}.
$$ 

\goodbreak

\begin{defn}
Let $\vec w=s_{i_1}s_{i_2}\cdots s_{i_{\ell}}\sigma$ and let $p=(v_0,\ldots,v_{\ell},v_{\ell}\sigma)$ be a positively $J$-folded alcove path. The index~$k\in\{1,2,\ldots,\ell\}$ is:
\begin{compactenum}[$(1)$]
\item a \textit{positive (respectively, negative) $s_{i_k}$-crossing} if $v_k=v_{k-1}s_{i_k}$ and $v_{k}A_0$ is on the positive (respectively, negative) side of the hyperplane separating the alcoves $v_{k-1}A_0$ and $v_kA_0$;
\item a \textit{(positive) $s_{i_k}$-fold} if $v_k=v_{k-1}$ and $v_{k-1}s_{i_k}A_0\subseteq\cA_J$ (in which case $v_{k-1}A_0$ is necessarily on the positive side of the hyperplane separating $v_{k-1}A_0$ and $v_{k-1}s_{i_k}A_0$);
\item a \textit{positive (respectively, negative) bounce} if $v_k=v_{k-1}$ and $v_{k-1}s_{i_k}A_0\not\subseteq \cA_J$ and $v_{k-1}A_0$ is on the positive (respectively, negative) side of the hyperplane $H$ separating $v_{k-1}A_0$ and $v_{k-1}s_{i_k}A_0$. Necessarily $H=H_{\alpha,0}$ (respectively, $H=H_{\alpha,1}$) for some $\alpha\in\Phi_J^+$ and we say that $k$ \textit{occurs} on the hyperplane $H_{\alpha,0}$ (respectively, $H_{\alpha,1}$).
\end{compactenum}
\end{defn}

Less formally, these steps are denoted as follows (where $x=v_{k-1}$ and $s=s_{i_k}$). In each case, the alcoves contained in $\cA_J$ are shaded green. 
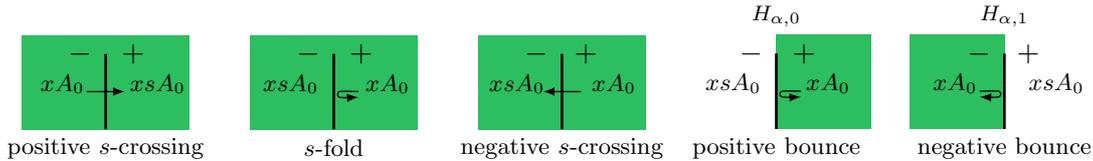
\begin{center}
\begin{tikzpicture}[xscale=0.5, yscale=0.5]
\path [fill=green2] (-8.2,-1)--(-3.8,-1)--(-3.8,1.5)--(-8.2,1.5);
\draw[line width=1pt] (-6,-1)--(-6,1);
\draw[line width=0.5pt,-latex](-6.5,0)--(-5.5,0);
\node at (-6.8,1) {{ $-$}};
\node at (-7.25,.2) {\footnotesize{ $xA_0$}};
\node at (-4.75,.2) {\footnotesize{ $xsA_0$}};
\node at (-5.4,1) {{ $+$}};
\node at (-6,-1.5) {\footnotesize{positive $s$-crossing}};
\path [fill=green2] (-2.2,-1)--(2.2,-1)--(2.2,1.5)--(-2.2,1.5);
\draw[line width=1pt] (-0,-1)--(-0,1);
\draw[line width=0.5pt,-latex] plot[smooth] coordinates {(.65,0)(0.15,0)(0.15,-.15)(.65,-.15)};
\node at (-.8,1){{ $-$}};
\node at (-1.25,.2){\footnotesize{ $xsA_0$}};
\node at (1.3,.2){\footnotesize{ $xA_0$}};
\node at (.6,1){{ $+$}};
\node at (0,-1.5){\footnotesize{$s$-fold}};
\path [fill=green2] (3.8,-1)--(8.2,-1)--(8.2,1.5)--(3.8,1.5);
\draw[line width=1pt] (6,-1)--(6,1);
\draw[line width=0.5pt,-latex] (6.5,0)--(5.5,0);
\node at (6.5,1){{ $+$}};
\node at (7.25,.2){\footnotesize{ $xA_0$}};
\node at (4.7,.2){\footnotesize{ $xsA_0$}};
\node at (5.2,1){{ $-$}};
\node at (-3,1.5){\scriptsize{\vphantom{$H_{\alpha,0}$}}};
\node at (6,-1.5){\footnotesize{negative $s$-crossing}};
\end{tikzpicture}
\begin{tikzpicture}[xscale=0.5, yscale=0.5]
\path [fill=green2] (0.5,-1)--(3,-1)--(3,1.5)--(0.5,1.5);
\path [fill=green2] (-0.5,-1)--(-3,-1)--(-3,1.5)--(-0.5,1.5);
\draw[line width=1pt] (3,-1)--(3,1);
\draw[line width=0.5pt,-latex] plot[smooth] coordinates {(2.35,0)(2.85,0)(2.85,-.15)(2.35,-.15)};
\node at (3.5,1){{ $+$}};
\node at (4.25,.2){\footnotesize{ $xsA_0$}};
\node at (1.7,.2){\footnotesize{$xA_0$}};
\node at (2.2,1){{ $-$}};
\node at (3,2){\scriptsize{$H_{\alpha,1}$}};
\node at (3,-1.5){\footnotesize{negative bounce}};
\draw[line width=1pt] (-3,-1)--(-3,1);
\draw[line width=0.5pt,-latex] plot[smooth] coordinates {(-2.35,0)(-2.85,0)(-2.85,-.15)(-2.35,-.15)};
\node at (-3.8,1){{ $-$}};
\node at (-4.25,.2){\footnotesize{ $xsA_0$}};
\node at (-1.75,.2){\footnotesize{ $xA_0$}};
\node at (-2.4,1){{ $+$}};
\node at (-3,2){\scriptsize{$H_{\alpha,0}$}};
\node at (-3,-1.5){\footnotesize{positive bounce}};
\end{tikzpicture}
\end{center}

It will turn out that bounces play a very different role in the theory to folds, and so we emphasise the distinction between these two concepts. Put briefly, all of the interactions a path makes with the walls of $\cA_J$ are bounces, and the folds can only occur in the ``interior'' of $\cA_J$ (that is, folds occur on the panel between two alcoves that both lie in $\cA_J$, while bounces occur on panels that lie on the boundary of~$\cA_J$). 

Let $p$ be a positively $J$-folded alcove path. For each $i\in \Iaff$ and $\alpha\in\Phi_J^+$ we write 
\begin{align*}
f_i(p)&=\#(\text{$i$-folds in $p$})\\
b_{\alpha}^+(p)&=\#(\text{positive bounces in $p$ occurring on $H_{\alpha,0}$})\\
b_{\alpha}^-(p)&=\#(\text{negative bounces in $p$ occurring on $H_{\alpha,1}$})\\
b_{\alpha}(p)&=b_{\alpha}^+(p)+b_{\alpha}^-(p).
\end{align*}

\begin{remark}
By Lemma~\ref{lem:boundingwalls} we have $b_{\alpha}^-(p)=0$ unless $\alpha=\varphi_K$ for some $K\in\cK(J)$, and $b_{\alpha}^+(p)=0$ unless $\alpha=\alpha_j$ or $\alpha=2\alpha_j$ for some $j\in J$. Of course the case $\alpha=2\alpha_j$ only occurs if $\Phi_J$ is not reduced and $j=n$, and in this case $b_{\alpha_n}^+(p)=b_{2\alpha_n}^+(p)$.
\end{remark}

\subsection{Coweights and final directions of positively $J$-folded alcove paths}

By Corollary~\ref{cor:Tiso} the set $W^J$ is a fundamental domain for the action of $\TJ$ on $\WJ$. While in many respects this choice of fundamental domain is ``natural'', in examples and applications it turns out to be important to have additional flexibility in order to better incorporate symmetries present (see, for example, Section~\ref{sec:An}).

Let $\sF$ be a fundamental domain for the action of $\TJ$ on $\WJ$. Thus each $w\in\WJ$ has a unique expression as 
$$
w=\st_{\lambda}u\quad\text{with $\lambda\in P^{(J)}$ and $u\in\sF$},
$$
and we define the \textit{coweight of $w$ relative to $\sF$} and the \textit{final direction of $w$ relative to $\sF$} by
$$
\wt(w,\sF)=\lambda\quad\text{and}\quad \theta(w,\sF)=u.
$$
If $p$ is a positively $J$-folded alcove path then the \textit{coweight of $p$ relative to $\sF$} and the \textit{final direction of $p$ relative to $\sF$} are
\begin{align}\label{eq:generalweights}
\wt(p,\sF)=\wt(\mathrm{end}(p),\sF)\quad\text{and}\quad \theta(p,\sF)=\theta(\mathrm{end}(p),\sF).
\end{align}
In particular, note that $\wt(p,W^J)=\wt(p)$ and $\theta(p,W^J)=\theta^J(p)$, where $\wt(p)=\wt(\mathrm{end}(p))$ and $\theta^J(p)=\theta^J(\mathrm{end}(p))$ (see Definition~\ref{defn:theta}).

\begin{example}\label{ex:G2fold}
Figure~\ref{fig:G2example} illustrates a positively $J$-folded alcove path in type $\sG_2$ with $J=\{1\}$ with two choices of fundamental domain $\sF$ shaded blue (see Figure~\ref{fig:G2TJ} for the root system conventions). In both cases the path has $3$ bounces (two positive, and one negative, occurring at the fourth, fourteenth, and twenty first steps), and 2 folds (a $2$-fold and a $0$-fold, occurring at the eighth and twenty fourth steps). In Figure~\ref{fig:G2example}(a) we have $\theta^J(p)=21$ and $\wt(p,W^J)=\wt(p)=\omega_1$. In Figure~\ref{fig:G2example}(b) we have $\theta(p,\sF)=020$ and $\wt(p,\sF)=\omega_2$.

\begin{figure}[H]
\centering
\begin{subfigure}{.495\textwidth}
\begin{center}
\begin{tikzpicture}[scale=1.3]
\clip ({-4*0.866+0.433},{-2}) rectangle + ({6*0.866},{5.5});
\path [fill=green2] ({-4*0.866},-3)--({-2*0.866},-3)--({3*0.866},4.5)--({1*0.866},4.5)--({-4*0.866},-3);
\path [fill=blue2] ({-1*0.433},-0.75)--({-2*0.433},-0.5)--({-2*0.433},0.5)--({0*0.433},1)--({1*0.433},0.75)--(-1*0.433,-0.75);
\node at (0.15,0.6) {$e$};
\draw(-4.33,4.5)--(4.33,4.5);
\draw(-4.33,3)--(4.33,3);
\draw(-4.33,1.5)--(4.33,1.5);
\draw(-4.33,0)--(4.33,0);
\draw(-4.33,-1.5)--(4.33,-1.5);
\draw(-4.33,-3)--(4.33,-3);
\draw(-4.33,-3)--(-4.33,4.5);
\draw(-3.464,-3)--(-3.464,4.5);
\draw(-2.598,-3)--(-2.598,4.5);
\draw(-1.732,-3)--(-1.732,4.5);
\draw(-.866,-3)--(-.866,4.5);
\draw(0,-3)--(0,4.5);
\draw(.866,-3)--(.866,4.5);
\draw(1.732,-3)--(1.732,4.5);
\draw(2.598,-3)--(2.598,4.5);
\draw(3.464,-3)--(3.464,4.5);
\draw(4.33,-3)--(4.33,4.5);
\draw(-4.33,3.5)--({-3*0.866},4.5);
\draw(-4.33,2.5)--({-1*0.866},4.5);
\draw(-4.33,1.5)--({1*0.866},4.5);
\draw(-4.33,.5)--({3*0.866},4.5);
\draw(-4.33,-.5)--(4.33,4.5);
\draw(-4.33,-1.5)--(4.33,3.5);
\draw(-4.33,-2.5)--(4.33,2.5);
\draw(-3.464,-3)--(4.33,1.5);
\draw(-1.732,-3)--(4.33,.5);
\draw(0,-3)--(4.33,-.5);
\draw(1.732,-3)--(4.33,-1.5);
\draw(3.464,-3)--(4.33,-2.5);
\draw(4.33,3.5)--({3*0.866},4.5);
\draw(4.33,2.5)--({1*0.866},4.5);
\draw(4.33,1.5)--({-1*0.866},4.5);
\draw(4.33,.5)--({-3*0.866},4.5);
\draw(4.33,-.5)--(-4.33,4.5);
\draw(4.33,-1.5)--(-4.33,3.5);
\draw(4.33,-2.5)--(-4.33,2.5);
\draw(3.464,-3)--(-4.33,1.5);
\draw(1.732,-3)--(-4.33,.5);
\draw(0,-3)--(-4.33,-.5);
\draw(-1.732,-3)--(-4.33,-1.5);
\draw(-3.464,-3)--(-4.33,-2.5);
\draw(-4.33,-1.5)--(-3.464,-3);
\draw(-4.33,1.5)--(-1.732,-3);
\draw(-4.33,4.5)--(0,-3);
\draw({-3*0.866},4.5)--(1.732,-3);
\draw({-1*0.866},4.5)--(3.464,-3);
\draw({1*0.866},4.5)--(4.33,-1.5);
\draw({3*0.866},4.5)--(4.33,1.5);
\draw(4.33,-1.5)--(3.464,-3);
\draw(4.33,1.5)--(1.732,-3);
\draw(4.33,4.5)--(0,-3);
\draw({3*0.866},4.5)--(-1.732,-3);
\draw({1*0.866},4.5)--(-3.464,-3);
\draw({-1*0.866},4.5)--(-4.33,-1.5);
\draw({-3*0.866},4.5)--(-4.33,1.5);
\draw[line width=2pt]({-5*0.866},1.5)--({-3*0.866},4.5);
\draw[line width=2pt]({-5*0.866},-1.5)--({-1*0.866},4.5);
\draw[line width=2pt]({-4*0.866},-3)--({1*0.866},4.5);
\draw[line width=2pt]({-2*0.866},-3)--({3*0.866},4.5);
\draw[line width=2pt]({0*0.866},-3)--({5*0.866},4.5);
\draw[line width=2pt]({2*0.866},-3)--({5*0.866},1.5);
\draw[line width=2pt]({4*0.866},-3)--({5*0.866},-1.5);
\draw[line width=1pt,-latex](-0.47,0.45)--(-0.68,0.15);
\draw[line width=1pt,-latex](-0.68,0.15)--(-0.68,-0.2);
\draw[line width=1pt,-latex](-0.68,-0.2)--(-0.47,-0.45);
\draw[line width=1pt,-latex] plot[smooth] coordinates {(-0.47,-0.45) (-0.38,-0.52) (-0.42,-0.6) (-0.55,-0.54)};
\draw[line width=1pt,-latex](-0.55,-0.54)--(-0.75,-0.9);
\draw[line width=1pt,-latex](-0.75,-0.9)--(-1.02,-0.9);
\draw[line width=1pt,-latex](-1.02,-0.9)--(-1.3,-1.1);
\draw[line width=1pt,-latex] plot[smooth] coordinates {(-1.3,-1.1) (-1.36,-1.17) (-1.45,-1.12) (-1.35,-1)};
\draw[line width=1pt,-latex](-1.35,-1)--(-1.02,-0.78);
\draw[line width=1pt,-latex](-1.02,-0.78)--(-1.25,-0.45);
\draw[line width=1pt,-latex](-1.25,-0.45)--(-1.05,-0.15);
\draw[line width=1pt,-latex](-1.05,-0.15)--(-1.05,0.15);
\draw[line width=1pt,-latex](-1.05,0.15)--(-1.3,0.4);
\draw[line width=1pt,-latex] plot[smooth] coordinates {(-1.3,0.4) (-1.4,0.45) (-1.36,0.55) (-1.2,0.45)};
\draw[line width=1pt,-latex](-1.2,0.45)--(-1,0.8);
\draw[line width=1pt,-latex](-1,0.8)--(-0.7,0.8);
\draw[line width=1pt,-latex](-0.7,0.8)--(-0.35,1);
\draw[line width=1pt,-latex](-0.35,1)--(-0.15,1.3);
\draw[line width=1pt,-latex](-0.15,1.3)--(0.22,1.3);
\draw[line width=1pt,-latex](0.22,1.3)--(0.38,1.02);
\draw[line width=1pt,-latex] plot[smooth] coordinates {(0.38,1.02)(0.48,0.96)(0.55,1.05)(0.42,1.125)};
\draw[line width=1pt,-latex](0.44,1.12)--(0.3,1.35);
\draw[line width=1pt,-latex](0.3,1.35)--(0.3,1.65);
\draw[line width=1pt,-latex] plot[smooth] coordinates {(0.3,1.65)(0.05,1.65)(0.05,1.72)(0.28,1.72)};
\draw[line width=1pt,-latex](0.28,1.72)--(0.48,1.91);
\draw[line width=1pt,-latex](0.48,1.91)--(0.75,2.1);
\draw[line width=1pt,-latex](0.75,2.1)--(0.5,2.5);
\draw[line width=1pt,-latex](0.5,2.5)--(0.68,2.79);
\node at ({-4*0.866},{2*1.5}) {\Large{$\bullet$}};
\node at ({-2*0.866},{2*1.5}) {\Large{$\bullet$}};
\node at ({0*0.866},{2*1.5}) {\Large{$\bullet$}};
\node at ({2*0.866},{2*1.5}) {\Large{$\bullet$}};
\node at ({4*0.866},{2*1.5}) {\Large{$\bullet$}};
\node at ({-3*0.866},{1.5}) {\Large{$\bullet$}};
\node at ({-1*0.866},{1.5}) {\Large{$\bullet$}};
\node at ({1*0.866},{1.5}) {\Large{$\bullet$}};
\node at ({3*0.866},{1.5}) {\Large{$\bullet$}};
\node at ({-4*0.866},0) {\Large{$\bullet$}};
\node at ({-2*0.866},0) {\Large{$\bullet$}};
\node at ({0*0.866},0) {\Large{$\bullet$}};
\node at ({2*0.866},0) {\Large{$\bullet$}};
\node at ({4*0.866},0) {\Large{$\bullet$}};
\node at ({-3*0.866},{-1*1.5}) {\Large{$\bullet$}};
\node at ({-1*0.866},{-1*1.5}) {\Large{$\bullet$}};
\node at ({1*0.866},{-1*1.5}) {\Large{$\bullet$}};
\node at ({3*0.866},{-1*1.5}) {\Large{$\bullet$}};
\draw [line width=2pt,dotted] (0.433,0.75)--(0,1)--(-2*0.433,0.5)--(-2*0.433,-0.5)--(-0.433,-0.75);
\draw [line width=2pt,dotted] (-2*0.433,-0.5)--(-4*0.433,-1)--(-4*0.433,-2);
\draw [line width=2pt,dotted] (0,1)--(0,2)--(2*0.433,2.5)--(3*0.433,2.25);
\draw [line width=2pt,dotted] (2*0.433,2.5)--(2*0.433,3.5);
\draw [line width=2pt,dotted] (0,2)--(-0.433,2.25);
\draw [line width=2pt,dotted] (0-2*0.433,2-1.5)--(-0.433-2*0.433,2.25-1.5);
\draw [line width=2pt,dotted] (0-4*0.433,2-2*1.5)--(-0.433-4*0.433,2.25-2*1.5);
\end{tikzpicture}
\caption{Fundamental domain $\mathsf{F}=W^J$}
\end{center}
\end{subfigure}
\begin{subfigure}{.495\textwidth}
\begin{center}
\begin{tikzpicture}[scale=1.3]
\clip ({-4*0.866+0.433},{-2}) rectangle + ({6*0.866},{5.5});
\path [fill=green2] ({-4*0.866},-3)--({-2*0.866},-3)--({3*0.866},4.5)--({1*0.866},4.5)--({-4*0.866},-3);
\path [fill=blue2] (0,0)--(2*0.433,1.5)--(0,2)--(0,1.5)--(-2*0.433,1.5)--(0,1)--(-0.433,0.75)--(0,0);
\node at (0.15,0.6) {$e$};
\draw(-4.33,4.5)--(4.33,4.5);
\draw(-4.33,3)--(4.33,3);
\draw(-4.33,1.5)--(4.33,1.5);
\draw(-4.33,0)--(4.33,0);
\draw(-4.33,-1.5)--(4.33,-1.5);
\draw(-4.33,-3)--(4.33,-3);
\draw(-4.33,-3)--(-4.33,4.5);
\draw(-3.464,-3)--(-3.464,4.5);
\draw(-2.598,-3)--(-2.598,4.5);
\draw(-1.732,-3)--(-1.732,4.5);
\draw(-.866,-3)--(-.866,4.5);
\draw(0,-3)--(0,4.5);
\draw(.866,-3)--(.866,4.5);
\draw(1.732,-3)--(1.732,4.5);
\draw(2.598,-3)--(2.598,4.5);
\draw(3.464,-3)--(3.464,4.5);
\draw(4.33,-3)--(4.33,4.5);
\draw(-4.33,3.5)--({-3*0.866},4.5);
\draw(-4.33,2.5)--({-1*0.866},4.5);
\draw(-4.33,1.5)--({1*0.866},4.5);
\draw(-4.33,.5)--({3*0.866},4.5);
\draw(-4.33,-.5)--(4.33,4.5);
\draw(-4.33,-1.5)--(4.33,3.5);
\draw(-4.33,-2.5)--(4.33,2.5);
\draw(-3.464,-3)--(4.33,1.5);
\draw(-1.732,-3)--(4.33,.5);
\draw(0,-3)--(4.33,-.5);
\draw(1.732,-3)--(4.33,-1.5);
\draw(3.464,-3)--(4.33,-2.5);
\draw(4.33,3.5)--({3*0.866},4.5);
\draw(4.33,2.5)--({1*0.866},4.5);
\draw(4.33,1.5)--({-1*0.866},4.5);
\draw(4.33,.5)--({-3*0.866},4.5);
\draw(4.33,-.5)--(-4.33,4.5);
\draw(4.33,-1.5)--(-4.33,3.5);
\draw(4.33,-2.5)--(-4.33,2.5);
\draw(3.464,-3)--(-4.33,1.5);
\draw(1.732,-3)--(-4.33,.5);
\draw(0,-3)--(-4.33,-.5);
\draw(-1.732,-3)--(-4.33,-1.5);
\draw(-3.464,-3)--(-4.33,-2.5);
\draw(-4.33,-1.5)--(-3.464,-3);
\draw(-4.33,1.5)--(-1.732,-3);
\draw(-4.33,4.5)--(0,-3);
\draw({-3*0.866},4.5)--(1.732,-3);
\draw({-1*0.866},4.5)--(3.464,-3);
\draw({1*0.866},4.5)--(4.33,-1.5);
\draw({3*0.866},4.5)--(4.33,1.5);
\draw(4.33,-1.5)--(3.464,-3);
\draw(4.33,1.5)--(1.732,-3);
\draw(4.33,4.5)--(0,-3);
\draw({3*0.866},4.5)--(-1.732,-3);
\draw({1*0.866},4.5)--(-3.464,-3);
\draw({-1*0.866},4.5)--(-4.33,-1.5);
\draw({-3*0.866},4.5)--(-4.33,1.5);
\draw[line width=2pt]({-5*0.866},1.5)--({-3*0.866},4.5);
\draw[line width=2pt]({-5*0.866},-1.5)--({-1*0.866},4.5);
\draw[line width=2pt]({-4*0.866},-3)--({1*0.866},4.5);
\draw[line width=2pt]({-2*0.866},-3)--({3*0.866},4.5);
\draw[line width=2pt]({0*0.866},-3)--({5*0.866},4.5);
\draw[line width=2pt]({2*0.866},-3)--({5*0.866},1.5);
\draw[line width=2pt]({4*0.866},-3)--({5*0.866},-1.5);
\draw[line width=1pt,-latex](-0.47,0.45)--(-0.68,0.15);
\draw[line width=1pt,-latex](-0.68,0.15)--(-0.68,-0.2);
\draw[line width=1pt,-latex](-0.68,-0.2)--(-0.47,-0.45);
\draw[line width=1pt,-latex] plot[smooth] coordinates {(-0.47,-0.45) (-0.38,-0.52) (-0.42,-0.6) (-0.55,-0.54)};
\draw[line width=1pt,-latex](-0.55,-0.54)--(-0.75,-0.9);
\draw[line width=1pt,-latex](-0.75,-0.9)--(-1.02,-0.9);
\draw[line width=1pt,-latex](-1.02,-0.9)--(-1.3,-1.1);
\draw[line width=1pt,-latex] plot[smooth] coordinates {(-1.3,-1.1) (-1.36,-1.17) (-1.45,-1.12) (-1.35,-1)};
\draw[line width=1pt,-latex](-1.35,-1)--(-1.02,-0.78);
\draw[line width=1pt,-latex](-1.02,-0.78)--(-1.25,-0.45);
\draw[line width=1pt,-latex](-1.25,-0.45)--(-1.05,-0.15);
\draw[line width=1pt,-latex](-1.05,-0.15)--(-1.05,0.15);
\draw[line width=1pt,-latex](-1.05,0.15)--(-1.3,0.4);
\draw[line width=1pt,-latex] plot[smooth] coordinates {(-1.3,0.4) (-1.4,0.45) (-1.36,0.55) (-1.2,0.45)};
\draw[line width=1pt,-latex](-1.2,0.45)--(-1,0.8);
\draw[line width=1pt,-latex](-1,0.8)--(-0.7,0.8);
\draw[line width=1pt,-latex](-0.7,0.8)--(-0.35,1);
\draw[line width=1pt,-latex](-0.35,1)--(-0.15,1.3);
\draw[line width=1pt,-latex](-0.15,1.3)--(0.22,1.3);
\draw[line width=1pt,-latex](0.22,1.3)--(0.38,1.02);
\draw[line width=1pt,-latex] plot[smooth] coordinates {(0.38,1.02)(0.48,0.96)(0.55,1.05)(0.42,1.125)};
\draw[line width=1pt,-latex](0.44,1.12)--(0.3,1.35);
\draw[line width=1pt,-latex](0.3,1.35)--(0.3,1.65);
\draw[line width=1pt,-latex] plot[smooth] coordinates {(0.3,1.65)(0.05,1.65)(0.05,1.72)(0.28,1.72)};
\draw[line width=1pt,-latex](0.28,1.72)--(0.48,1.91);
\draw[line width=1pt,-latex](0.48,1.91)--(0.75,2.1);
\draw[line width=1pt,-latex](0.75,2.1)--(0.5,2.5);
\draw[line width=1pt,-latex](0.5,2.5)--(0.68,2.79);
\node at ({-4*0.866},{2*1.5}) {\Large{$\bullet$}};
\node at ({-2*0.866},{2*1.5}) {\Large{$\bullet$}};
\node at ({0*0.866},{2*1.5}) {\Large{$\bullet$}};
\node at ({2*0.866},{2*1.5}) {\Large{$\bullet$}};
\node at ({4*0.866},{2*1.5}) {\Large{$\bullet$}};
\node at ({-3*0.866},{1.5}) {\Large{$\bullet$}};
\node at ({-1*0.866},{1.5}) {\Large{$\bullet$}};
\node at ({1*0.866},{1.5}) {\Large{$\bullet$}};
\node at ({3*0.866},{1.5}) {\Large{$\bullet$}};
\node at ({-4*0.866},0) {\Large{$\bullet$}};
\node at ({-2*0.866},0) {\Large{$\bullet$}};
\node at ({0*0.866},0) {\Large{$\bullet$}};
\node at ({2*0.866},0) {\Large{$\bullet$}};
\node at ({4*0.866},0) {\Large{$\bullet$}};
\node at ({-3*0.866},{-1*1.5}) {\Large{$\bullet$}};
\node at ({-1*0.866},{-1*1.5}) {\Large{$\bullet$}};
\node at ({1*0.866},{-1*1.5}) {\Large{$\bullet$}};
\node at ({3*0.866},{-1*1.5}) {\Large{$\bullet$}};
\draw [line width=2pt,dotted] (-6*0.433,-1.5)--(-4*0.433,-1.5)--(-4*0.433,-1)--(-2*0.433,-1.5);
\draw [line width=2pt,dotted] (-4*0.433,0)--(-2*0.433,0)--(-2*0.433,0.5)--(0*0.433,0);
\draw [line width=2pt,dotted] (-2*0.433,1.5)--(0*0.433,1.5)--(0*0.433,2)--(2*0.433,1.5);
\draw [line width=2pt,dotted] (0*0.433,3)--(2*0.433,3)--(2*0.433,3.5)--(4*0.433,3);
\draw [line width=2pt,dotted] (-6*0.433,-1.5)--(-4*0.433,-2);
\draw [line width=2pt,dotted] (-4*0.433,0)--(-2*0.433,-0.5)--(-3*0.433,-0.75)--(-2*0.433,-1.5);
\draw [line width=2pt,dotted] (-2*0.433,1.5)--(0*0.433,1)--(-1*0.433,0.75)--(0*0.433,0);
\draw [line width=2pt,dotted] (0*0.433,3)--(2*0.433,2.5)--(1*0.433,2.25)--(2*0.433,1.5);
\draw [line width=2pt,dotted] (4*0.433,3)--(3*0.433,3.75);
\end{tikzpicture}
\caption{Fundamental domain $\mathsf{F}=\{e,2,020,021,02,0\}$}
\end{center}
\end{subfigure}
\caption{Positively $J$-folded alcove path for $\Phi=\sG_2$ with $J=\{1\}$}\label{fig:G2example}
\end{figure}
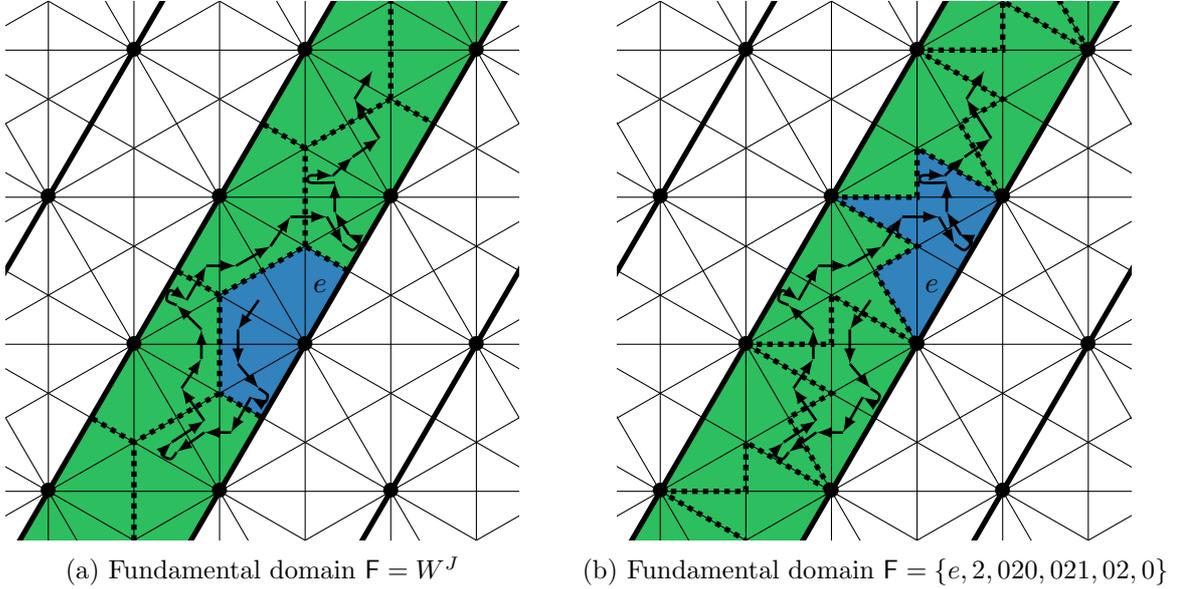
\end{example}

\goodbreak

There is a natural action of $\TJ$ on positively $J$-folded alcove paths, as follows (recall the definition of $\lambda^{(J)}$ from Definition~\ref{defn:projections}).

\begin{lemma}\label{lem:actiononpaths}
If $p=(v_0,v_1,\ldots,v_{\ell},v_{\ell}\sigma)$ is a positively $J$-folded alcove path of type $\vec{w}=s_{i_1}\cdots s_{i_{\ell}}\sigma$ and $\lambda\in  P^{(J)}$ then 
$$
\st_{\lambda}\cdot p=(\st_{\lambda}v_0,\st_{\lambda}v_1,\ldots,\st_{\lambda}v_{\ell},\st_{\lambda}v_{\ell}\sigma)
$$
is a positively $J$-folded alcove path of type $\vec{w}$, and folds are mapped to folds and bounces to bounces. For any fundamental domain $\sF$ we have
$$
\wt(\st_{\lambda}\cdot p,\sF)=(\lambda+\wt(p,\sF))^{(J)}\quad\text{and}\quad \theta(\st_{\lambda}\cdot p,\sF)=\theta(p,\sF).
$$
\end{lemma}

\begin{proof}
Since $\TJ$ acts on $\WJ$ we have $\st_{\lambda}v_0,\ldots,\st_{\lambda}v_{\ell},\st_{\lambda}v_{\ell}\sigma\in \WJ$, and moreover bounces are mapped to bounces, and folds to folds (because if $v_{k}=v_{k-1}$ then $v_{k-1}s_{i_k}\in\WJ$ if and only if $\st_{\lambda}v_{k-1}s_{i_k}\in\WJ$). It remains to show that the image of a fold remains positively oriented. Thus, suppose that $v_{k-1}=v_k$ and $v_{k-1}s_{i_k}\in\WJ$. Then $v_{k-1}A_0$ is on the positive side of the hyperplane separating the alcoves $v_{k-1}A_0$ and $v_{k-1}s_{i_k}A_0$, and this hyperplane has a linear root $\alpha\in\Phi^+$ with $\alpha\notin\Phi_J$. We have $\st_{\lambda}=t_{\lambda}\sy_{\lambda}$, and since $\sy_{\lambda}\in W_J$ we have $\sy_{\lambda}\alpha\in\Phi^+$, and hence the fold remains positively oriented. The statements about the coweight and final direction are then clear using Lemma~\ref{lem:utbasics}. 
\end{proof}

\subsection{$J$-parameter systems and the $\sv$-mass of a path}


Recall the definition of the parameters $\sq_i$, $i\in \{0\}\cup\Ifin$ from Section~\ref{sec:hecke} and Convention~\ref{conv:parameters}. The modules that we construct in Section~\ref{sec:moduleconstruction} will be indexed by two pieces of data: a subset $J\subseteq I$, and a choice of a \textit{$J$-parameter system} $\sv$ as defined below. Roughly speaking (and restricting temporarily to the reduced case), a $J$-parameter system amounts to choosing elements $\sv_{\alpha_j}\in \{\sq_j,-\sq_j^{-1}\}$ for each $j\in J$ in such a way that if $j,j'\in J$ with $s_{j'}$ and $s_j$ conjugate in $W_J$, then $\sv_{\alpha_j}=\sv_{\alpha_{j'}}$. The formal definition is as follows.

\begin{defn}\label{defn:Jparameters}
A \textit{$J$-parameter system} is a family $\sv=(\sv_{\alpha})_{\alpha\in\Phi_J}$ with:
\begin{compactenum}[$(1)$]
\item $\sv_{\alpha}=\sv_{\beta}$ if $\beta\in W_J\alpha$;
\item $\sv_{\alpha_j}\in\{\sq_j,-\sq_j^{-1}\}$ for $j\in J$ with $2\alpha_j\notin\Phi_J$;
\item $\sv_{2\alpha_n}\in\{\sq_0,-\sq_0^{-1}\}$ and $\sv_{\alpha_n}\sv_{2\alpha_n}\in\{\sq_n,-\sq_n^{-1}\}$ if $\Phi_J$ is not reduced.
\end{compactenum}
By convention we set $\sv_{\alpha}=1$ if $\alpha\notin\Phi_J$. Thus we have $\sv_{\alpha_j}\sv_{2\alpha_j}\in\{\sq_j,-\sq_j^{-1}\}$ for all $j\in J$.
\end{defn}

\begin{example} If $\Phi=\sC_{13}$ and $J=\{1,2,3,6,7,9,11,12,13\}$ as depicted below
\begin{center}
 \begin{tikzpicture}[scale=0.5,baseline=-0.5ex]
\node at (0,0.5) {};
\node [inner sep=0.8pt,outer sep=0.8pt] at (-7,0) (-7) {$\bullet$};
\node [inner sep=0.8pt,outer sep=0.8pt] at (-6,0) (-6) {$\bullet$};
\node [inner sep=0.8pt,outer sep=0.8pt] at (-5,0) (-5) {$\bullet$};
\node [inner sep=0.8pt,outer sep=0.8pt] at (-4,0) (-4) {$\bullet$};
\node [inner sep=0.8pt,outer sep=0.8pt] at (-3,0) (-3) {$\bullet$};
\node [inner sep=0.8pt,outer sep=0.8pt] at (-2,0) (-2) {$\bullet$};
\node [inner sep=0.8pt,outer sep=0.8pt] at (-1,0) (-1) {$\bullet$};
\node [inner sep=0.8pt,outer sep=0.8pt] at (0,0) (0) {$\bullet$};
\node [inner sep=0.8pt,outer sep=0.8pt] at (1,0) (1) {$\bullet$};
\node [inner sep=0.8pt,outer sep=0.8pt] at (2,0) (2) {$\bullet$};
\node [inner sep=0.8pt,outer sep=0.8pt] at (3,0) (3) {$\bullet$};
\node [inner sep=0.8pt,outer sep=0.8pt] at (4,0) (4) {$\bullet$};
\node [inner sep=0.8pt,outer sep=0.8pt] at (5,0) (5) {$\bullet$};
\draw (-7,0)--(4,0);
\draw (4,0.1)--(5,0.1);
\draw (4,-0.1)--(5,-0.1);
\draw (4.5+0.15,0.3) -- (4.5-0.08,0) -- (4.5+0.15,-0.3);
\draw [line width=0.5pt,line cap=round,rounded corners] (-7.north west)  rectangle (-5.south east);
\draw [line width=0.5pt,line cap=round,rounded corners] (-2.north west)  rectangle (-1.south east);
\draw [line width=0.5pt,line cap=round,rounded corners] (1.north west)  rectangle (1.south east);
\draw [line width=0.5pt,line cap=round,rounded corners] (3.north west)  rectangle (5.south east);
\end{tikzpicture}
\end{center}
then there are $32$ distinct $J$-parameter systems. Specifically, these systems are given by the independent choices $\sv_{\alpha_1},\sv_{\alpha_6},\sv_{\alpha_9},\sv_{\alpha_{11}}\in\{\sq_1,-\sq_1^{-1}\}$ and $\sv_{\alpha_{13}}\in\{\sq_{13},-\sq_{13}^{-1}\}$ (and then $\sv_{\alpha_j}$ for $j=2,3,7,12$ are determined by condition~(1) in the definition of $J$-parameter systems). If, instead, $\Phi=\sBC_{13}$ then there are $64$ distinct $J$-parameter systems due to the additional freedom in choosing $\sv_{2\alpha_{13}}\in \{\sq_0,-\sq_0^{-1}\}$. 
\end{example}

If $\sv=(\sv_{\alpha})_{\alpha\in\Phi_J}$ is a $J$-parameter system, for $\lambda\in P$ and $y\in W_J$ we write
\begin{align}\label{eq:vlambda}
\sv^{\lambda}=\prod_{\alpha\in\Phi_J^+}\sv_{\alpha}^{\langle\lambda,\alpha\rangle}\quad\text{and}\quad \sv(y)=\prod_{\alpha\in\Phi(y)}\sv_{\alpha}\sv_{2\alpha}.
\end{align}
Note that if $\lambda\in P\cap V^J$ then $\sv^{\lambda}=1$ (as $\langle\lambda,\alpha\rangle=0$ for all $\alpha\in\Phi_J^+$).

In the following definition we introduce the \textit{$\sv$-mass} of a positively $J$-folded alcove path. This quantity will play an important role in the combinatorial formula of Theorem~\ref{thm:mainpath1}. 

\begin{defn}
Let $\sv=(\sv_{\alpha})_{\alpha\in\Phi_J}$ be a $J$-parameter system. The \textit{$\sv$-mass} of a positively $J$-folded alcove path~$p$ is
\begin{align*}
\cQ_{J,\sv}(p)=\bigg[\prod_{\alpha\in \Phi_J^+}\sv_{\alpha}^{b_{\alpha}(p)}\bigg]\bigg[\prod_{i\in \Iaff}(\sq_i-\sq_i^{-1})^{f_i(p)}\bigg].
\end{align*}
\end{defn}

\begin{example}
The path in Figure~\ref{fig:G2example} has $\sv$-mass 
$$
\cQ_{J,\sv}(p)=\sv_{\alpha_1}^3(\sq_2-\sq_2^{-1})(\sq_0-\sq_0^{-1}),
$$
where $\sv_{\alpha_1}\in\{\sq_1,-\sq_1^{-1}\}$. A positively $J$-folded path with $\Phi=\sBC_2$ and $J=\{2\}$ is illustrated in Figure~\ref{fig:BC2}. This path has $\sv$-mass
$
\cQ_{J,\sv}(p)=\sv_{\alpha_2}^3\sv_{2\alpha_2}^5(\sq_1-\sq_1^{-1})=(\sv_{\alpha_2}\sv_{2\alpha_2})^3\sv_{2\alpha_2}^2(\sq_1-\sq_1^{-1}).
$
We have $\sv_{\alpha_2}\sv_{2\alpha_2}\in\{\sq_2,-\sq_2^{-1}\}$ and $\sv_{2\alpha_n}\in\{\sq_0,-\sq_0^{-1}\}$, and note that the exponents $3$ and $2$ in the expression for $\cQ_{J,\sv}(p)$ count the number of bounces on the walls $H_{\alpha_2,0}$ and $H_{2\alpha_2,1}$. 
\begin{figure}[H]
\begin{center}
\begin{tikzpicture}[scale=1.2]
\clip ({-3.5},{-1.5}) rectangle + ({6},{4});
\path [fill=green2] (-4,0)--(-4,1)--(4,1)--(4,0)--(-4,0);
\path [fill=black!60] (0-2,0) -- (1-2,0) -- (1-2,1) -- (0-2,0);
\draw (-4,-4)--(4,-4);
\draw (-4,-3)--(4.5,-3);
\draw (-4,-2)--(4.5,-2);
\draw (-4,-1)--(4.5,-1);
\draw (-4,0)--(4.5,0);
\draw (-4,1)--(4.5,1);
\draw (-4,2)--(4.5,2);
\draw (-4,3)--(4.5,3);
\draw (-4,4)--(4,4);
\draw (-4,-4)--(-4,4);
\draw (-3,-4)--(-3,4.5);
\draw (-2,-4)--(-2,4.5);
\draw (-1,-4)--(-1,4.5);
\draw (0,-4)--(0,4.5);
\draw (1,-4)--(1,4.5);
\draw (2,-4)--(2,4.5);
\draw (3,-4)--(3,4.5);
\draw (4,-4)--(4,4);
\draw (-4,-4)--(4,4);
\draw (-4.6,-2.25)--(-4.25,-2.25);
\draw (-4.25,-2.25)--(2,4);
\draw (-4.6,-0.25)--(-4.25,-0.25);
\draw (-4.25,-0.25)--(0,4);
\draw (-4.6,1.75)--(-4.25,1.75);
\draw (-4.25,1.75)--(-2,4);
\draw (-2,-4)--(4,2);
\draw (0,-4)--(4,0);
\draw (2,-4)--(4,-2);
\draw (-4,-2)--(-1.75,-4.25);
\draw (-1.75,-4.25)--(-1.75,-4.6);
\draw (-4,0)--(0.25,-4.25);
\draw (0.25,-4.25)--(0.25,-4.6);
\draw (-4,2)--(2.25,-4.25);
\draw (2.25,-4.25)--(2.25,-4.6);
\draw (-4,4)--(4,-4);
\draw (-2,4)--(4,-2);
\draw (0,4)--(4,0);
\draw (2,4)--(4,2);
\draw[line width=1pt,-latex](-0.225-2,0.55)--(-0.6-2,0.2);
\draw[line width=1pt,-latex] plot[smooth] coordinates {(-0.6-2,0.2)(-0.6-2,0.05)(-0.7-2,0.05)(-0.7-2,0.35)};
\draw[line width=1pt,-latex](-0.7-2,0.35)--(-0.25-2,0.75);
\draw[line width=1pt,-latex] plot[smooth] coordinates {(-0.25-2,0.75)(-0.25-2,0.95)(-0.15-2,0.95)(-0.15-2,0.65)};
\draw[line width=1pt,-latex](-0.15-2,0.65)--(0.25-2,0.65);
\draw[line width=1pt,-latex](0.25-2,0.65)--(0.65-2,0.2);
\draw[line width=1pt,-latex] plot[smooth] coordinates {(0.65-2,0.2)(0.65-2,0.05)(0.75-2,0.05)(0.75-2,0.35)};
\draw[line width=1pt,-latex] plot[smooth] coordinates {(0.75-2,0.35)(0.55-2,0.5)(0.65-2,0.6)(0.85-2,0.4)};
\draw[line width=1pt,-latex](0.85-2,0.4)--(1.3-2,0.4);
\draw[line width=1pt,-latex](1.3-2,0.4)--(1.75-2,0.7);
\draw[line width=1pt,-latex](1.75-2,0.7)--(2.25-2,0.7);
\draw[line width=1pt,-latex] plot[smooth] coordinates {(2.25-2,0.7)(2.25-2,0.95)(2.35-2,0.95)(2.35-2,0.7)};
\draw[line width=1pt,-latex](2.35-2,0.7)--(2.7-2,0.25);
\draw[line width=1pt,-latex] plot[smooth] coordinates {(2.7-2,0.25)(2.7-2,0.05)(2.8-2,0.05)(2.8-2,0.4)};
\draw[line width=1pt,-latex](2.8-2,0.4)--(3.3-2,0.4);
\draw[line width=1pt,-latex](3.3-2,0.4)--(3.7-2,0.7);
\end{tikzpicture}
\caption{Positively $J$-folded path with $\Phi=\sBC_2$ and $J=\{2\}$}
\label{fig:BC2}
\end{center}
\end{figure}
\end{example}

Recall, from Lemma~\ref{lem:actiononpaths}, that $\TJ$ acts on the set of positively $J$-folded alcove paths. 

\begin{lemma}\label{lem:preserveQ}
If $p$ is a positively $J$-folded alcove path and $\lambda\in P^{(J)}$ then $\cQ_{J,\sv}(\st_{\lambda}\cdot p)=\cQ_{J,\sv}(p)$. 
\end{lemma}

\begin{proof}
We showed in Lemma~\ref{lem:actiononpaths} that folds are mapped to folds and bounces are mapped to bounces. Moreover, if a fold/bounce of $p$ occurs on a panel of type~$i$ then the corresponding fold/bounce of $\st_{\lambda}\cdot p$ occurs on a panel of type~$\sigma(i)$ for some $\sigma\in\Sigma$. Since $\sq_{\sigma(i)}=\sq_i$ (see Convention~\ref{conv:parameters}) the result follows. 
\end{proof}

We shall need the following technical results concerning the functions $\sv^{\lambda}$ and $\sv(y)$ in Section~\ref{sec:TheModule}.

\begin{lemma}\label{lem:Jparameter1}
Let $\sv$ be a $J$-parameter system. If $\lambda\in  P^{(J)}$ and $y\in W_J$ then
$$
\sv(y\sy_{\lambda})=\sv^{y\lambda}\sv(y).
$$
\end{lemma}

\begin{proof} From the definition of $\sv(\cdot)$ we have
$$
\frac{\sv(y\sy_{\lambda})}{\sv(y)}=\prod_{\alpha\in\Phi_J^+\cap\Phi_0}(\sv_{\alpha}\sv_{2\alpha})^{\sigma_{\alpha}}=\prod_{\alpha\in\Phi_J^+}\sv_{\alpha}^{\sigma_{\alpha}},
\quad\text{where}\quad 
\sigma_{\alpha}=\begin{cases}
1&\text{if $\alpha\in\Phi(y\sy_{\lambda})\backslash\Phi(y)$}\\
-1&\text{if $\alpha\in\Phi(y)\backslash\Phi(y\sy_{\lambda})$}\\
0&\text{otherwise}.
\end{cases}
$$
On the other hand, it follows from Lemma~\ref{lem:restrictingstrip} (see also Corollary~\ref{cor:lambdadecomp}) that 
$$
\sv^{y\lambda}=\prod_{\alpha\in\Phi_J^+}\sv_{\alpha}^{\langle\lambda,y^{-1}\alpha\rangle}=\prod_{\alpha\in\Phi_J^+}\sv_{\alpha}^{\sigma_{\alpha}'}\quad\text{where}\quad 
\sigma_{\alpha}'=\begin{cases}
1&\text{if $y^{-1}\alpha\in\Phi_J^+\backslash\Phi_{J\backslash J_{\lambda}}$}\\
-1&\text{if $y^{-1}\alpha\in(-\Phi_J^+)\backslash\Phi_{J\backslash J_{\lambda}}$}\\
0&\text{if $y^{-1}\alpha\in \Phi_{J\backslash J_{\lambda}}$}.
\end{cases}
$$
Since $\Phi(\sy_{\lambda})=\Phi_J^+\backslash\Phi_{J\backslash J_{\lambda}}$ (see Lemma~\ref{lem:utbasics1}) it follows that if $\alpha\in\Phi_J^+$ then
\begin{align*}
\alpha\in\Phi(y\sy_{\lambda})\backslash \Phi(y)&\Longleftrightarrow y^{-1}\alpha\in \Phi_J^+\backslash\Phi_{J\backslash J_{\lambda}}\\
\alpha\in\Phi(y)\backslash \Phi(y\sy_{\lambda})&\Longleftrightarrow y^{-1}\alpha\in (-\Phi_J^+)\backslash\Phi_{J\backslash J_{\lambda}}.
\end{align*}
Thus $\sigma_{\alpha}=\sigma_{\alpha}'$, and hence the result.
\end{proof}

\begin{lemma}\label{lem:Jparameter4}
Let $\sv$ be a $J$-parameter system. 
\begin{compactenum}[$(1)$]
\item If $j\in J$ with $2\alpha_j\notin\Phi_J$ then $\sv^{\alpha_j^{\vee}}=\sv_{\alpha_j}^2$. 
\item If $\Phi_J$ is not reduced then $\sv^{\alpha_n^{\vee}/2}=\sv_{\alpha_n}\sv_{2\alpha_n}^2$. 
\end{compactenum}
\end{lemma}

\begin{proof}
Let $K\in\cK(J)$ be the connected component with $j\in K$. If $\Phi_K$ is reduced then there are at most two root lengths in $\Phi_K$ (with all roots long in the simply laced case), and since all roots of the same length are conjugate in $W_K$ (as $K$ is connected), for $j\in K$ we have
$$
\sv^{\alpha_j^{\vee}}=\prod_{\alpha\in\Phi_J^+}\sv_{\alpha}^{\langle\alpha_j^{\vee},\alpha\rangle}=\prod_{\alpha\in\Phi_K^+}\sv_{\alpha}^{\langle\alpha_j^{\vee},\alpha\rangle}=\sv_{\mathrm{sh}}^{\langle\alpha_j^{\vee},2\rho_K'\rangle}\sv_{\mathrm{lo}}^{\langle\alpha_j^{\vee},2\rho_K\rangle},
$$
where $\sv_{\mathrm{sh}}$ (respectively $\sv_{\mathrm{lo}}$) is the constant value of $\sv_{\beta}$ on short (respectively long) roots $\beta$ in $\Phi_K$, and where $\rho_K'$ and $\rho_K$ are as Section~\ref{sec:parabolics}. It follows by Lemma~\ref{lem:orthogonal} that $\sv^{\alpha_j^{\vee}}=\sv_{\alpha_j}^2$.

If $\Phi_K$ is not reduced then direct analysis of the $\sBC_n$ root system gives
$$
\sv^{\alpha_j^{\vee}}=\prod_{\alpha\in\Phi_K^+}\sv_{\alpha}^{\langle\alpha_j^{\vee},\alpha\rangle}=\sv_{\alpha_1}^{\langle \alpha_j^{\vee},2\rho_K'\rangle}(\sv_{\alpha_n}\sv_{2\alpha_n}^2)^{\langle \alpha_j^{\vee},2\rho_K\rangle}.
$$
If $j\in K\backslash\{n\}$ we have $\langle\alpha_j^{\vee},2\rho_K\rangle=0$ and $\langle \alpha_j^{\vee},2\rho_K'\rangle=2$ and so $\sv^{\alpha_j^{\vee}}=\sv_{\alpha_1}^2=\sv_{\alpha_j}^2$, and if $j=n$ then $\langle\alpha_n^{\vee},2\rho_K'\rangle=0$ and $\langle\alpha_n^{\vee},2\rho_K\rangle=2$ giving $\sv^{\alpha_n^{\vee}/2}=\sv_{\alpha_n}\sv_{2\alpha_n}^2$ as required. 
\end{proof}

\begin{lemma}\label{lem:Jparameter3} Let $K\in\cK(J)$ and let $\alpha\in\Phi_K^+$ be a long root of $\Phi_K$ (with all roots long if $\Phi_K$ is simply laced). Then $\sv^{\alpha^{\vee}}\sv(s_{\alpha})^{-1}=\sv_{\alpha}$. 
\end{lemma}

\begin{proof}
Let $w=s_{j_1}\cdots s_{j_{\ell}}\in W_K$ be of minimal length subject to $w^{-1}\alpha=\alpha_k$ for some $k\in K$ (if $\Phi_K$ is reduced) or $w^{-1}\alpha=2\alpha_n$ (if $\Phi_K$ is not reduced). Let $\beta_0=\alpha$ and define $\beta_1,\ldots,\beta_{\ell}\in\Phi_K^+$ by $\beta_r=s_{j_r}\beta_{r-1}$ for $1\leq r\leq \ell$, and so $\beta_{\ell}=\alpha_k$ (or $\beta_{\ell}=2\alpha_n$ if $\Phi_K$ is not reduced). Since each $\beta_r$ is a long root of $\Phi_K$ we have $\langle\beta_{r-1}^{\vee},\alpha_{j_r}\rangle\in\{-1,0,1\}$ (see \cite[IV, \S1, Proposition~8]{Bou:02}) and so
$$
\beta_r^{\vee}=s_{j_r}\beta_{r-1}^{\vee}=\beta_{r-1}^{\vee}-\langle\beta_{r-1}^{\vee},\alpha_{j_r}\rangle\alpha_{j_r}^{\vee}=\beta_{r-1}^{\vee}-\alpha_{j_r}^{\vee}
$$
(because $\langle \beta_{r-1}^{\vee},\alpha_{j_r}\rangle\in\{0,-1\}$ contradicts minimality of $w$, see for example \cite[Lemma~1.7]{BH:93}). 

We claim that 
$$
\sv^{\beta_r^{\vee}}\sv(s_{\beta_r})^{-1}=\sv_{\alpha}\quad\text{for $0\leq r\leq \ell$}.
$$
We argue by downward induction on~$r$. If $r=\ell$ then $\beta_{\ell}=\alpha_k$ (if $\Phi_K$ is reduced) or $\beta_{\ell}=2\alpha_n$ (if $\Phi_K$ is not reduced), and the result follows from Lemma~\ref{lem:Jparameter4}, starting the induction. 

Since $s_{\beta_{r-1}}=s_{j_r}s_{\beta_{r}}s_{j_r}$ (with length adding) we have
\begin{align*}
\sv^{\beta_{r-1}^{\vee}}\sv(s_{\beta_{r-1}})^{-1}&=\sv^{\alpha_{j_r}^{\vee}}\sv(s_{j_r})^{-2}\sv^{\beta_r^{\vee}}\sv(s_{\beta_r})^{-1}.
\end{align*}
Note that if $\Phi_K$ is not reduced then $j_r\neq n$ for $1\leq r\leq \ell$ (because $s_{n-1}\cdots s_k(2e_k)=2e_n$), and so by Lemma~\ref{lem:Jparameter4} we have $\sv^{\alpha_{j_r}^{\vee}}\sv(s_{j_r})^{-2}=1$ for all $1\leq r\leq \ell$, and the result follows by induction. 
\end{proof}

\begin{lemma}\label{lem:Jparameter2}
Let $\sv$ be a $J$-parameter system. If $K\in\cK(J)$ and $y\in W_J$ then
 $$
 \frac{\sv(y)}{\sv(ys_{\varphi_K})}\sv^{y\varphi_K^{\vee}}=\begin{cases}
 \sv_{\varphi_K}&\text{if $y\varphi_K\in\Phi_J^+$}\\
 \sv_{\varphi_K}^{-1}&\text{if $y\varphi_K\in-\Phi_J^+$.}
 \end{cases}
 $$
\end{lemma}

\begin{proof}
We may assume without loss of generality that $y\in W_K$ (making use of Lemma~\ref{lem:Jparameter3}). Let $\alpha\in\Phi_K^+$. If $y^{-1}\alpha\neq\pm\varphi_K$ we claim that
$$
\langle\varphi_K^{\vee},y^{-1}\alpha\rangle=\begin{cases}1&\text{if $y^{-1}\alpha\in\Phi(s_{\varphi_K})$}\\
-1&\text{if $y^{-1}\alpha\in-\Phi(s_{\varphi_K})$}\\
0&\text{otherwise}.
\end{cases}
$$
To see this, note that by \cite[Chapter VI, \S1.8]{Bou:02} if $\beta\in\Phi_K^+$ then $\langle\varphi_K^{\vee},\beta\rangle\in\{0,1\}$, and since $s_{\varphi_K}(\beta)=\beta-\langle\varphi_K^{\vee},\beta\rangle\varphi_K$ we have $\beta\in\Phi(s_{\varphi_K})$ if and only if $\langle\varphi_K^{\vee},\beta\rangle=1$, and the claim follows. 

Let $\epsilon\in\{-1,1\}$ be such that $y\varphi_K\in\epsilon \Phi^+$. Using the claim, and the fact that $\langle\varphi_K^{\vee},\varphi_K\rangle=2$, we have
\begin{align*}
\sv^{y\varphi_K^{\vee}}&=\prod_{\alpha\in\Phi_K^+}\sv_{\alpha}^{\langle\varphi_K^{\vee},y^{-1}\alpha\rangle}=\sv_{\varphi_K}^{\epsilon}\prod_{\alpha\in\Phi_K^+}\sv_{\alpha}^{\sigma_{\alpha}}\quad\text{where}\quad 
\sigma_{\alpha}=\begin{cases}
1&\text{if $y^{-1}\alpha\in\Phi(s_{\varphi_K})$}\\
-1&\text{if $y^{-1}\alpha\in-\Phi(s_{\varphi_K})$}\\
0&\text{otherwise}.
\end{cases}
\end{align*}

On the other hand, as in the proof of Lemma~\ref{lem:Jparameter1} we have
$$
\frac{\sv(ys_{\varphi_K})}{\sv(y)}=\prod_{\alpha\in\Phi_K^+}\sv_{\alpha}^{\sigma_{\alpha}'},
\quad\text{where}\quad 
\sigma_{\alpha}'=\begin{cases}
1&\text{if $\alpha\in\Phi(ys_{\varphi_K})\backslash\Phi(y)$}\\
-1&\text{if $\alpha\in\Phi(y)\backslash\Phi(ys_{\varphi_K})$}\\
0&\text{otherwise}.
\end{cases}
$$
The result now follows from the fact that if $\alpha\in\Phi_K^+$ then $\alpha\in\Phi(ys_{\varphi_K})\backslash \Phi(y)$ if and only if $y^{-1}\alpha\in\Phi(s_{\varphi_K})$, and $\alpha\in\Phi(y)\backslash \Phi(ys_{\varphi_K})$ if and only if $y^{-1}\alpha\in-\Phi(s_{\varphi_K})$. 
\end{proof}

\subsection{$J$-straightening}

Given a positively $J$-folded alcove path $p\in\cP_J(\vec{w},v)$ (with $v\in\WJ$) we define the \textit{$J$-straightened} alcove path $p_J$ to be the path obtained by straightening all bounces of $p$ (note that the folds are not straightened). More formally, let $p=(v_0,v_1,\ldots,v_{\ell},v_{\ell}\sigma)$ and suppose that the bounces occur at indices $1\leq k_1<k_2<\cdots<k_r\leq \ell$, and that they occur on the hyperplanes $H_{\beta_1,\nu_1},\ldots,H_{\beta_r,\nu_r}$ with $\beta_1,\ldots,\beta_r\in\Phi_J^+$ and $\nu_1,\ldots,\nu_r\in\{0,1\}$. Write $p=p_0\cdot p_1\cdot p_2 \cdots p_r$, where $p_0=(v_0,\ldots,v_{k_1-1})$, $p_j=(v_{k_j},\ldots,v_{k_{j+1}-1})$ for $1\leq j\leq r-1$, and $p_r=(v_{k_r},\ldots,v_{\ell},v_{\ell}\sigma)$. Then $p_J$ is the path 
$
p_0\cdot (s_{\beta_1,\nu_1}p_1)\cdot (s_{\beta_1,\nu_1}s_{\beta_2,\nu_2}p_2)\cdots  (s_{\beta_1,\nu_1}\cdots s_{\beta_r,\nu_r}p_r).
$

\begin{example}
The path in Figure~\ref{fig:G2examplestraight} is the $J$-straightening of the positively $J$-folded alcove path in Figure~\ref{fig:G2example}.
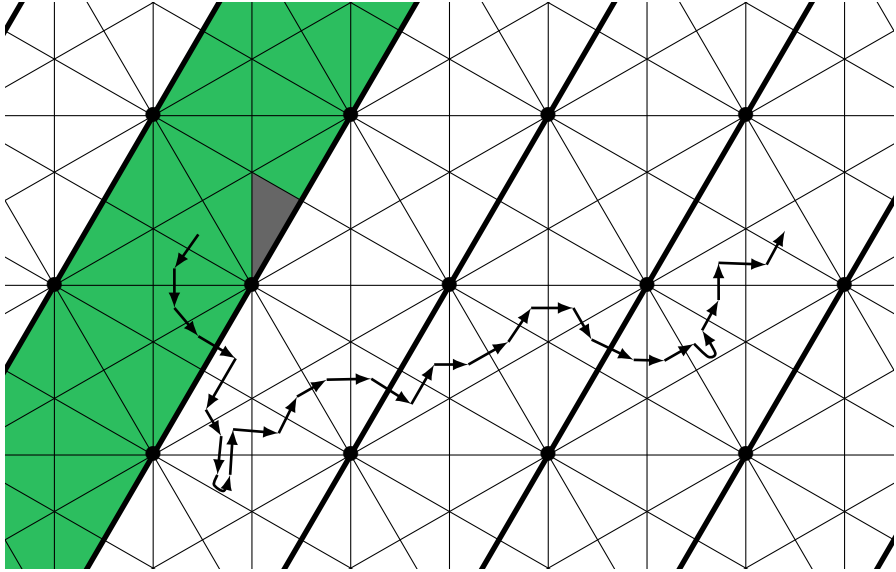
\begin{figure}[H]
\centering
\begin{tikzpicture}[scale=1.5]
\clip ({-5*0.866+0.433},{-2.5}) rectangle + ({9*0.866},{5});
\path [fill=green2] ({-4*0.866-1.732},-3)--({-2*0.866-1.732},-3)--({3*0.866-1.732},4.5)--({1*0.866-1.732},4.5)--({-4*0.866-1.732},-3);
\path [fill=black!60] (0-1.732,0) -- (0.433-1.732,0.75) -- (0-1.732,1) -- (0-1.732,0);
\draw(-4.33,4.5)--(4.33,4.5);
\draw(-4.33,3)--(4.33,3);
\draw(-4.33,1.5)--(4.33,1.5);
\draw(-4.33,0)--(4.33,0);
\draw(-4.33,-1.5)--(4.33,-1.5);
\draw(-4.33,-3)--(4.33,-3);
\draw(-4.33,-3)--(-4.33,4.5);
\draw(-3.464,-3)--(-3.464,4.5);
\draw(-2.598,-3)--(-2.598,4.5);
\draw(-1.732,-3)--(-1.732,4.5);
\draw(-.866,-3)--(-.866,4.5);
\draw(0,-3)--(0,4.5);
\draw(.866,-3)--(.866,4.5);
\draw(1.732,-3)--(1.732,4.5);
\draw(2.598,-3)--(2.598,4.5);
\draw(3.464,-3)--(3.464,4.5);
\draw(4.33,-3)--(4.33,4.5);
\draw(-4.33,3.5)--({-3*0.866},4.5);
\draw(-4.33,2.5)--({-1*0.866},4.5);
\draw(-4.33,1.5)--({1*0.866},4.5);
\draw(-4.33,.5)--({3*0.866},4.5);
\draw(-4.33,-.5)--(4.33,4.5);
\draw(-4.33,-1.5)--(4.33,3.5);
\draw(-4.33,-2.5)--(4.33,2.5);
\draw(-3.464,-3)--(4.33,1.5);
\draw(-1.732,-3)--(4.33,.5);
\draw(0,-3)--(4.33,-.5);
\draw(1.732,-3)--(4.33,-1.5);
\draw(3.464,-3)--(4.33,-2.5);
\draw(4.33,3.5)--({3*0.866},4.5);
\draw(4.33,2.5)--({1*0.866},4.5);
\draw(4.33,1.5)--({-1*0.866},4.5);
\draw(4.33,.5)--({-3*0.866},4.5);
\draw(4.33,-.5)--(-4.33,4.5);
\draw(4.33,-1.5)--(-4.33,3.5);
\draw(4.33,-2.5)--(-4.33,2.5);
\draw(3.464,-3)--(-4.33,1.5);
\draw(1.732,-3)--(-4.33,.5);
\draw(0,-3)--(-4.33,-.5);
\draw(-1.732,-3)--(-4.33,-1.5);
\draw(-3.464,-3)--(-4.33,-2.5);
\draw(-4.33,-1.5)--(-3.464,-3);
\draw(-4.33,1.5)--(-1.732,-3);
\draw(-4.33,4.5)--(0,-3);
\draw({-3*0.866},4.5)--(1.732,-3);
\draw({-1*0.866},4.5)--(3.464,-3);
\draw({1*0.866},4.5)--(4.33,-1.5);
\draw({3*0.866},4.5)--(4.33,1.5);
\draw(4.33,-1.5)--(3.464,-3);
\draw(4.33,1.5)--(1.732,-3);
\draw(4.33,4.5)--(0,-3);
\draw({3*0.866},4.5)--(-1.732,-3);
\draw({1*0.866},4.5)--(-3.464,-3);
\draw({-1*0.866},4.5)--(-4.33,-1.5);
\draw({-3*0.866},4.5)--(-4.33,1.5);
\draw[line width=2pt]({-5*0.866},1.5)--({-3*0.866},4.5);
\draw[line width=2pt]({-5*0.866},-1.5)--({-1*0.866},4.5);
\draw[line width=2pt]({-4*0.866},-3)--({1*0.866},4.5);
\draw[line width=2pt]({-2*0.866},-3)--({3*0.866},4.5);
\draw[line width=2pt]({0*0.866},-3)--({5*0.866},4.5);
\draw[line width=2pt]({2*0.866},-3)--({5*0.866},1.5);
\draw[line width=2pt]({4*0.866},-3)--({5*0.866},-1.5);
\draw[line width=1pt,-latex](-0.47-1.732,0.45)--(-0.68-1.732,0.15);
\draw[line width=1pt,-latex](-0.68-1.732,0.15)--(-0.68-1.732,-0.2);
\draw[line width=1pt,-latex](-0.68-1.732,-0.2)--(-0.47-1.732,-0.45);
\draw[line width=1pt,-latex](-0.47-1.732,-0.45)--(-0.1926-1.732+0.05,-0.7463+0.1);
\draw[line width=1pt,-latex](-0.1926-1.732+0.05,-0.7463+0.1)--(-0.4044-1.732,-1.0995);
\draw[line width=1pt,-latex](-0.4044-1.732,-1.0995)--(-0.269-1.732,-1.333);
\draw[line width=1pt,-latex](-0.269-1.732,-1.333)--(-0.3026-1.732,-1.676);
\draw[line width=1pt,-latex] plot[smooth] coordinates {(-0.3026-1.732,-1.676) (-0.333-1.732,-1.7627) (-0.2449-1.732,-1.816) (-0.191-1.732,-1.675)};
\draw[line width=1pt,-latex](-0.191-1.732,-1.675)--(-0.1655-1.732,-1.2733);
\draw[line width=1pt,-latex](-0.1655-1.732,-1.2733)--(0.235-1.732,-1.307);
\draw[line width=1pt,-latex](0.235-1.732,-1.307)--(0.395-1.732,-0.984);
\draw[line width=1pt,-latex](0.395-1.732,-0.984)--(0.655-1.732,-0.834);
\draw[line width=1pt,-latex](0.655-1.732,-0.834)--(0.996-1.732+0.05,-0.9258+0.1);
\draw[line width=1pt,-latex](0.996-1.732+0.05,-0.9258+0.1)--(-1.2+2.598-1.732,0.45-1.5);
\draw[line width=1pt,-latex](-1.2+2.598-1.732,0.45-1.5)--(-1+2.598-1.732,0.8-1.5);
\draw[line width=1pt,-latex](-1+2.598-1.732,0.8-1.5)--(-0.7+2.598-1.732,0.8-1.5);
\draw[line width=1pt,-latex](-0.7+2.598-1.732,0.8-1.5)--(-0.35+2.598-1.732,1-1.5);
\draw[line width=1pt,-latex](-0.35+2.598-1.732,1-1.5)--(-0.15+2.598-1.732,1.3-1.5);
\draw[line width=1pt,-latex](-0.15+2.598-1.732,1.3-1.5)--(0.22+2.598-1.732,1.3-1.5);
\draw[line width=1pt,-latex](0.22+2.598-1.732,1.3-1.5)--(0.38+2.598-1.732,1.02-1.5);
\draw[line width=1pt,-latex](0.38+2.598-1.732,1.02-1.5)--(0.749+2.598-1.732, 0.941-1.5-0.1);
\draw[line width=1pt,-latex](0.749+2.598-1.732, 0.941-1.5-0.1)--(1.019+2.598-1.732, 0.934-1.5-0.1);
\draw[line width=1pt,-latex](1.019+2.598-1.732, 0.934-1.5-0.1)--(1.278+2.598-1.732, 1.084-1.5-0.1);
\draw[line width=1pt,-latex] plot[smooth] coordinates {(1.278+2.598-1.732, 1.084-1.5-0.1)(1.403+2.598-1.732, 0.868-1.5)(1.464+2.598-1.732, 0.903-1.5)(1.349+2.598-1.732, 1.102-1.5)};
\draw[line width=1pt,-latex](1.349+2.598-1.732, 1.102-1.5)--(1.414+2.598-1.732+0.08, 1.370-1.5);
\draw[line width=1pt,-latex](1.414+2.598-1.732+0.08, 1.370-1.5)--(1.443+2.598-1.732+0.05, 1.699-1.5);
\draw[line width=1pt,-latex](1.443+2.598-1.732+0.05, 1.699-1.5)--(1.915+2.598-1.732, 1.683-1.5);
\draw[line width=1pt,-latex](1.915+2.598-1.732, 1.683-1.5)--(2.076+2.598-1.732, 1.983-1.5);
\node at ({-4*0.866},{2*1.5}) {\Large{$\bullet$}};
\node at ({-2*0.866},{2*1.5}) {\Large{$\bullet$}};
\node at ({0*0.866},{2*1.5}) {\Large{$\bullet$}};
\node at ({2*0.866},{2*1.5}) {\Large{$\bullet$}};
\node at ({4*0.866},{2*1.5}) {\Large{$\bullet$}};
\node at ({-3*0.866},{1.5}) {\Large{$\bullet$}};
\node at ({-1*0.866},{1.5}) {\Large{$\bullet$}};
\node at ({1*0.866},{1.5}) {\Large{$\bullet$}};
\node at ({3*0.866},{1.5}) {\Large{$\bullet$}};
\node at ({-4*0.866},0) {\Large{$\bullet$}};
\node at ({-2*0.866},0) {\Large{$\bullet$}};
\node at ({0*0.866},0) {\Large{$\bullet$}};
\node at ({2*0.866},0) {\Large{$\bullet$}};
\node at ({4*0.866},0) {\Large{$\bullet$}};
\node at ({-3*0.866},{-1*1.5}) {\Large{$\bullet$}};
\node at ({-1*0.866},{-1*1.5}) {\Large{$\bullet$}};
\node at ({1*0.866},{-1*1.5}) {\Large{$\bullet$}};
\node at ({3*0.866},{-1*1.5}) {\Large{$\bullet$}};
\end{tikzpicture}
\caption{$J$-straightening of the path in Figure~\ref{fig:G2example}}\label{fig:G2examplestraight}
\end{figure}
\end{example}

In Figure~\ref{fig:G2examplestraight} the $J$-straightening $p_J$ turned out to be positively folded. The following proposition shows that this is no coincidence.

\begin{prop}\label{prop:unfold}
Let $w\in\Wext$ and let $\vec{w}$ be any expression for $w$ (not necessarily reduced). Let $v\in\WJ$ and let $p\in\cP_J(\vec{w},v)$. The $J$-straightening map $p\mapsto p_J$ is a bijection from $\cP_J(\vec{w},v)$ to the set
$
\{p'\in\cP(\vec{w},v)\mid \text{ $p'$ has no folds on hyperplanes $H_{\alpha,k}$ with $\alpha\in\Phi_J^+$ and $k\in\ZZ$}\}.
$
\end{prop}

\begin{proof}
Consider the effect of straightening a bounce. Let $\alpha\in\Phi_J^+$ be the linear root associated to the wall on which the bounce occurs. A later positive fold that occurs on a $\beta$-wall (necessarily with $\beta\in\Phi^+\backslash \Phi_J$) now occurs on a $s_{\alpha}\beta$ wall. Since $\beta\notin\Phi_J$ and $s_{\alpha}\in W_J$ we have $s_{\alpha}\beta\in\Phi^+\backslash \Phi_J$ (as $s_{\alpha}$ permutes the set $\Phi_J$) and hence this reflected fold remains a positive fold (see~(\ref{eq:affineroots})), and it does not occur on a hyperplane $H_{\gamma,k}$ with $\gamma\in\Phi_J$ (a ``$\Phi_J$-wall''). This shows that $p_J$ is positively folded, with no folds occurring on $\Phi_J$-walls.  

A similar argument shows that one may apply the reverse procedure, starting with a positively folded alcove path $p'\in\cP(\vec{w},v)$ with no folds on $\Phi_J$-walls and forcing the bounces on the $\Phi_J$-walls. These operations are mutually inverse procedures, hence the result.
\end{proof}

We now give some more precise information on the $J$-straightening relating to the $\sv$-mass. If $p$ is a positively folded alcove path, for each $\alpha-k\delta\in\widetilde{\Phi}^+$ we define
$$
c_{\alpha,k}(p)=\#(\text{crossings in $p$ that occur on the hyperplane $H_{\alpha,k}$}).
$$
Note that in the non-reduced case, if $2\alpha\in\Phi$ then $c_{2\alpha,2k}(p)=c_{\alpha,k}$. 

Given a $J$-parameter system $\sv=(\sv_{\alpha})_{\alpha\in\Phi_J}$ we define
\begin{align}\label{eq:extendv}
\sv_{\alpha+k\delta}=\sv_{\alpha}\quad\text{if $\alpha\in\Phi_J$ and $k\in\ZZ$},
\end{align}
and we set $\sv_{\alpha+k\delta}=1$ if $\alpha+k\delta\notin\Phi_J+\ZZ\delta$.

\begin{prop}\label{prop:unfold2}
Let $p$ be a positively $J$-folded alcove path (not necessarily of reduced type) and let $\sv=(\sv_{\alpha})_{\alpha\in\Phi_J}$ be a $J$-parameter system. Then

$$
\cQ_{J,\sv}(p)=\cQ(p_J)\prod_{\alpha-k\delta\in\widetilde{\Phi}^+}\sv_{\alpha-k\delta}^{c_{\alpha,k}(p_J)}
$$
where $\cQ(\cdot)$ is as in Proposition~\ref{prop:basischange} (note that the product has finitely many terms $\neq 1$). 
\end{prop}

\begin{proof}
First, note that under the bijection $p\mapsto p_J$ the types of the folds are preserved (that is, $i$-folds in $p$ are mapped to $i$-folds in $p_J$) because the $J$-affine Weyl group $\WJaff$ is type preserving. This shows that 
$$
\prod_{i\in\Iaff}(\sq_i-\sq_i^{-1})^{f_i(p)}=\prod_{i\in\Iaff}(\sq_i-\sq_i^{-1})^{f_i(p_J)}=\cQ(p_J).
$$

Thus it remains to show that if $p\in\cP_J(\vec{w},v)$ then
$$
\prod_{\alpha\in\Phi_J^+}\sv_{\alpha}^{b_{\alpha}(p)}=\prod_{\alpha-k\delta\in\widetilde{\Phi}^+}\sv_{\alpha-k\delta}^{c_{\alpha,k}(p_J)},
$$
which in turn follows from the fact that each crossing of a $\Phi_J+\ZZ\delta$ wall in the $J$-straightened path $p_J$ corresponds to a bounce on a wall of $\cA_J$ in the path $p$, and from the definition of $J$-straightening if $p_J$ has a crossing on $H_{\beta,k}$ with $\beta\in\Phi_J$ then there is $v\in \WJaff$ with $vH_{\beta,k}$ equal to the wall $H_{\gamma,r}$ of $\cA_J$ on which the corresponding bounce occurred. 
 \end{proof}

\section{The $\Hext$-modules $M_{J,\sv}$}\label{sec:TheModule}

In this section we construct finite dimensional $\Hext$-modules $(\pi_{J,\sv},M_{J,\sv})$ (Theorem~\ref{thm:module}), and develop a combinatorial formula for the matrix entries of $\pi_{J,\sv}(T_w)$ in terms of positively $J$-folded alcove paths (Theorems~\ref{thm:mainpath1} and~\ref{thm:mainpath2}). We prove that these modules are irreducible (Corollary~\ref{cor:irreducible}) and realise them as induced representations from $1$-dimensional representations of Levi subalgebras (Theorem~\ref{thm:induced}). In fact we will show that all representations of $\Hext$ that are induced from $1$-dimensional representations of Levi subalgebras arise from our construction (after specialising certain ``variables'' appropriately, see Remark~\ref{rem:induced}). 

Our $\Hext$-modules are inspired by the modules constructed by Lusztig~\cite[Lemma~4.7]{Lus:97} and Deodhar \cite[Corollary~2.3]{Deo:87}, however note that those modules are infinite dimensional when applied to our setting. To motivate the general philosophy of Theorem~\ref{thm:module}, recall Lusztig's \textit{periodic Hecke module} $M=\bigoplus_{A\in\mathbb{A}} \sR A$ (with $\sR$ as in Section~\ref{sec:hecke}) with basis indexed by the (classical) alcoves and $\Hext$-action given by (see \cite[\S3.2]{Lus:97})
$$
(wA_0)\cdot T_i=\begin{cases}
ws_iA_0&\text{if $wA_0\,{^-}|^+\, ws_iA_0$}\\
ws_iA_0+(\sq_i-\sq_i^{-1})wA_0&\text{if $wA_0\,{^+}|^-\, ws_iA_0$}
\end{cases}
$$
for $i\in \{0\}\cup\Ifin$. The idea behind Theorem~\ref{thm:module} is to adapt this action to create a $J$-analogue that also incorporates the action of the translation group $\TJ$ on $\cA_J$ to create a finite dimensional $\Hext$-module. We shall be able to achieve this goal for each subset $J\subseteq \Ifin$ and each $J$-parameter system~$\sv$, thus constructing many ``combinatorial'' finite dimensional $\Hext$-modules.

\subsection{Construction of the module $M_{J,\sv}$}\label{sec:moduleconstruction}

Let $\{\zeta_i\mid i\in I\}$ be a family of commuting invertible indeterminates, and for $\la\in P$ let
$\zeta^{\la}=\prod_{i\in I}\zeta_i^{a_i}$ if $\la=\sum_{i\in I}a_i\omega_i$. For $J\subseteq\Ifin$, let $\mathcal{I}_J$ denote the ideal of the Laurent polynomial ring $\sR[\{\zeta_i^{\pm 1}\mid i\in I\}]$ generated by the elements $\zeta^{\alpha_j^{\vee}}-1$ for $j\in J$. Let 
$$
\zeta_J^{\la}=\zeta^{\la}+\mathcal{I}_J\quad\text{and write}\quad \sR[\zeta_J]=\sR[\{\zeta_i^{\pm 1}\mid i\in I\}]/\mathcal{I}_J.
$$
Thus the ring $\sR[\zeta_J]$ consists of elements of the form 
$$
\sum_{\la\in P}a_{\la}\zeta_J^{\la}
$$
with only finitely many nonzero terms, where $a_{\la}\in\sR$. We note that $\zeta_J^{\gamma}=1$ for all $\gamma\in Q_J$, and thus $\zeta_J^{\la}=\zeta_J^{\la^{(J)}}$ for all $\la\in P$.

The following theorem introduces the main objects of study in this paper. Recall that by~(\ref{eq:extendv}) we extend the definition of $\sv_{\alpha}$ to $\sv_{\alpha+k\delta}$ for all $\alpha+k\delta\in\widetilde{\Phi}$, and that $\alpha_0=-\varphi+\delta$.

\begin{thm}\label{thm:module}
Let $J\subseteq \Ifin$ and let $\sv$ be a $J$-parameter system. Let $M_{J,\sv}$ be a module over the ring $\sR[\zJ]$ with basis $\sB_{J,\sv}=\{\bm_u\mid u\in W^J\}$, and for $i\in\{0\}\cup\Ifin$ and $\sigma\in\Sigma$ define
\begin{align*}
\bm_u\cdot T_i&=\begin{cases}
\zJ^{\wt(us_i)}\bm_{\theta^J(us_i)}&\text{if $u\to us_i$ positive and $us_i\in\WJ$}\\
\zJ^{\wt(us_i)}\bm_{\theta^J(us_i)}+(\sq_i-\sq_i^{-1})\bm_u&\text{if $u\to us_i$ is negative and $us_i\in\WJ$}\\
\sv_{u\alpha_i}\sv_{2u\alpha_i}\bm_u&\text{if $us_i\notin\WJ$ (hence $u\alpha_i\in\Phi_J+\ZZ\delta$)}
\end{cases}\\
\bm_u\cdot T_{\sigma}&=\zeta_J^{\wt(u\sigma)}\bm_{\theta^J(u\sigma)}.
\end{align*}
This extends to a right action of $\Hext$ on the module $M_{J,\sv}$. 
\end{thm}

Before proving Theorem~\ref{thm:module} we give some lemmas and auxiliary results.

\begin{lemma}\label{lem:littlebitsa}
Let $w,v\in\Wext$. Then
\begin{compactenum}[$(1)$]
\item $\theta^J(wv)=\theta^J(\theta^J(w)v)$.
\item $\theta_J(wv)=\theta_J(w)\theta_J(\theta^J(w)v)$
\item $\wt(wv)=\wt(w)+\theta_J(w)\wt(\theta^J(w)v)$. 
\end{compactenum}
\end{lemma}

\begin{proof}
We have
\begin{align*}
wv&=t_{\wt(w)}\theta_J(w)\theta^J(w)v\\
&=t_{\wt(w)}\theta_J(w)t_{\wt(\theta^J(w)v)}\theta_J(\theta^J(w)v)\theta^J(\theta^J(w)v)\\
&=t_{\wt(w)+\theta_J(w)\wt(\theta^J(w)v)}\theta_J(w)\theta_J(\theta^J(w)v)\theta^J(\theta^J(w)v)
\end{align*}
and hence (1), (2), and (3).
\end{proof}

\begin{lemma}\label{lem:bitofmagic} Let $w\in\Wext$ and $i\in\{0\}\cup\Ifin$. Let $u=\theta^J(w)$. Then
$$
\frac{\sv(\theta_J(w))}{\sv(\theta_J(ws_i))}\sv^{\wt(ws_i)-\wt(w)}=\begin{cases}
1&\text{if $us_i\in\WJ$}\\
\sv_{u\alpha_i}^{\epsilon}\sv_{2u\alpha_i}^{\epsilon}&\text{if $us_i\notin\WJ$}.
\end{cases}
$$
where $\epsilon\in\{-1,1\}$ is the sign of the crossing $w\to ws_i$.
\end{lemma}

\begin{proof} The expression on the left hand side of the equation is invariant under replacing $w$ with $t_{\lambda}w$ for any $\lambda\in P$, and the sign of the crossing $w\to ws_i$ is the same as the sign of the crossing $t_{\lambda}w\to t_{\lambda}ws_i$. Therefore we can assume without loss of generality that $w\in W_0$.  

So $w=yu\in W_0$ with $y=\theta_J(w)$ and $u=\theta^J(w)$. Suppose first that $i\in \Ifin$. If $us_i\in \WJ$ then $us_i\in W^J$ and so $\theta^J(ws_i)=us_i$, and hence $\theta_J(ws_i)=y$, and the result follows. If $us_i\notin \WJ$ then $us_i=s_ju$ for a unique $j\in J$, and so $\theta_J(ws_i)=ys_j$. If $w\to ws_i$ is positive then $\ell(ws_i)=\ell(w)-1$ (as $w\in W_0$) and so $\ell(ys_j)=\ell(y)-1$, and if $w\to ws_i$ is negative then $\ell(ys_j)=\ell(y)+1$. Thus 
$$
\frac{\sv(\theta_J(w))}{\sv(\theta_J(ws_i))}\sv^{\wt(ws_i)-\wt(w)}=\frac{\sv(y)}{\sv(y)\sv_{\alpha_j}^{-\epsilon}\sv_{2\alpha_j}^{-\epsilon}}=\sv_{\alpha_j}^{\epsilon}\sv_{2\alpha_j}^{\epsilon},
$$
as required. 

Suppose now that $i=0$. If $us_0\in\WJ$ then $u\varphi^{\vee}\in P^{(J)}$, and so by Corollary~\ref{cor:splitting2} we have $\theta_J(us_0)=\sy_{u\varphi^{\vee}}$. Thus $\theta_J(ws_0)=y\sy_{u\varphi^{\vee}}$ and so 
$$
\frac{\sv(\theta_J(w))}{\sv(\theta_J(ws_i))}\sv^{\wt(ws_i)-\wt(w)}=\frac{\sv(y)}{\sv(y\sy_{u\varphi^{\vee}})}\sv^{yu\varphi^{\vee}},
$$
and since $u\varphi^{\vee}\in P^{(J)}$ the result follows by Lemma~\ref{lem:Jparameter1}. 

If $us_0\notin\WJ$ then $u\varphi^{\vee}\in\Phi_J$. We have $\wt(ws_0)=w\varphi^{\vee}$, and $ws_0=t_{w\varphi^{\vee}}yus_{\varphi}=t_{yu\varphi^{\vee}}(ys_{u\varphi})u$ and so $\theta_J(ws_0)=ys_{u\varphi}$ (as $s_{u\varphi}\in W_J$). Thus
$$
\frac{\sv(\theta_J(w))}{\sv(\theta_J(ws_i))}\sv^{\wt(ws_i)-\wt(w)}=\frac{\sv(y)}{\sv(ys_{u\varphi})}\sv^{yu\varphi^{\vee}}.
$$
Since $us_0\notin\WJ$ we have $u\alpha_0\in\Phi_J+\ZZ\delta$. Thus $u\varphi\in\Phi_J$ (since $\alpha_0=-\varphi+\delta$). Hence by Lemma~\ref{lem:boundingwalls} we have $u\varphi=\varphi_K$ for some $K\in\cK(J)$. Moreover $w\to ws_0$ is positive if and only if $yu\varphi\in\Phi_J^+$, if and only if $y\varphi_K\in\Phi_J^+$. The result now follows from Lemma~\ref{lem:Jparameter2}. 
\end{proof}

\begin{lemma}\label{lem:littlebitsb}
Let $w\in \Wext$ and $i\in\{0\}\cup\Ifin$. Write $u=\theta^J(w)$, and suppose that $us_i \in\WJ$. Then $w\to ws_i$ is a positive crossing if and only if $u\to us_i$ is a positive crossing. 
\end{lemma}

\begin{proof}
Since $us_i\in\WJ$ we have $u\alpha_i\notin\Phi_J+\ZZ\delta$. Suppose that $w\to ws_i$ is positive, and so  $w\alpha_i=-\alpha+k\delta\in -\Phi^++\ZZ\delta$ (recall~(\ref{eq:affineroots})). Write $w=t_{\lambda}yu$ with $\lambda=\wt(w)$, $y= \theta_J(w)$, and $u=\theta^J(w)$. Then
$
u\alpha_i=y^{-1}t_{-\lambda}(w\alpha_i)=y^{-1}t_{-\lambda}(-\alpha+k\delta)\in -y^{-1}\alpha+\ZZ\delta.
$
Since $y^{-1}\alpha\notin\Phi_J$ and $y\in W_J$ we have $\alpha\notin\Phi_J$. Then $y^{-1}\alpha$ is necessarily a positive (linear) root, and so $u\to us_i$ is positive. 

Conversely, suppose that $u\to us_i$ is positive. So $u\alpha_i=-\beta+k\delta$ with $\beta\in\Phi^+$, and since $us_i\in\WJ$ we have $\beta\notin\Phi_J$. Then 
$
w\alpha_i=t_{\lambda}yu\alpha_i=t_{\lambda}(-y\beta+k\delta)\in-y\beta+\ZZ\delta,
$
and since $\beta\notin\Phi_J$ we have $y\beta>0$, hence the result.
\end{proof}

The following proposition is the key to proving Theorem~\ref{thm:module}. Define a linear map $\varpi_{J,\sv}:\Hext\to M_{J,\sv}$ by linearly extending the definition
\begin{align*}
\varpi_{J,\sv}(X_w)&=(\sv\zJ)^{\wt(w)}\sv(\theta_J(w))^{-1}\bm_{\theta^J(w)}
\end{align*}
for $w\in \Wext$. (To motivate this definition, note that if we assume for the moment that the proposed action given in Theorem~\ref{thm:module} is indeed an action, then $\varpi_{J,\sv}(X_w)=\bm_e\cdot X_w$).

\begin{prop}\label{prop:module}
For $i\in\{0\}\cup\Ifin$ and $\sigma\in\Sigma$ we have
\begin{align*}
\varpi_{J,\sv}(hT_i)=\varpi_{J,\sv}(h)\cdot T_i\quad\text{and}\quad
\varpi_{J,\sv}(hT_{\sigma})=\varpi_{J,\sv}(h)\cdot T_{\sigma}
\end{align*}
for all $h\in\Hext$ (with $\bm\cdot T_i$ and $\bm\cdot T_{\sigma}$ as in the statement of Theorem~\ref{thm:module} for $\bm\in M_{J,\sv}$). 
\end{prop}

\begin{proof}
By linearity it is sufficient to prove that $\varpi_{J,\sv}(X_wT_i)=\varpi_{J,\sv}(X_w)\cdot T_i$ and $\varpi_{J,\sv}(X_wT_{\sigma})=\varpi_{J,\sv}(X_w)\cdot T_{\sigma}$ for all $w\in\Wext$, $i\in\{0\}\cup\Ifin$, and $\sigma\in\Sigma$. Consider the second formula. Let $\mu\in P$ be such that $t_{\mu}w\in\WJ$. Then $\theta(t_{\mu}w)=\theta(w)$ and so by Corollary~\ref{cor:splitting2} and Lemma~\ref{lem:Jparameter1} we have
$$
\sv(\theta_J(w))=\sv(\theta_J(t_{\mu}w))=\sv(\sy_{\wt(t_{\mu}w)})=\sv^{\wt(t_{\mu}w)}=\sv^{\mu+\wt(w)}.
$$
Since $t_{\mu}w\sigma\in \WJ$ (as $\sigma A_0=A_0$) we similarly have $\sv(\theta_J(w\sigma))=\sv^{\mu+\wt(w\sigma)}$, and it follows that 
$$
\sv^{\wt(w\sigma)}\sv(\theta_J(w\sigma))^{-1}=\sv^{-\mu}=\sv^{\wt(w)}\sv(\theta_J(w))^{-1}.
$$
Using this formula, along with the definition of $\varpi_{J,\sv}$, we have
\begin{align*}
\varpi_{J,\sv}(X_wT_{\sigma})=\varpi_{J,\sv}(X_{w\sigma})&=(\sv\zJ)^{\wt(w\sigma)}\sv(\theta_J(w\sigma))^{-1}\bm_{\theta^J(w\sigma)}\\
&=\zeta_J^{\wt(w\sigma)-\wt(w)}(\sv\zeta_J)^{\wt(w)}\sv(\theta_J(w))^{-1}\bm_{\theta^J(w\sigma)}
\end{align*}
and so $\varpi_{J,\sv}(X_wT_{\sigma})=\varpi_{J,\sv}(X_w)\cdot T_{\sigma}$ as required.

We now prove the first formula, $\varpi_{J,\sv}(X_wT_i)=\varpi_{J,\sv}(X_w)\cdot T_i$. Write $u=\theta^J(w)$. Since $T_i=T_i^{-1}+(\sq_i-\sq_i^{-1})$ we have $X_wT_i=X_{ws_i}$ if $w\to ws_i$ is positive, and $X_wT_i=X_{ws_i}+(\sq_i-\sq_i^{-1})X_w$ if $w\to ws_i$ is negative. Thus, since $\theta^J(ws_i)=\theta^J(us_i)$ (by Lemma~\ref{lem:littlebitsa}) we have
\begin{align*}
\varpi_{J,\sv}(X_wT_i)=\begin{cases}
(\sv\zJ)^{\wt(ws_i)}\sv(\theta_J(ws_i))^{-1}\bm_{\theta^J(us_i)}&\text{if $\epsilon=1$}\\
(\sv\zJ)^{\wt(ws_i)}\sv(\theta_J(ws_i))^{-1}\bm_{\theta^J(us_i)}+(\sq_i-\sq_i^{-1})(\sv\zJ)^{\wt(w)}\sv(\theta_J(w))^{-1}\bm_{u}&\text{if $\epsilon=-1$}
\end{cases}
\end{align*}
where $\epsilon\in\{-1,1\}$ is the sign of the crossing $w\to ws_i$. 

We now compute $\varpi_{J,\sv}(X_w)\cdot T_i$. There are various cases to consider. 
\smallskip

\noindent\textit{Case 1:} Suppose that $u\to us_i$ is positive and $us_i\in\WJ$. Then
\begin{align*}
\varpi_{J,\sv}(X_w)\cdot T_i&=(\sv\zJ)^{\wt(w)}\sv(\theta_J(w))^{-1}\bm_u\cdot T_i=(\sv\zJ)^{\wt(w)}\zJ^{\wt(us_i)}\sv(\theta_J(w))^{-1}\bm_{\theta^J(us_i)}
\end{align*}
By Lemma~\ref{lem:littlebitsb} we have that $w\to ws_i$ is necessarily positive. By Lemma~\ref{lem:littlebitsa} (with $v=s_i$) we have $\wt(ws_i)=\wt(w)+\theta_J(w)\wt(us_i)$, and since $\theta_J(w)\in W_J$ we have $\wt(ws_i)=\wt(w)+\wt(us_i)+\gamma$ for some $\gamma\in Q_J$. Since $\zJ^{\gamma}=1$ it follows that $\varpi_{J,\sv}(X_w)\cdot T_i=\varpi_{J,\sv}(X_wT_i)$ if and only if
$$
\sv^{\wt(w)}\sv(\theta_J(w))^{-1}=\sv^{\wt(ws_i)}\sv(\theta_J(ws_i))^{-1},
$$
and the result follows from Lemma~\ref{lem:bitofmagic}. 
\smallskip

\noindent\textit{Case 2:} Suppose that $u\to us_i$ is negative and $us_i\in\WJ$. Then Lemma~\ref{lem:littlebitsb} gives that $w\to ws_i$ is also negative. We compute 
\begin{align*}
\varpi_{J,\sv}(X_w)\cdot T_i&=(\sv\zJ)^{\wt(w)}\sv(\theta_J(w))^{-1}\bm_u\cdot T_i\\
&=(\sv\zJ)^{\wt(w)}\sv(\theta_J(w))^{-1}\big[\zJ^{\wt(us_i)}\bm_{\theta^J(us_i)}+(\sq_i-\sq_i^{-1})\bm_u\big].
\end{align*}
The same analysis as in Case 1 deals with the $\zJ$ factors, and the result again follows from Lemma~\ref{lem:bitofmagic}. 
\smallskip

\noindent\textit{Case 3:} Suppose that $us_i\notin\WJ$, and so $\theta^J(us_i)=u$. Then we have
\begin{align*}
\varpi_{J,\sv}(X_w)\cdot T_i&=(\sv\zJ)^{\wt(w)}\sv(\theta_J(w))^{-1}\sv_{u\alpha_i}\sv_{2u\alpha_i}\bm_u.
\end{align*}
Since $\theta^J(us_i)=u$, the calculation of $\varpi_{J,\sv}(X_wT_i)$ from above gives
\begin{align*}
\varpi_{J,\sv}(X_wT_i)=\begin{cases}
(\sv\zJ)^{\wt(ws_i)}\sv(\theta_J(ws_i))^{-1}\bm_{u}&\text{if $\epsilon=1$}\\
\big[(\sv\zJ)^{\wt(ws_i)}\sv(\theta_J(ws_i))^{-1}+(\sq_i-\sq_i^{-1})(\sv\zJ)^{\wt(w)}\sv(\theta_J(w))^{-1}\big]\bm_{u}&\text{if $\epsilon=-1$},
\end{cases}
\end{align*}
where, as before, $\epsilon$ is the sign of the crossing $w\to ws_i$. Thus if $\epsilon=1$ then $\varpi_{J,\sv}(X_w)\cdot T_i=\varpi_{J,\sv}(X_wT_i)$ if and only if 
$$
(\sv\zJ)^{\wt(w)}\sv(\theta_J(w))^{-1}\sv_{u\alpha_i}\sv_{2u\alpha_i}=(\sv\zJ)^{\wt(ws_i)}\sv(\theta_J(ws_i))^{-1}
$$
Recall that $\wt(ws_i)=\wt(w)+\theta_J(w)\wt(us_i)$. Moreover, since $us_i\notin\WJ$ we have $u\alpha_i\in \Phi_J+\ZZ\delta$, and hence $\wt(us_i)\in Q_J$. Thus $\zJ^{\wt(ws_i)}=\zJ^{\wt(w)}$, and the result follows from Lemma~\ref{lem:bitofmagic}. 

Finally, if $\epsilon=-1$ then $\varpi_{J,\sv}(X_w)\cdot T_i=\varpi_{J,\sv}(X_wT_i)$ if and only if
$$
\sv^{\wt(w)}\sv(\theta_J(w))^{-1}\sv_{u\alpha_i}\sv_{2u\alpha_i}=\big[\sv^{\wt(ws_i)}\sv(\theta_J(ws_i))^{-1}+(\sq_i-\sq_i^{-1})\sv^{\wt(w)}\sv(\theta_J(w))^{-1}\big]
$$
Rearranging, and noting that $\sv_{u\alpha_i}\sv_{2u\alpha_i}-\sq_i+\sq_i^{-1}=\sv_{u\alpha_i}^{-1}\sv_{2u\alpha_i}^{-1}$, we have $\varpi_{J,\sv}(X_w)\cdot T_i=\varpi_{J,\sv}(X_wT_i)$ if and only if 
$$
\sv^{\wt(w)}\sv(\theta_J(w))^{-1}\sv_{u\alpha_i}^{-1}\sv_{2u\alpha_i}^{-1}=\sv^{\wt(ws_i)}\sv(\theta_J(ws_i))^{-1}
$$
and the result again follows from Lemma~\ref{lem:bitofmagic}. 
\end{proof}

We now give the proof of Theorem~\ref{thm:module}. 

\begin{proof}[Proof of Theorem~\ref{thm:module}]
It is sufficient to check that the relations (for $i,j\in\{0\}\cup\Ifin$, $\sigma,\sigma'\in\Sigma$)
\begin{align*}
\cdots T_iT_jT_i=\cdots T_jT_iT_j,\quad 
T_i^2=1+(\sq_i-\sq_i^{-1})T_i,\quad 
T_{\sigma}T_i=T_{\sigma(i)}T_{\sigma},\quad 
T_{\sigma}T_{\sigma'}=T_{\sigma\sigma'}
\end{align*}
are respected by the proposed action (with $m_{ij}$ terms on either side in the first relation). 
For example, to verify that the braid relation is respected one use Proposition~\ref{prop:module} repeatedly to get
\begin{align*}
\varpi_{J,\sv}(X_uT_iT_jT_i\cdots)=(\cdots (((\bm_u\cdot T_i)\cdot T_j)\cdot T_i)\cdots )
\end{align*}
for all $u\in W^J$, and note that $\varpi_{J,\sv}(X_uT_iT_jT_i\cdots)=\varpi_{J,\sv}(X_uT_jT_iT_j\cdots)$. The remaining relations follow in the same way. 
\end{proof}

\begin{cor}\label{cor:action}
The map $\varpi_{J,\sv}:\Hext\to M_{J,\sv}$ satisfies
$$
\varpi_{J,\sv}(hh')=\varpi_{J,\sv}(h)\cdot h'\quad\text{for all $h,h'\in\Hext$.}
$$
\end{cor}

\begin{proof}
This is immediate from Theorem~\ref{thm:module}.
\end{proof}

\subsection{The multiplicative character $\psi_{J,\sv}$}\label{sec:character}

The \textit{$J$-Levi subalgebra} of $\Hext$ is the subalgebra $\cL_J$ generated by the elements $T_j$, $j\in J$, and $X^{\lambda}$, $\lambda\in P$.  By construction of the module $M_{J,\sv}$ it is clear that $\sR[\zJ]\bm_e$ is stable under the action of $\cL_J$, and thus we may define a map $\psi_{J,\sv}:\cL_J\to \sR[\zJ]$ by
$$
\bm_e\cdot h=\psi_{J,\sv}(h)\bm_e\quad\text{for $h\in \cL_J$}.
$$
Recall the definition of $\st_{\lambda}$ and $\sy_{\lambda}$ from Definition~\ref{defn:utau}.

\begin{prop}\label{prop:multiplicativecharacter}
We have the following.
\begin{compactenum}[$(1)$]
\item The map $\psi_{J,\sv}$ is a multiplicative character (a $1$-dimensional representation) of $\cL_J$.
\item $\psi_{J,\sv}(T_j)=\sv_{\alpha_j}\sv_{2\alpha_j}\in\{\sq_j,-\sq_j^{-1}\}$ for all $j\in J$.
\item $\psi_{J,\sv}(T_y)=\sv(y)$ for all $y\in W_J$.
\item $\psi_{J,\sv}(X^{\lambda})=\sv^{\lambda}\zJ^{\lambda}$ for all $\lambda\in P$.
\item If $2\alpha_j\notin\Phi_J$ then $\psi(X^{\alpha_j^{\vee}})=\psi(T_j)^2$.
\item If $\Phi_J$ is not reduced then $\psi(X^{\alpha_n^{\vee}/2}T_n^{-1})\in\{\sq_0,-\sq_0^{-1}\}$. 
\item $\psi_{J,\sv}(T_{\sy_{\lambda}})=\sv^{\lambda}$ for all $\lambda\in P^{(J)}$.
\item $\psi_{J,\sv}(X_{\st_{\lambda}})=\zJ^{\lambda}$ for all $\lambda\in  P^{(J)}$. 
\end{compactenum}
\end{prop}

\begin{proof}
(1) follows from Corollary~\ref{cor:action}, and (2), (3) and (4) follow directly from the definition of the action. For (5) and (6), if $2\alpha_j\notin\Phi_J$ then the Bernstein-Lusztig relation gives
$
T_jX^{\omega_j}=X^{\omega_j-\alpha_j^{\vee}}T_j+(\sq_j-\sq_j^{-1})X^{\omega_j},
$
and so $T_j^{-1}X^{\omega_j}=X^{\omega_j-\alpha_j^{\vee}}T_j$, from which it follows that $\psi(X^{\alpha_j^{\vee}})=\psi(T_j)^2$. If $\Phi_J$ is not reduced then the Bernstein-Lusztig relation gives
$
T_n^{-1}X^{\omega_n}=X^{\omega_n-\alpha_n^{\vee}}T_n+(\sq_0-\sq_0^{-1})X^{\omega_n-\alpha_n^{\vee}/2},
$
and it follows that $\psi(X^{\alpha_n^{\vee}/2}T_n^{-1})$ satisfies a quadratic relation, giving 
$
\psi(X^{\alpha_n^{\vee}/2}T_n^{-1})\in\{\sq_0,-\sq_0^{-1}\}
$
as required. 

To prove (7) and (8), since $\psi_{J,\sv}(X^{\lambda})=\sv^{\lambda}\zJ^{\lambda}$ it is sufficient to prove that 
$
\psi_{J,\sv}(T_{\sy_{\lambda}})=\sv^{\lambda},
$ and since $\psi_{J,\sv}(T_{\sy_{\lambda}})=\sv(\sy_{\lambda})$ this follows from Lemma~\ref{lem:Jparameter1} (with $y=e$). 
\end{proof}

We also record the following fact. 

\begin{lemma}\label{lem:characterproperty} Let $K\in\cK(J)$ and let $\alpha\in\Phi_K^+$ be a long root of $\Phi_K$ (with all roots long if $\Phi_K$ is simply laced). Then $\psi_{J,\sv}(X^{\alpha^{\vee}}T_{s_{\alpha}}^{-1})=\sv_{\alpha}$. 
\end{lemma}

\begin{proof}
Since $\alpha\in\Phi_J$ we have $\zJ^{\alpha^{\vee}}=1$, and so by Proposition~\ref{prop:multiplicativecharacter}(3) and (4) we have 
$$
\psi_{J,\sv}(X^{\alpha^{\vee}}T_{s_{\alpha}}^{-1})=\sv^{\alpha^{\vee}}\sv(s_{\alpha})^{-1},
$$
and the result follows from Lemma~\ref{lem:Jparameter3}. 
\end{proof}

\subsection{The $J$-affine Hecke algebra}\label{sec:JaffineHecke}

Recall the definition of $\WJaff$ and its Coxeter generators $\{s_j'\mid j\in\Jaff\}$ from Section~\ref{sec:JAffineWeyl}. For $K\in\cK(J)$ and $j\in J$ let 
$$
T_j'=T_j\quad\text{and}\quad T_{0_K}'=X^{\varphi_K^{\vee}}T_{s_{\varphi_K}}^{-1}.
$$
Let $\HJaff$ be the subalgebra of $\Hext$ generated by $\{T_j'\mid j \in\Jaff\}$. By the general theory of Section~\ref{sec:hecke} (in particular~(\ref{eq:T0})) the algebra $\HJaff$ is an affine Hecke algebra of type $\WJaff$, and 
$$
\HJaff=\prod_{K\in\cK(J)}\HKaff
$$
with each $\HKaff$ an affine Hecke algebra of irreducible type~$\WKaff$.

Since $s_{\varphi_K}\in W_J$ for $K\in\cK(J)$ the algebra $\HJaff$ is a subalgebra of the Levi subalgebra~$\cL_J$. Thus if $\sv=(\sv_{\alpha})_{\alpha\in\Phi_J}$ is a $J$-parameter system then $\psi_{J,\sv}$, restricted to $\HJaff$, gives a $1$-dimensional representation of $\HJaff$. By definition we have $\psi_{J,\sv}(T_j')=\sv_{\alpha_j}\sv_{2\alpha_j}$ for all $j\in J$. The following corollary allows us to compute $\psi_{J,\sv}(T_{0_K}')$ for $K\in\cK(J)$.

\begin{cor}\label{cor:psi(T0)}
For $K\in\cK(J)$ we have $\psi_{J,\sv}(T_{0_K}')=\sv_{\varphi_K}=\sv_{\varphi_K}\sv_{2\varphi_K}$.
\end{cor}

\begin{proof}
This follows from Lemma~\ref{lem:characterproperty} and the fact that $\sv_{2\varphi_K}=1$.
\end{proof}

\subsection{The path formula for $\pi_{J,\sv}(T_w)$}

We write $(\pi_{J,\sv},M_{J,\sv})$ for the representation of $\Hext$ in Theorem~\ref{thm:module}. In this section we prove the path formula (Theorem~\ref{thm:mainpath1}) for the representations $(\pi_{J,\sv},M_{J,\sv})$. We will need the following simple lemma.

\begin{lemma}\label{lem:walls}
If a wall $H$ contains panels of type $i$ and $j$ then $\sq_i=\sq_j$. 
\end{lemma}

\begin{proof}
Let $v,w\in\Waff$ be such that $vA_0\cap vs_iA_0$ and $wA_0\cap ws_jA_0$ are panels of $H$ (of types $i,j$ respectively). Applying $v^{-1}$ we may assume that $v=e$ (as $\Waff$ acts in a type preserving way), and replacing $w$ with $ws_j$ if necessary we may assume that $eA_0$ and $wA_0$ lie in the same halfspace determined by~$H$ (equivalently, $\ell(ws_j)=\ell(w)+1$). Then $ws_j=s_iw$ (because these elements lie on the same side of all walls), and so $s_j$ and $s_i$ are conjugate in $\Waff$. Hence $\sq_i=\sq_j$.
\end{proof}

The following theorem is a preliminary step to proving the path formula, and gives a $J$-analogue of Proposition~\ref{prop:basischange}.

\begin{thm}\label{thm:prelim}
For $u\in W^J$ and $w\in \Wext$ we have
\begin{align*}
\varpi_{J,\sv}(X_uT_w)=\sum_{p\in\mathcal{P}_{J}(\vec{w},u)}\cQ_{J,\sv}(p)\varpi_{J,\sv}(X_{\mathrm{end}(p)}),
\end{align*}
where $\vec{w}$ is any reduced expression for $w$.
\end{thm}

\begin{proof}
We argue by induction on $\ell(w)$, with the case $\ell(w)=0$ true by definition of the action $\bm_u\cdot T_{\sigma}$. Suppose that $\ell(ws_k)=\ell(w)+1$. Then by the induction hypothesis
\begin{align*}
\varpi_{J,\sv}(X_uT_{ws_k})&=\varpi_{J,\sv}(X_uT_w)\cdot T_k=\sum_{p\in\mathcal{P}_{J}(\vec{w},u)}\mathcal{Q}_{J,\sv}(p)\varpi_{J,\sv}(X_{\mathrm{end}(p)}T_k).
\end{align*}
Let $p\in\mathcal{P}_{J}(\vec{w},u)$ and write $v=\mathrm{end}(p)$. There are 4 cases to consider.
\smallskip

\noindent\textit{Case 1:} If $vA_0\,{^-}\hspace{-0.1cm}\mid^+\, vs_kA_0$ with $vs_k\in\WJ$ then $\varpi_{J,\sv}(X_vT_k)=\varpi_{J,\sv}(X_{vs_k})$. Writing $p\epsilon_k^+$ for the path obtained from $p$ by appending a positive $s_k$-crossing we have $\cQ_{J,\sv}(p)=\cQ_{J,\sv}(p\epsilon_k^+)$ and $vs_k=\mathrm{end}(p\epsilon_k^+)$, and so
$$
\mathcal{Q}_{J,\sv}(p)\varpi_{J,\sv}(X_{v}T_{k})=\mathcal{Q}_{J,\sv}(p\epsilon_k^+)\varpi_{J,\sv}(X_{\mathrm{end}(p\epsilon_k^+)}).
$$ 

\noindent\textit{Case 2:}  If $vA_0\,{^+}\hspace{-0.1cm}\mid^-\, vs_kA_0$ with $vs_k\in\WJ$ then using $T_k=T_k^{-1}+	(\sq_k-\sq_k^{-1})$ gives
$$
\mathcal{Q}_{J,\sv}(p)\varpi_{J,\sv}(X_{v}T_k)=\mathcal{Q}_{J,\sv}(p\epsilon_k^-)\varpi_{J,\sv}(X_{\mathrm{end}(p \epsilon^-_k)})+\mathcal{Q}_{J,\sv}(pf_k)\varpi_{J,\sv}(X_{\mathrm{end}(pf_k)}),
$$ 
where $p\epsilon_k^-$ denotes the path obtained from $p$ by appending a negative $s_k$-crossing and $pf_k$ denotes the path obtained from $p$ by appending an $s_k$-fold.
\smallskip

\noindent\textit{Case 3:} If $vA_0\,{^-}\hspace{-0.1cm}\mid^+\, vs_kA_0$ with $vs_k\notin\WJ$ then by Lemma~\ref{lem:boundingwalls} the panel $vA_0\cap vs_kA_0$ is contained in $H_{\varphi_{K},1}$ for some $K\in \cK(J)$. Then $vs_k=s_{\varphi_{K},1}v$, and since $s_{\varphi_K,1}=t_{\varphi_K^{\vee}}s_{\varphi_K}$ we have 
$$
X_vT_k=X_{vs_k}=X_{t_{\varphi_K^{\vee}}s_{\varphi_K}v}=X^{\varphi_K^{\vee}}X_{s_{\varphi_K}v}.
$$
We claim that $X_{s_{\varphi_K}v}=T_{s_{\varphi_K}}^{-1}X_v$. To see this, let $\vec{v}$ be a path of reduced type from $e$ to $v$, and consider the reflected path $s_{\varphi_K}(\vec{v})$ (a path joining $s_{\varphi_K}$ to $s_{\varphi_K}v$). Since $v\in \WJ$ no reduced path from $e$ to $v$ crosses any hyperplanes parallel to any wall $H_{\alpha}$ with $\alpha\in\Phi_J^+$, and since $\Phi(s_{\varphi_K})\subseteq \Phi_K^+\subseteq \Phi_J^+$ it follows that each positive (respectively negative) crossing in $\vec{v}$ is mapped to a positive (respectively negative) crossing in $s_{\varphi_K}(\vec{v})$. Thus $X_{s_{\varphi_K}v}=X_{s_{\varphi_K}}X_v$, and since $s_{\varphi_K}\in W_0$ we have $X_{s_{\varphi_K}}=T_{s_{\varphi_K}}^{-1}$, and hence the claim.

Thus 
$
X_vT_k=X^{\varphi_K^{\vee}}T_{s_{\varphi_K}}^{-1}X_v$, and since $\varphi_K$ is long in $\Phi_K$ Lemma~\ref{lem:characterproperty}, and Proposition~\ref{prop:multiplicativecharacter}, gives
\begin{align*}
\mathcal{Q}_{J,\sv}(p)\varpi_{J,\sv}(X_{v}T_k)&=\cQ_{J,\sv}(p)\varpi_{J,\sv}(X^{\varphi_K^{\vee}}T_{s_{\varphi_K}}^{-1}X_v)\\
&=\cQ_{J,\sv}(p)\psi_{J,\sv}(X^{\varphi_K^{\vee}}T_{s_{\varphi_K}}^{-1})\varpi_{J,\sv}(X_v)\\
&=\cQ_{J,\sv}(p)\sv_{\varphi_K}\varpi_{J,\sv}(X_v)\\
&=\cQ_{J,\sv}(pb_{\varphi_K}^-)\varpi_{J,\sv}(X_{\mathrm{end}(pb_{\varphi_K}^-)}),
\end{align*}
where $pb_{\varphi_K}^-$ denotes the path obtained from $p$ by appending a negative bounce on the wall~$H_{\varphi_K;1}$ (note also that $\sv_{\varphi_K}=\sv_{\varphi_K}\sv_{2\varphi_K}$ to calculate $\cQ_{J,\sv}(pb_{\varphi_K}^-)$).
\smallskip

\noindent\textit{Case 4:} If $vA_0\,{^+}\hspace{-0.1cm}\mid^-\, vs_kA_0$ with $vs_k\notin\WJ$ then by Lemma~\ref{lem:boundingwalls} the panel $vA_0\cap vs_kA_0$ is contained in $H_{\alpha_j,0}$ for some $j\in J$. Using the formula $T_k=T_k^{-1}+(\sq_k-\sq_k^{-1})$ we have
\begin{align*}
\varpi_{J,\sv}(X_vT_k)&=\varpi_{J,\sv}(X_{vs_k})+(\sq_k-\sq_k^{-1})\varpi_{J,\sv}(X_v).
\end{align*}
Since $vs_k=s_jv$ we have $\varpi_{J,\sv}(X_vT_k)=\varpi_{J,\sv}(X_{s_jv})+(\sq_k-\sq_k^{-1})\varpi_{J,\sv}(X_v)$. Similar arguments to Case 3 shows that $X_{s_jv}=T_{s_j}^{-1}X_v$, and hence 
\begin{align*}
\varpi_{J,\sv}(X_vT_k)&=(\psi_{J,\sv}(T_j)^{-1}+\sq_k-\sq_k^{-1})\varpi_{J,\sv}(X_{v}).
\end{align*}
By Lemma~\ref{lem:walls} we have $\sq_k=\sq_j$, and so $\psi_{J,\sv}(T_j^{-1})+\sq_k-\sq_k^{-1}=\psi_{J,\sv}(T_j^{-1})+\sq_j-\sq_j^{-1}=\psi_{J,\sv}(T_j)$. If either $\Phi_J$ is reduced, or $\Phi_J$ is not reduced and $j\neq n$, then $\psi_{J,\sv}(T_j)=\sv_{\alpha_j}$. If $\Phi_J$ is not reduced and $j=n$ then $\psi_{J,\sv}(T_n)=\sv_{\alpha_n}\sv_{2\alpha_n}$. In all cases we have $\cQ_{J,\sv}(p)\psi_{J,\sv}(T_j)=\cQ_{J,\sv}(pb_{\alpha_j}^+)$ where $pb_{\alpha_j}^{+}$ denotes the path obtained from $p$ by appending a positive bounce on the wall~$H_{\alpha_j,0}$. Thus we have
$$
\cQ_{J,\sv}(p)\varpi_{J,\sv}(X_vT_k)=\cQ_{J,\sv}(p)\psi_{J,\sv}(T_j)\varpi_{J,\sv}(X_v)=\cQ_{J,\sv}(pb_{\alpha_j}^+)\varpi_{J,\sv}(X_{\mathrm{end}(pb_{\alpha_j}^+)}).
$$
Hence the result.
\end{proof}

We now prove the path formula for the matrix entries of $\pi_{J,\sv}(T_w)$ with respect to the basis $\sB_{J,\sv}=\{\bm_u\mid u\in W^J\}$ of $M_{J,\sv}$. A version of this formula for more general bases will be given in Theorem~\ref{thm:mainpath2}.

\begin{thm}\label{thm:mainpath1}
Let $w\in \Wext$. The matrix entries of $\pi_{J,\sv}(T_w)$ with respect to the basis $\sB_{J,\sv}$ are
\begin{align*}
[\pi_{J,\sv}(T_w)]_{u,v}=\sum_{\{p\in\mathcal{P}_J(\vec{w},u)\,\mid\,\theta^J(p)=v\}}\mathcal{Q}_{J,\sv}(p)\zJ^{\wt(p)}\quad\text{for $u,v\in W^J$},
\end{align*}
where $\vec{w}$ is any choice of reduced expression for~$w$.
\end{thm}

\begin{proof}
Writing $\lambda=\wt(p)$ we have, by Corollary~\ref{cor:splitting2} and~(\ref{eq:splitting}),
$$
X_{\mathrm{end}(p)}=X^{\lambda}T_{\sy_{\lambda}^{-1}}^{-1}T_{\theta^{J}(p)^{-1}}^{-1}.
$$
Then by Theorem~\ref{thm:prelim} and Proposition~\ref{prop:multiplicativecharacter}
\begin{align*}
\varpi_{J,\sv}(X_u)\cdot T_w&=\sum_{p\in\mathcal{P}_{J}(\vec{w},u)}\cQ_{J,\sv}(p)\varpi_{J,\sv}(X^{\wt(p)}T_{\sy_{\wt(p)}^{-1}}^{-1}T_{\theta^{J}(p)^{-1}}^{-1})\\
&=\sum_{p\in\mathcal{P}_{J}(\vec{w},u)}\cQ_{J,\sv}(p)\psi_{J,\sv}(X^{\wt(p)}T_{\sy_{\wt(p)}^{-1}}^{-1})\varpi_{J,\sv}(T_{\theta^{J}(p)^{-1}}^{-1})\\
&=\sum_{p\in\mathcal{P}_{J}(\vec{w},u)}\cQ_{J,\sv}(p)\zJ^{\wt(p)}\varpi_{J,\sv}(T_{\theta^{J}(p)^{-1}}^{-1}),
\end{align*}
completing the proof since $\varpi_{J,\sv}(T_{\theta^{J}(p)^{-1}}^{-1}) = \bm_{\theta^J(p)}$.
\end{proof}

\subsection{Changing fundamental domains}

It will be useful to have a more general version of Theorem~\ref{thm:mainpath1} adapted to other choices of basis (for example, see Section~\ref{sec:An}). 

\begin{prop}\label{prop:basis1}
If $\sF$ is a fundamental domain for the action of $\TJ$ on $\WJ$ then 
$$
\sB_{\sF}=\{\varpi_{J,\sv}(X_u)\mid u\in \sF\}
$$
is a basis of $M_{J,\sv}$.
\end{prop}

\begin{proof}
By construction of the module $M_{J,\sv}$ the result holds for the fundamental domain $\sF=W^J$. Now let $\sF$ be an arbitrary fundamental domain. For $u\in \sF$ define $u'\in W^J$ and $\lambda\in  P^{(J)}$ by $u=\st_{\lambda}u'$ (so $\lambda=\wt(u)$ and $u'=\theta^J(u)$). Thus by Proposition~\ref{prop:multiplicativecharacter} we have 
\begin{align*}
\varpi_{J,\sv}(X_u)&=\varpi_{J,\sv}(X_{\st_{\lambda}}X_{u'})=\bm_e\cdot X_{\st_{\lambda}}X_{u'}=\zJ^{\lambda}\bm_e\cdot X_{u'}=\zJ^{\lambda}\varpi_{J,\sv}(X_{u'}),
\end{align*}
and hence the result.
\end{proof}

Note that $\sB_{W^J}=\sB_{J,\sv}$. A version of Theorem~\ref{thm:mainpath1} for arbitrary fundamental domains is given below. Recall the definition of $\wt(p,\sF)$ and $\theta(p,\sF)$ from~(\ref{eq:generalweights}).

\begin{thm}\label{thm:mainpath2}
Let $\sF$ be a fundamental domain for the action of $\TJ$ on $\WJ$. With respect to the basis $\sB_{\sF}$ of $M_{J,\sv}$ from Proposition~\ref{prop:basis1}, the matrix entries of $\pi_{J,\sv}(T_w)$, with $w\in \Wext$, are
\begin{align*}
[\pi_{J,\sv}(T_w)]_{u,v}=\sum_{\{p\in\mathcal{P}_J(\vec{w},u)\,\mid\,\theta(p,\sF)=v\}}\mathcal{Q}_{J,\sv}(p)\zJ^{\wt(p,\sF)}\quad\text{for $u,v\in\sF$},
\end{align*}
where $\vec{w}$ is any choice of reduced expression for~$w$.
\end{thm}

\begin{proof}
Using the fact that $\sF$ is a fundamental domain, define functions $g:\sF\to W^J$ and $h:\sF\to  P^{(J)}$ by the equation 
$u=\st_{h(u)}g(u)$ for $u\in \sF$. Then, as in the proof of Proposition~\ref{prop:basis1}, we have 
$
\varpi_{J,\sv}(X_u)=\zJ^{h(u)}\varpi_{J,\sv}(X_{g(u)}),
$
and by changing basis from the fundamental domain $W^J$ case (proved in Theorem~\ref{thm:mainpath1}) we have
$$
[\pi_{J,\sv}(T_w)]_{u,v}=\sum_{\{p\in\cP_J(\vec{w},g(u))\,\mid\, \theta(p,W^J)=g(v)\}}\cQ_{J,\sv}(p)\zJ^{\wt(p)+h(u)-h(v)}.
$$
Using Lemma~\ref{lem:actiononpaths} and Lemma~\ref{lem:preserveQ} it follows that 
\begin{align*}
[\pi_{J,\sv}(T_w)]_{u,v}&=\sum_{\{p\in\cP_J(\vec{w},g(u))\,\mid\, \theta(p,W^J)=g(v)\}}\cQ_{J,\sv}(\st_{h(u)}\cdot p)\zJ^{\wt(\st_{h(u)}\cdot p)-h(v)}\\
&=\sum_{\{p\in\cP_J(\vec{w},u)\,\mid\, \theta(p,\sF)=v\}}\cQ_{J,\sv}(p)\zJ^{\wt(p)-h(v)},
\end{align*}
and the result follows since $\wt(p,\sF)=\wt(p)-h(v)$ if $\theta(p,\sF)=v$. 
\end{proof}

\subsection{Intertwiners and irreducibility}

For $i\in\Ifin$ define \textit{intertwiners} $U_i\in\Hext$ by
$$
U_i=\begin{cases}
(1-X^{-\alpha_i^{\vee}})T_i-(\sq_i-\sq_i^{-1})&\text{if $2\alpha_i\notin\Phi$}\\
(1-X^{-\alpha_n^{\vee}})T_n-(\sq_n-\sq_n^{-1}+(\sq_0-\sq_0^{-1})X^{-\alpha_n^{\vee}/2})&\text{if $\Phi=\sBC_n$ and $i=n$ and $\sq_0\neq\sq_n$}\\
(1-X^{-\alpha_n^{\vee}/2})T_n-(\sq_n-\sq_n^{-1})&\text{if $\Phi=\sBC_n$ and $i=n$ and $\sq_0=\sq_n$}
\end{cases}
$$
The terminology comes from the fact that these elements ``intertwine'' the weight spaces of $\Hext$-modules (see Proposition~\ref{prop:weightspaces}(1)). A direct calculation, using the Bernstein-Lusztig relation, gives
\begin{align}\label{eq:tau2}
U_i^2&=\sq_i^2(1-\sq_i^{-2}X^{-\alpha_i^{\vee}})(1-\sq_i^{-2}X^{\alpha_i^{\vee}})&&\text{if $2\alpha_i\notin\Phi$},
\end{align}
while if $\Phi=\sBC_n$ and $i=n$ then if $\sq_0\neq \sq_n$ we have
$$
U_n^2=\sq_n^2(1-\sq_0^{-1}\sq_n^{-1}X^{-\alpha_n^{\vee}/2})(1+\sq_0\sq_n^{-1}X^{-\alpha_n^{\vee}/2})(1-\sq_0^{-1}\sq_n^{-1}X^{\alpha_n^{\vee}/2})(1+\sq_0\sq_n^{-1}X^{\alpha_n^{\vee}/2})
$$
and if $\sq_0=\sq_n$ then 
$
U_n^2=\sq_n^2(1-\sq_n^{-2}X^{-\alpha_n^{\vee}/2})(1-\sq_n^{-2}X^{\alpha_n^{\vee}/2}).
$

The elements $U_i$ satisfy the braid relations (see \cite[Proposition~2.14]{Ram:03}, however note that we have normalised these elements so that they are elements of the Hecke algebra), and hence for $w\in W_0$ we may define
$$
U_w=U_{i_1}\cdots U_{i_{\ell}}
$$
whenever $w=s_{i_1}\cdots s_{i_{\ell}}$ is a reduced expression. From the Bernstein-Lusztig relation we have $U_iX^{\lambda}=X^{s_i\lambda}U_i$, and hence we have the very useful relation
$$
U_wX^{\lambda}=X^{w\lambda}U_w\quad\text{for all $w\in W_0$ and $\lambda\in P$}.
$$

\begin{lemma}\label{lem:nonvanish}
If $\alpha\in\Phi\backslash \Phi_J$ then $\psi_{J,\sv}(X^{\alpha^{\vee}})\notin \sR$. 
\end{lemma}

\begin{proof}
Let $\alpha\in \Phi$. If $\sR\ni\psi_{J,\sv}(X^{\alpha^{\vee}})=\sv^{\lambda}\zJ^{\alpha^{\vee}}$ then $\zJ^{\alpha^{\vee}}=1$. Thus $(\alpha^{\vee})^J=0$, and so $\alpha^{\vee}\in V_J$, and so $\alpha\in\Phi_J$. 
\end{proof}

\begin{prop}\label{prop:newbasis}
The module $M_{J,\sv}$ has basis 
$
\{\varpi_{J,\sv}(U_u)\mid u\in W^J\}.
$
\end{prop}

\begin{proof}
Let $u=s_{i_1}\cdots s_{i_{\ell}}\in W_0$ be a reduced expression. From the Bernstein-Lusztig relation we have
\begin{align*}
U_u&=(1-X^{-\alpha_{i_1}^{\vee}})T_{i_1}(1-X^{-\alpha_{i_2}^{\vee}})T_{i_2}\cdots (1-X^{-\alpha_{i_{\ell}}^{\vee}})T_{i_{\ell}}+\text{lower terms}\\
&=(1-X^{-\alpha_{i_1}^{\vee}})(1-X^{-s_{i_1}\alpha_{i_2}^{\vee}})\cdots (1-X^{-s_{i_1}\cdots s_{i_{\ell-1}}\alpha_{i_{\ell}}^{\vee}})T_u+\text{lower terms}\\
&=\bigg[\prod_{\alpha\in\Phi(u)}(1-X^{-\alpha^{\vee}})\bigg]T_u+\text{lower terms},
\end{align*}
where on each line ``lower terms'' denotes a linear combination of terms $p_v(X)T_v$ with $v<u$ (in Bruhat order). Since $T_u=X_u+\text{(terms $X_v$ with $v<u$)}$ and $\varpi_{J,\sv}(X_u)=\bm_u$ for $u\in W^J$ we have
$$
\varpi_{J,\sv}(U_u)=\bigg[\prod_{\alpha\in\Phi(u)}(1-\psi_{J,\sv}(X^{\alpha^{\vee}})^{-1})\bigg]\bm_u+\text{lower terms}.
$$
For $u\in W^J$ we have $\Phi_J(u)=\emptyset$ (by Lemma~\ref{lem:decomposition}), and so by Lemma~\ref{lem:nonvanish} the coefficient
$$
\prod_{\alpha\in\Phi(u)}(1-\psi_{J,\sv}(X^{-\alpha^{\vee}}))
$$
does not vanish, and the result follows. 
\end{proof}

The following proposition gives the decomposition of $M_{J,\sv}$ into weight spaces.

\begin{prop}\label{prop:weightspaces}
Let $J\subseteq\Ifin$ and let $\sv$ be a $J$-parameter system. For $u\in W_0$ let 
$$
M_u=\{\bm\in M_{J,\sv}\mid \bm\cdot X^{\lambda}=\psi_{J,\sv}(X^{u\lambda})\bm\text{ for all $\lambda\in P$}\}.
$$
\begin{compactenum}[$(1)$]
\item If $u,us_i\in W^J$ then the map $\tilde{U}_i:M_u\to M_{us_i}$ with $\tilde{U}_i(\bm)=\bm\cdot U_i$ is bijective. 
\item For $u\in W^J$ we have $M_u=\{r\varpi_{J,\sv}(U_u)\mid r\in \sR[\zJ]\}$, and
$$
M_{J,\sv}=\bigoplus_{u\in W^J}M_u.
$$
\end{compactenum}
\end{prop}

\begin{proof}
(1) For $u\in W_0$ let $z_u:P\to \sR[\zJ]$ be the map $z_u^{\lambda}=\psi_{J,\sv}(X^{u\lambda})$. Then $z_u^{\lambda}=z^{u\lambda}$, where we write $z=z_e$, and 
$
M_u=\{\bm\in M_{J,\sv}\mid \bm\cdot X^{\lambda}=z_u^{\lambda}\bm\text{ for all $\lambda\in P$}\}.
$ For any $u\in W_0$ and $i\in\Ifin$, if $\bm\in M_u$ then 
$$
(\bm\cdot U_i)X^{\lambda}=(\bm\cdot X^{s_i\lambda})U_i=z_u^{s_i\lambda}(\bm\cdot U_i)=z_{us_i}^{\lambda}(\bm\cdot U_i),
$$
and so $\bm\cdot U_i\in M_{us_i}$. It follows that there are operators $\tilde{U}_i:M_u\to M_{us_i}$ and $\tilde{U}_i:M_{us_i}\to M_{u}$ given by $\tilde{U}_i(\bm)=\bm\cdot U_i$. Thus $\tilde{U}_i^2:M_u\to M_u$. If $2\alpha_i\notin\Phi$ then by (\ref{eq:tau2}) we have 
$$
\tilde{U}_i^2(\bm)=\bm\cdot U_i^2=\sq_i^2(1-\sq_i^{-2}z^{-u\alpha_i^{\vee}})(1-\sq_i^{-2}z^{u\alpha_i^{\vee}})\bm
$$
and so if $z^{u\alpha_i^{\vee}}\neq \sq_i^{\pm 2}$ then the operators $\tilde{U}_i:M_u\to M_{us_i}$ and $\tilde{U}_i:M_{us_i}\to M_u$ are bijective. If $\Phi$ is not reduced and $i=n$ then the same result holds for $\tilde{U}_n$ provided $z^{u\alpha_n^{\vee}/2}\neq (\sq_0\sq_n)^{\pm 1},-(\sq_0^{-1}\sq_n)^{\pm 1}$ (with the $-(\sq_0^{-1}\sq_n)^{\pm 1}$ case omitted if $\sq_0=\sq_n$). 

Suppose now that $u\in W^J$ and $i\in\Ifin$ with $us_i\in W^J$. Then $u\alpha_i\notin\Phi_J$, and thus Lemma~\ref{lem:nonvanish} gives $z^{u\alpha_i^{\vee}}\notin \sR$ (and also $z^{u\alpha_n^{\vee}/2}\notin\sR$ in the non reduced case) and so by the previous paragraph $\tilde{U}_i:M_u\to M_{us_i}$ is bijective, proving (1).

(2) It is clear that $\{r\varpi_{J,\sv}(U_u)\mid r\in\sR[\zJ]\}\subseteq M_u$ for each $u\in W^J$, and thus by Proposition~\ref{prop:newbasis} the spaces $M_u$, $u\in W^J$, span $M_{J,\sv}$. Thus to prove (2) it is sufficient to show that if $u_1,u_2\in W^J$ with $M_{u_1}=M_{u_2}$ then $u_1= u_2$. To see this, $M_{u_1}=M_{u_2}$ implies that $\psi_{J,\sv}(X^{u_1\lambda})=\psi_{J,\sv}(X^{u_2\lambda})$ for all $\lambda\in P$. Replacing $\lambda$ by $u_2^{-1}\lambda$ we have $\psi_{J,\sv}(X^{u_1u_2^{-1}\lambda-\lambda})=1$ for all $\lambda\in P$. Thus $(u_1u_2^{-1}\lambda-\lambda)^J=0$ for all $\lambda\in P$, and so $u_1u_2^{-1}\lambda\in \lambda+V_J$ for all $\lambda\in P$. Applying this to $\lambda=\alpha^{\vee}\in\Phi_J^{\vee}$ it follows that $u_1u_2^{-1}\alpha^{\vee}\in\Phi_J^{\vee}$ for all $\alpha^{\vee}\in \Phi_J^{\vee}$, and hence $u_1u_2^{-1}\in W_J$. So $u_1\in W_Ju_2$, forcing $u_1=u_2$ (as $u_1,u_2\in W^J$). 
\end{proof}
 
\begin{cor}\label{cor:irreducible}
Let $J\subseteq\Ifin$ and let $\sv$ be a $J$-parameter system. The representation $(\pi_{J,\sv},M_{J,\sv})$ is irreducible.
\end{cor}

\begin{proof}
Let $N$ be a nonzero $\Hext$-invariant submodule of $M_{J,\sv}$. Since $N$ is invariant under the action of the elements $X^{\lambda}$, $\lambda\in P$, it follows from Proposition~\ref{prop:weightspaces}(2) that $N\cap M_u\neq\emptyset$ for some $u\in W^J$. Since $M_u$ is $1$-dimensional we have $M_u\subseteq N$. But then Proposition~\ref{prop:weightspaces}(1), along with $\Hext$-invariance, forces $N=M_{J,\sv}$. 
\end{proof}

\subsection{Generic induced representations}

In this subsection we realise the combinatorial modules $M_{J,\sv}$ introduced in Theorem~\ref{thm:module} as induced representations from $1$-dimensional representations of a Levi subalgebra. Let $\psi_{J,\sv}$ be as in Section~\ref{sec:character}, and let $\xi_{J,\sv}$ be a generator of the $1$-dimensional $\cL_J$ module $\sR[\zJ]\xi_{J,\sv}$ affording the character $\psi_{J,\sv}$. That is,
$$
\xi_{J,\sv}\cdot h=\psi_{J,\sv}(h)\xi_{J,\sv}\quad\text{for all $h\in \cL_J$}.
$$
Let $M_{J,\sv}'=\mathrm{Ind}_{\cL_J}^{\Hext}(\psi_{J,\sv})=(\sR[\zJ]\xi_{J,\sv})\otimes_{\cL_J}\Hext$. 

\begin{prop}\label{prop:basis0}
The module $M_{J,\sv}'$ has basis 
$\{\xi_{J,\sv}\otimes X_u\mid u\in W^J\}$.
\end{prop}

\begin{proof}
Since $\{X_w\mid w\in\Wext\}$ is a basis of $\Hext$ the set $\{\xi_{J,\sv}\otimes X_w\mid w\in\Wext\}$ spans $M_{J,\sv}'$. If $w\in\Wext$ then by~(\ref{eq:splitting}) we have $X_w=X^{\lambda}T_{u^{-1}}^{-1}$ where $\lambda=\wt(w)$ and $u=\theta(w)$. Write $u=u_1u_2\in W_JW^J$. Since $X^{\lambda}T_{u_1^{-1}}^{-1}\in \cL_J$ we have
$$
\xi_{J,\sv}\otimes X_w=\xi_{J,\sv}\otimes X^{\lambda}T_{u_1^{-1}}^{-1}T_{u_2^{-1}}^{-1}=\psi_{J,\sv}(X^{\lambda}T_{u_1^{-1}}^{-1})(\xi_{J,\sv}\otimes X_{u_2}).
$$
Thus $M_{J,\sv}'$ is spanned by $\{\xi_{J,\sv}\otimes X_u\mid u\in W^J\}$, and these elements are linearly independent because $\{X^{\lambda}T_{v^{-1}}^{-1}T_{u^{-1}}^{-1}\mid \lambda\in P,\,v\in W_J,\,u\in W^J\}$ is a basis of $\Hext$.
\end{proof}

The following theorem shows that $(\pi_{J,\sv},M_{J,\sv})$ is isomorphic to the representation $(\pi_{J,\sv}',M_{J,\sv}')$, and moreover identifies bases of each module giving an isomorphism of matrix representations. 

\begin{thm}\label{thm:induced}
We have $M_{J,\sv}\cong M_{J,\sv}'$. Moreover 
$$[\pi_{J,\sv}(h)]_{u,v}=[\pi_{J,\sv}'(h)]_{u,v}\quad\text{for all $h\in\Hext$ and $u,v\in W^J$},$$
where $M_{J,\sv}$ and $M_{J,\sv}'$ are endowed with the bases $\{\bm_u\mid u\in W^J\}$ and $\{\xi_{J,\sv}\otimes X_u\mid u\in W^J\}$ respectively.
\end{thm}

\begin{proof}
We will show that the action of $T_i$ ($i\in\{0\}\cup\Ifin$) and $T_{\sigma}$ ($\sigma\in\Sigma$) on $M_{J,\sv}'$ with respect to the basis $\{\xi_{J,\sv}\otimes X_u\mid u\in W^J\}$ agrees with the action from Theorem~\ref{thm:module}. The analysis is easy for the cases $T_i$ with $i\in \Ifin$ and $T_{\sigma}$ with $\sigma\in\Sigma$, and so we focus on the action of $T_0$. 
\medskip

\noindent\textit{Case 1:} Suppose that $uA_0\,{^-}\hspace{-0.1cm}\mid^+\, us_0A_0$ with $us_0\in\WJ$. Then $X_uT_0=X_{us_0}$, and since $us_0\in\WJ$ Corollary~\ref{cor:splitting2} gives $us_0=\st_{\lambda}\theta^J(uw_0)$ where $\lambda=\wt(us_0)=u\varphi^{\vee}$. Then 
\begin{align*}
(\xi_{J,\sv}\otimes X_u)\cdot T_0&=\xi_{J,\sv}\otimes X_{us_0}=\xi_{J,\sv}\otimes X_{\st_{u\varphi^{\vee}}}X_{\theta^J(us_0)}=\psi_{J,\sv}(X_{\st_{u\varphi^{\vee}}})(\xi_{J,\sv}\otimes X_{\theta^J(us_0)}),
\end{align*}
and Proposition~\ref{prop:multiplicativecharacter}(8) gives $(\xi_{J,\sv}\otimes X_u)\cdot T_0=\zJ^{u\varphi^{\vee}}(\xi_{J,\sv}\otimes X_{\theta^J(us_0)})$, and hence the result in this case. 
\smallskip

\noindent\textit{Case 2:} Suppose that $uA_0\,{^+}\hspace{-0.1cm}\mid^-\, us_0A_0$ with $us_0\in\WJ$. Since $T_0=T_0^{-1}+(\sq_0-\sq_0^{-1})$ we have
\begin{align*}
(\xi_{J,\sv}\otimes X_u)\cdot T_0&=(\xi_{J,\sv}\otimes X_{us_0})+(\sq_0-\sq_0^{-1})(\xi_{J,\sv}\otimes X_u)\\
&=\zJ^{u\varphi^{\vee}}(\xi_{J,\sv}\otimes X_{\theta^J(us_0)})+(\sq_0-\sq_0^{-1})(\xi_{J,\sv}\otimes X_u),
\end{align*}
and hence the result in this case. 
\smallskip

\noindent\textit{Case 3:} If $uA_0\,{^-}\hspace{-0.1cm}\mid^+\, us_0A_0$ with $us_0\notin\WJ$ then, exactly as in Case 3 of the proof of Theorem~\ref{thm:prelim}, we have
$$
X_uT_0=X^{\varphi_K^{\vee}}T_{s_{\varphi_K}}^{-1}X_u
$$
for some $K\in\cK(J)$. Since $\varphi_K$ is long in $\Phi_K$ Lemma~\ref{lem:characterproperty} gives
\begin{align*}
(\xi_{J,\sv}\otimes X_u)\cdot T_0&=\xi_{J,\sv}\otimes X^{\varphi_K^{\vee}}T_{s_{\varphi_K}}^{-1}X_u=\psi_{J,\sv}(X^{\varphi_K^{\vee}}T_{s_{\varphi_K}}^{-1})(\xi_{J,\sv}\otimes X_u)=\sv_{\alpha}\sv_{2\alpha}(\xi_{J,\sv}\otimes X_u),
\end{align*}
as required.
\smallskip

\noindent\textit{Case 4:} If $uA_0\,{^+}\hspace{-0.1cm}\mid^-\, us_0A_0$ with $us_0\notin\WJ$  then by Lemma~\ref{lem:boundingwalls} the panel $uA_0\cap us_0A_0$ is contained in $H_{\alpha_j,0}$ for some $j\in J$. Thus this case is impossible, as $u\alpha_0=-u\varphi+\delta\neq\pm\alpha_i$ for any $i\in \Ifin$.
\end{proof}

\begin{remark}
If $J=\emptyset$ then $\sv$ is vacuous, and $(\pi_{\emptyset,\sv}',M_{\emptyset,\sv}')$ is the \textit{principal series representation} of $\Hext$ with central character~$\zeta=\zeta_I$. 
\end{remark}

\begin{remark}\label{rem:induced}
Up to specialising the ``variables'' $\zeta_i$ and extending scalars, all representations of $\Hext$ induced from $1$-dimensional representations of a Levi subalgebra can be realised by the construction in Theorem~\ref{thm:module}. Let us briefly explain this. Let $\psi:\cL_J\to \sR'$ be a one dimensional representation of $\cL_J$ over an integral domain $\sR'$ containing~$\sR$. For $\alpha\in\Phi_J$ let
$$
\sv_{\alpha}=\begin{cases}
\psi(T_j)&\text{if $\alpha\in W_J\alpha_j$ with $j\in J$ and $\alpha\in\Phi_0\cap\Phi_1$}\\
\psi(X^{\alpha_n^{\vee}/2})\psi(T_n)^{-1}&\text{if $\Phi_J$ not reduced and $\alpha\in W_J(2\alpha_n)$}\\
\psi(X^{\alpha_n^{\vee}/2})^{-1}\psi(T_n)^2&\text{if $\Phi_J$ is not reduced and $\alpha\in W_J\alpha_n$,}
\end{cases}
$$
and define $z:P\to \sR'$ by 
$
z^{\lambda}=\sv^{-\lambda}\psi(X^{\lambda}),
$ with $\sv^{\lambda}$ as in~(\ref{eq:vlambda}). It is not difficult to see that $\sv$ is a $J$-parameter system and that $z^{\gamma}=1$ for all $\gamma\in Q_J$. 
 
After extending the ring $\sR'$ to a ring $\sR''$ if necessary one can choose a specialisation $\zeta_i\mapsto z_i\in \sR''$ such that $\zJ^{\lambda}\mapsto z^{\lambda}$ for all $\lambda\in P$. Then, by Theorem~\ref{thm:induced}, the combinatorial module $M_{J,\sv}$ constructed in Theorem~\ref{thm:module} specialises to $\mathrm{Ind}_{\cL_J}^{\Hext}(\psi)$ (after extending scalars if necessary).
\end{remark}

\section{Bounded representations}\label{sec:boundedreps}

Let $L:\Waff\to \mathbb{N}$ be a \textit{positive weight function} on~$\Waff$. That is, $L(w)>0$ for $w\in \Waff\backslash\{e\}$ and $L(uv)=L(u)+L(v)$ whenever $\ell(uv)=\ell(u)+\ell(v)$. Let $\sq$ be an invertible indeterminate, and specialise $\sq_i=\sq^{L(s_i)}$ for $i\in\{0\}\cup \Ifin$. Without loss of generality we may assume that $\mathrm{gcd}(L(s_0),\ldots,L(s_n))=1$, and thus the ring $\sR$ specialises to $\mathbb{Z}[\sq,\sq^{-1}]$, and one can consider the associated \textit{weighted} Hecke algebras $\Haff(L)$ and $\Hext(L)$ over the ring $\ZZ[\sq,\sq^{-1}]$. In this context a $J$-parameter system becomes a \textit{weighted $J$-parameter system}, where each occurrence of $\sq_j$ in Definition~\ref{defn:Jparameters} (with $j\in \{0\}\cup J$) is replaced by $\sq^{L(s_j)}$.

 It is convenient to make the following convention throughout this section (by the natural symmetry present in $\sBC_n$ this convention does not result in a loss of generality). 

\begin{convention}\label{conv:sym}
If $\Phi$ is of type $\sBC_n$ we assume that $L(s_n)\geq L(s_0)$.
\end{convention}

\subsection{Boundedness}

In~\cite{GP:19,GP:19b} the first and third authors introduced the notion of a \textit{balanced system of cell representations}, inspired by the work of Geck~\cite{Geck:02,Geck:11} in the finite case, and we used this notion to prove Lusztig's Conjectures $\mathbf{P1}$--$\mathbf{P15}$ for rank~$3$ affine Coxeter systems with arbitrary positive weight function~$L$. A key part of this concept was the property of \textit{boundedness} of a matrix representation. We recall this theory here.

Every nonzero rational function $f(\sq)=a(\sq)/b(\sq)$ in $\sq$ can be written under the form $f(\sq)=\sq^{N}a'(\sq^{-1})/b'(\sq^{-1})$ with $N\in\ZZ$ and with $a'(\sq^{-1})$ and $b'(\sq^{-1})$ polynomials in $\sq^{-1}$ with nonzero constant term. The integer $N$ in this expression is unique, and we shall call it the \textit{degree} of the rational function~$f(\sq)$, written $\deg f(\sq)=N$. For example, $\deg((\sq^2+1)(\sq^3+1)/(\sq^7-\sq+1))=-2$, and if $f(\sq)$ is a polynomial then $\deg f(\sq)$ agrees with the usual degree. We set $\deg(0) = -\infty$.

By a \textit{matrix representation} of $\Hext(L)$ we shall mean a triple $(\pi,M,\sB)$ where $M$ is a right $\Hext(L)$-module over an $\ZZ[\sq,\sq^{-1}]$-polynomial ring $\mathsf{S}$, and $\sB$ is a basis of~$M$. We write (for $h\in\Hext(L)$ and $u,v\in\sB$)
$$
\pi(h,\sB)\quad\text{and}\quad [\pi(h,\sB)]_{u,v}
$$
for the matrix of $\pi(h)$ with respect to the basis $\sB$, and the $(u,v)^{th}$ entry of $\pi(h,\sB)$. When the basis $\sB$ is clear from context we omit it from the notation.

\begin{defn}
A matrix representation $(\pi,M,\sB)$ is called \textit{bounded} if 
$
\deg([\pi(T_w,\sB)]_{u,v})
$ is bounded from above for all $u,v\in \sB$ and all $w\in \Wext$. In this case we call the integer
\begin{align}\label{eq:replace1}
\ba_{\pi,M,\sB}=\max\{\deg([\pi(T_w,\sB)]_{u,v})\mid  u,v\in \sB,\,w\in \Wext\}
\end{align}
the \textit{bound} of the matrix representation. 
\end{defn}

\begin{defn}\label{defn:cellandleading}
If $(\pi,M,\sB)$ is a bounded matrix representation with bound $\ba_{\pi,M,\sB}$ then the \textit{cell recognised by $(\pi,M,\sB)$} is the set
$$
\Gamma_{\pi,M,\sB}=\{w\in\Wext\mid \deg([\pi(T_w,\sB)]_{u,v})=\ba_{\pi,M,\sB}\text{ for some $u,v\in \sB$}\}.
$$
If $w\in\Gamma_{\pi,M,\sB}$ the \textit{leading matrix} of $w$ is defined by
$$
\fc_{\pi,M,\sB}(w)=\mathrm{sp}_{\sq^{-1}=0}\big(\sq^{-\ba_{\pi,M,\sB}}\pi(T_w,\sB)\big),
$$
where $\mathrm{sp}_{\sq^{-1}=0}$ is specialisation at $\sq^{-1}=0$.
\end{defn}

The notion of leading matrices comes from the work of Geck~\cite{Geck:02,Geck:11}, and played a crucial role in \cite{GP:19,GP:19b}. We will not discuss leading matrices further in this work, until the final example in Section~\ref{sec:An}.

\begin{remark}\label{rem:changebasis}
If a finite dimensional representation $(\pi,M,\sB)$ is bounded then $(\pi,M,\sB')$ is bounded for all bases $\sB'$ of $M$, because the (finite) change of basis matrix has a bound on the degrees of its entries. Thus we may talk of a ``bounded (finite dimensional) representation'' without specifying the basis. However we note that both the value of the bound, and the cell recognised, are highly dependent on the particular basis. In this paper we will study the matrix representations $(\pi_{J,\sv},M_{J,\sv},\sB_{\sF})$ where $\sB_{\sF}$ is the basis from Proposition~\ref{prop:basis1} associated to a fundamental domain~$\sF$. These bases appear to have some remarkable and beautiful properties (see Conjectures~\ref{conj:strongconjecture} and~\ref{conj:cells}). We note that the basis from Proposition~\ref{prop:newbasis}, while easier to work with in many respects, does not appear to have the same remarkable properties. 
\end{remark}

Let $\sF$ and $\sF'$ be fundamental domains for the action $\TJ$ on $\WJ$. By Remark~\ref{rem:changebasis} $(\pi_{J,\sv},M_{J,\sv},\sB_{\sF})$ is bounded if and only if $(\pi_{J,\sv},M_{J,\sv},\sB_{\sF'})$ is bounded. However the connection is much stronger, as the following proposition shows. 

\begin{prop}\label{prop:basisindependent}
Let $\sF$ and $\sF'$ be fundamental domains for the action of $\TJ$ on $\WJ$ such that the associated matrix representations are bounded. Then the matrix representations have the same bound, and recognise the same cell. 
\end{prop}

\begin{proof}
Define $g:\sF\to W^J$ and $h:\sF\to  P^{(J)}$ by the equation $u=\st_{h(u)}g(u)$ for $u\in \sF$ (as in the proof of Theorem~\ref{thm:mainpath2}). Then $\varpi_{J,\sv}(X_u)=\zJ^{h(u)}\bm_{g(u)}$, showing that the change of basis matrix from $\sB_{\sF}$ to $\sB_{J,\sv}=\sB_{W^J}$ is a monomial matrix with entries independent of~$\sq$. It follows that the matrix representations with respect to the bases $\sB_{\sF}$ and $\sB_{J,\sv}$ have the same bound and recognise the same elements, hence the result.
\end{proof}

Thus if $(\pi_{J,\sv},M_{J,\sv})$ is bounded we write $\ba_{J,\sv}=\ba_{\pi_{J,\sv},M_{J,\sv},\sB_{\sF}}$ and $\Gamma_{J,\sv}=\Gamma_{\pi_{J,\sv},M_{J,\sv},\sB_{\sF}}$ (for any choice of fundamental domain $\sF$ for the action of $\TJ$ on $\WJ$).

\begin{example}\label{ex:G21}
The concept of boundedness is already interesting and delicate for $1$-dimensional representations (note that the particular choice of basis is irrelevant in the $1$-dimensional case). Let $\pi_{I,\sv}=\psi_{I,\sv}$ be a $1$-dimensional representation of a weighted affine Hecke algebra $\Hext(L)$. We claim that $\pi_{I,\sv}$ is bounded if and only if $\deg\pi_{I,\sv}(X^{\omega_i})\leq 0$ for all $i\in I$. To see this, if $w\in\Wext$ then $w\in W_0t_{\la}W_0$ for some $\la\in P^+$. Then $w=um_{\la}v$ with $u,v\in W_0$ and $m_{\la}$ the minimal length element of $W_0t_{\la}W_0$, and $\ell(w)=\ell(u)+\ell(m_{\la})+\ell(v)$. Moreover, $t_{\la}=m_{\la}w_{\la}$ for some $w_{\la}\in W_0$ with $\ell(t_{\la})=\ell(m_{\la})+\ell(w_{\la})$. Thus $\pi_{I,\sv}(T_w)=\pi_{I,\sv}(T_u)\pi_{I,\sv}(X^{\la})\pi_{I,\sv}(T_{w_{\la}})^{-1}\pi_{I,\sv}(T_v)$, and the claim follows since $u,v,w_{\la}$ are in the finite group $W_0$ (Theorem~\ref{thm:bound} below generalises this argument to higher dimensional representations). 

For example, consider the $\tilde{\sG}_2$ case with weight function $L(s_0)=L(s_2)=b$ and $L(s_1)=a$.
\begin{center}
\begin{tikzpicture}[scale=0.75,baseline=-0.5ex]
\node [inner sep=0.8pt,outer sep=0.8pt] at (3,0) (3) {\Large{$\bullet$}};
\node [inner sep=0.8pt,outer sep=0.8pt] at (4,0) (4) {\Large{$\bullet$}};
\node [inner sep=0.8pt,outer sep=0.8pt] at (5,0) (5) {\Large{$\bullet$}};
\node at (3,-0.5) {$0$};
\node at (4,-0.5) {$2$};
\node at (5,-0.5) {$1$};
\node at (3,0.5) {$b$};
\node at (4,0.5) {$b$};
\node at (5,0.5) {$a$};
\draw (3,0)--(4,0);
\draw (4,0.1)--(5,0.1);
\draw (4,-0.1)--(5,-0.1);
\draw (4,0)--(5,0);
\draw (4.5-0.15,0.3) -- (4.5+0.08,0) -- (4.5-0.15,-0.3);
\end{tikzpicture}
\end{center}
Let $\pi=\psi_{I,\sv}$ be the $1$-dimensional representation of $\Hext(L)$ with $\sv_{\alpha_1}=\pi(T_1)=\sq^a$ and $\sv_{\alpha_2}=\pi(T_2)=\pi(T_0)=-\sq^{-b}$. Since $t_{\omega_1}=0212012121$ and $t_{\omega_2}=021212$ we have $\pi(X^{\omega_1})=4a-6b$ and $\pi(X^{\omega_2})=2a-4b$, and hence $\pi$ is bounded if and only if $a/b\leq 3/2$. Moreover, we claim that if $a/b\leq 3/2$ then the bound $\ba=\ba_{I,\sv}$ of $\pi$ and the cell $\Ga=\Ga_{I,\sv}$ recognised by $\pi$ are as follows:
\begin{compactenum}[$(1)$]
\item if $a/b<1$ then $\ba=a$ and $\Gamma=\{1\}$ (recall $1=s_1$);
\item if $a/b=1$ then $\ba=a=3a-2b$ and $\Gamma=\{1,121,12121\}$;
\item if $1<a/b< 3/2$  then $\ba=3a-2b$ and $\Gamma=\{12121\}$;
\item if $a/b=3/2$ then $\ba=3a-2b$ and $
\Gamma=\{(12121)\cdot(02121)^k\mid k\geq 0\}$.
\end{compactenum} 
We will prove (3) and (4), with the proofs of (1) and (2) being similar. Thus suppose that $1<a/b\leq 3/2$. Since $\deg\pi(T_{12121})=3a-2b$ we have $\ba\geq 3a-2b$. Let $w\in\Ga$, and so $\deg\pi(T_w)=\ba$. Let $D_L(w)$ denote the left descent set of~$w$. If $s\in D_L(w)$ then $w=sw_1$ with $\ell(sw_1)=\ell(w_1)+1$ and hence $\deg\pi(T_{w_1})=\ba-\deg\pi(T_s)$. If $s=s_0$ or $s=s_2$ then we have $\deg\pi(T_{w_1})=\ba+b>\ba$, a contradiction. Thus $D_L(w)=\{s_1\}$. Since $s_1\notin\Ga$ (as $\ba\geq 3a-2b>a$) we have $w_1\neq e$, and hence $D_L(w_1)=\{s_2\}$ (for if $s_0\in D_L(w_1)$ then $s_0\in D_L(w)$). Write $w=s_1s_2w_2$ with $\ell(w)=\ell(w_2)+2$. Since $s_1s_2\notin\Ga$ we have $w_2\neq e$. Note that $s_0\notin D_L(w_2)$, for otherwise $w=s_1s_2s_0w_3'$ with $\ell(w)=\ell(w_3')+3$, giving $\deg\pi(T_{w_3'})=\ba+2b-a\geq \ba+b/2>\ba$. Thus $D_L(w_2)=\{s_1\}$ and so $w=s_1s_2s_1w_3$ with $\ell(w)=\ell(w_3)+3$. Since $s_1s_2s_1\notin\Ga$ we have $w_3\neq e$ and hence $D_L(w_3)=\{s_2\}$ (for if $s_0\in D_L(w_3)$ then $s_0\in D_L(w_2)$). So $w=s_1s_2s_1s_2w_4$ with $\ell(w)=\ell(w_4)+4$. Since $s_1s_2s_1s_2\notin \Ga$ we have $w_4\neq e$. 

There are two cases. Suppose that $a/b<3/2$. Then $D_L(w_4)=\{s_1\}$ (for if $s_0\in D_L(w_4)$ then $w=s_1s_2s_1s_2s_0w_5'$ with $\ell(w)=\ell(w_5')+5$, but then $\deg\pi(T_{w_5'})=\ba+3b-2a>\ba$, a contradiction). Thus $w=s_1s_2s_1s_2s_1w_5$ with $\ell(w)=\ell(w_5)+5$. If $w_5\neq e$ then either $s_2\in D_L(w_5)$, which forces $s_2\in D_L(w)$, or $s_0\in D_L(w_5)$, which forces $s_0\in D_L(w_4)$, in both cases a contradiction. Thus $w_5=e$ and so $\ba=3a-2b$ and $\Ga=\{12121\}$, proving (3).

On the other hand, suppose that $a/b=3/2$. We have either $s_1\in D_L(w_4)$ or $s_0\in D_L(w_4)$. In the latter case, we have $w=s_1s_2s_1s_2s_0w_5'$ with $\deg\pi(T_{w_5'})=\ba$, and hence $w_5'\in\Ga$. Hence the above analysis gives $s_1\in D_L(w_5')$, and since $s_1$ and $s_0$ commute we have $s_1\in D_L(w_4)$. So $s_1\in D_L(w_4)$ in all cases, and so we have $w=s_1s_2s_1s_2s_1w_5$ with $\ell(w)=\ell(w_5)+5$. If $w_5\neq e$ then $D_L(w_5)=\{s_0\}$ (for if $s_2\in D_L(w_5)$ then $s_2\in D_L(w)$), and similar arguments to the above give $w=s_1s_2s_1s_2s_1s_0s_2s_1s_2s_1v$ with $\ell(w)=10+v$. Iterating the argument proves~(4). 
\end{example}

\begin{example}\label{ex:lowestcell} If $J=\emptyset$ then $(\pi_{\emptyset,\sv},M_{\emptyset,\sv},\sB_{\emptyset,\sv})=(\pi,M,\sB)$ is the principal series representation (see Theorem~\ref{thm:induced}). By \cite[Lemma~6.2]{GP:19} we have $\deg \cQ(p)\leq L(\sw_0)$ for all positively folded paths $p$ of reduced type, and hence by Theorem~\ref{thm:mainpath2} $(\pi,M,\sB)$ is bounded with bound $L(\sw_0)$.

For each $\lambda\in P$ let $W_{\lambda}$ be the stabiliser of $\lambda$ in $\Waff$ and let $\sw_{\lambda}$ be the longest element of $W_{\lambda}$. Let $P_0\subseteq P$ denote the set of $\la\in P$ for which $W_{\la}\cong \Wfin$ and $L(\sw_{\la})=L(\sw_0)$. By an analysis similar to that in \cite[Theorem~6.6]{GP:19} it can be shown that the set of elements of $\Wext$ that are recognised by $(\pi,M,\sB)$ is
$
\Gamma=\{w\in \Wext\mid w=w_1\cdot \sw_{\lambda}\cdot w_2,\,w_1,w_2\in \Wext,\,\lambda\in P_0\},
$
where the notation $u\cdot v$ means $\ell(uv)=\ell(u)+\ell(v)$ (recall Conventions~\ref{conv:parameters} and~\ref{conv:sym} are in force, and we further assume that $L(s_n)>L(s_0)$ for the above statement, and so in particular for the $\sBC_n$ case $\{\sw_{\lambda}\mid\lambda\in P_0\}=\{\sw_0\}$). 
\end{example}

\begin{remark}\label{rem:cancellations}
A main motivation for our path formula in Theorem~\ref{thm:mainpath2} is to give a combinatorial approach to studying bounded representations. Indeed positively $J$-folded alcove paths are very useful in studying boundedness (see, for example, Theorems~~\ref{thm:bound} and \ref{thm:negbound}, and the work in~\cite{GP:19,GP:19b}), however before proceeding we briefly discuss some complications and subtleties in the theory. 

Let $u,v\in W^J$ and let $\vec{w}$ be a reduced expression for $w\in\Wext$. If there is $N>0$ such that $\deg\cQ_{J,\sv}(p)\leq N$ for all $p\in\cP_J(\vec{w},u)$ with $\theta^J(p)=v$ then by Theorem~\ref{thm:mainpath2}  we have $\deg[\pi_{J,\sv}(T_w)]_{u,v}\leq N$, showing that boundedness of paths leads directly to boundedness of matrix entries. It is important to note that the reverse direction is more subtle, as this fact significantly complicates the theory. 

For example, let $\Phi$ be of type $\sA_3$ with $J=\{1\}$ and let $\sv$ be the $J$-parameter system with $\sv_{\alpha_1}=-\sq^{-1}$. Let $\vec{\sw}_0=323123$ (a reduced expression for the longest element of $W_0$). The paths $p\in\cP_J(\vec{\sw}_0,e)$ with $\theta^J(p)=e$, along with their respective $\sv$-masses, are as follows (where $\hat{i}$ denotes a fold, $\check{i}$ denotes a bounce and we omit the $s$'s to lighten the notation): 
\begin{align*}
p_1&=\hat{3}\hat{2}\hat{3}\check{1}\hat{2}\hat{3}&p_2&=\hat{3}2\hat{3}\hat{1}2\hat{3}&p_3&=\hat{3}\hat{2}3\check{1}\hat{2}3&p_4&=3\hat{2}3\check{1}\hat{2}\hat{3}&p_5&=32\hat{3}\hat{1}23,
\end{align*}
with $\cQ_{J,\sv}(p_1)=-\sq^{-1}(\sq-\sq^{-1})^5$, $\cQ_{J,\sv}(p_2)=(\sq-\sq^{-1})^4$, $\cQ_{J,\sv}(p_3)=-\sq^{-1}(\sq-\sq^{-1})^3$, $\cQ_{J,\sv}(p_4)=-\sq^{-1}(\sq-\sq^{-1})^3$, and $\cQ_{J,\sv}(p_5)=(\sq-\sq^{-1})^2$. Note that $\deg\cQ_{J,\sv}(p_1)=\deg\cQ_{J,\sv}(p_2)=4$. However by Theorem~\ref{thm:mainpath1} we have
$$
[\pi_{J,\sv}(T_{\sw_0})]_{e,e}=\sum_{i=1}^5\cQ_{J,\sv}(p_i)=\sq^{-4}(\sq-\sq^{-1})^2,
$$
which has degree $-2$. Note that the leading terms have cancelled. The combinatorics of this cancellation of leading terms in matrix coefficients appears to be rather delicate. In \cite{GP:19,GP:19b} the authors were able to deal with this phenomena in affine rank~$3$ Hecke algebras, as the cancellations are rather rare and tame in that low dimension case. Understanding these cancellations in arbitrary rank would lead to a significant advance in understanding boundedness, with implications to Kazhdan-Lusztig theory (as illustrated by the discussion in the following sections, and in the work~\cite{GP:19,GP:19b}). For example, in Section~\ref{sec:An} we are able to sufficiently control the cancellations that occur in the $\sA_n$ case with $J=\{1,2,\ldots,n-1\}$. 
\end{remark}

\subsection{Classification of bounded modules $M_{J,\sv}$}\label{sec:classification}

In this section we classify the weighted $J$-parameter systems for which $(\pi_{J,\sv},M_{J,\sv})$ is bounded (c.f. Remark~\ref{rem:changebasis}). Recall from Section~\ref{sec:JaffineHecke} that the $J$-affine Hecke algebra is a subalgebra of $\cL_J$, and hence $\psi_{J,\sv}$ restricts to a $1$-dimensional representation of~$\HJaff$ (see Corollary~\ref{cor:psi(T0)}).

\begin{thm}\label{thm:bound}
Let $\sv$ be a weighted $J$-parameter system. The following are equivalent.
\begin{compactenum}[$(1)$]
\item The representation $(\pi_{J,\sv},M_{J,\sv})$ is bounded.
\item We have $\deg \sv^{\lambda}\leq 0$ for all $\lambda\in P^+$. 
\item We have $\deg \sv^{\omega_j}\leq 0$ for all $j\in J$.
\item We have $\deg \sv^{u\lambda}\leq 0$ for all $\lambda\in P^+$ and all $u\in W^J$.
\item The associated $1$-dimensional representation $\psi_{J,\sv}$ of $\HJaff(L)$ is bounded.
\item There is a uniform bound $\deg\cQ_{J,\sv}(p)\leq N$ for all positively $J$-folded alcove paths of reduced type.
\end{compactenum}
\end{thm}

\begin{proof}
(1)$\Longrightarrow$(2). Suppose that $(\pi_{J,\sv},M_{J,\sv})$ is bounded. Since
$
\bm_e\cdot X^{\lambda}=\psi_{J,\sv}(X^{\lambda})\bm_e=\sv^{\lambda}\zJ^{\lambda}\bm_e
$
we have $[\pi_{J,\sv}(X^{\lambda})]_{e,e}=\sv^{\lambda}\zJ^{\lambda}$
and so $\deg \sv^{\lambda}\leq 0$ for all $\lambda\in P^+$ (for if not the degree of $[\pi_{J,\sv}(T_{t_{N\lambda}})]_{e,e}=[\pi_{J,\sv}(X^{N\lambda})]_{e,e}$ is unbounded for $N\in\NN$). 

(2)$\Longrightarrow$(3) is trivial.

(3)$\Longrightarrow$(4). Since $\sv^{\omega_i}=1$ for $i\in I\backslash J$ we have
$$
\sv^{u\lambda}=\prod_{j\in J}\sv^{\langle u\lambda,\alpha_j\rangle\omega_j}.
$$
If $u\in W^J$ then $u^{-1}\alpha\in\Phi^+$ for all $\alpha\in\Phi_J^+$, and hence $\langle u\lambda,\alpha\rangle=\langle\lambda,u^{-1}\alpha\rangle\geq 0$ for all $\lambda\in P^+$, $u\in W^J$, and $\alpha\in\Phi_J^+$. Thus $\deg \sv^{u\lambda}\leq 0$. 

(4)$\Longrightarrow$(1). For $\lambda\in P^+$ let $m_{\lambda}$ denote the unique minimal length element of $W_0t_{\lambda}W_0$. Write $t_{\lambda}=m_{\lambda}w_{\lambda}$ where $w_{\lambda}\in W_0$ and $\ell(t_{\lambda})=\ell(m_{\lambda})+\ell(w_{\lambda})$. Each $w\in \Wext$ can be written as $w=um_{\lambda}v$ with $u,v\in W_0$ and $\lambda\in P^+$, and moreover $\ell(w)=\ell(u)+\ell(m_{\lambda})+\ell(v)$. Thus 
$$
T_w=T_uT_{m_{\lambda}}T_v=T_uT_{t_{\lambda}}T_{w_{\lambda}}^{-1}T_v=T_uX^{\lambda}T_{w_{\lambda}}^{-1}T_v.
$$
Let $\sB_{J,\sv}'=\{\xi_{J,\sv}\otimes U_u\mid u\in W^J\}$ denote the basis of $M_{J,\sv}$ from Proposition~\ref{prop:newbasis}, and write $\pi_{J,\sv}'(h)=\pi_{J,\sv}(h,\sB_{J,\sv}')$. We have
$
\pi_{J,\sv}'(T_w)=\pi_{J,\sv}'(T_u)\pi_{J,\sv}'(X^{\lambda})\pi_{J,\sv}'(T_{w_{\lambda}}^{-1})\pi_{J,\sv}'(T_v).
$
Since $u,v,w_{\lambda}$ are in the finite group $W_0$ there is a global bound on the degree of the entries of the first, third, and fourth matrices in this product. Moreover, since $U_uX^{\lambda}=X^{u\lambda}U_u$ the matrix $\pi_{J,\sv}'(X^{\lambda})$ is diagonal with entries $\psi_{J,\sv}(X^{u\lambda})=\sv^{u\lambda}\zJ^{u\lambda}$ for $u\in W^J$, and hence by assumption the degree of the entries of this matrix is bounded for $\lambda\in P^+$. Thus the degree of the entries of $\pi_{J,\sv}'(T_w)$ is bounded. Hence the result (c.f. Remark~\ref{rem:changebasis}).

(3)$\Longleftrightarrow$(5). We have established the equivalence of (1) and (3). Applying this to the case $J=I$ we see that a $1$-dimensional representation $\psi_{I,\sv}$ of $\Hext(L)$ is bounded if and only if $\deg \sv^{\omega_i}\leq 0$ for all $i\in I$. Applying this to the case of the weighted $J$-affine Hecke algebra we see that $\psi_{J,\sv}$ is a bounded $1$-dimensional representation of this algebra if and only if $\deg\sv^{\omega_j}\leq 0$ for all $j\in J$, hence the result.

(1)$\Longleftrightarrow$(6). Theorem~\ref{thm:mainpath1} shows that (6) implies (1). The reverse implication is complicated by the cancellations discussed in Remark~\ref{rem:cancellations}, and we argue as follows. Suppose that (1) holds. Let $p$ be a positively $J$-folded path of reduced type, and let $p_J$ be the associated $J$-straightening. Since $p_J$ is positively folded a well known result (see \cite[Lemma~7.7]{Lus:85} or \cite[Lemma~6.2]{GP:19}) gives $f(p_J)\leq \ell(\sw_0)$, where $f(p_J)$ denotes the number of folds in $p_J$. Decompose $p_J$ as $p_0\cdot f_1\cdot p_1\cdot f_2\cdots\cdot f_k\cdot p_k$, where $p_0,\ldots,p_k$ are straight (ie, no folds) and $f_1,\ldots,f_k$ are the folds (so $k\leq \ell(\sw_0)$). Consider the path $p_j$. Let $x,y\in \WJaff$ such that $p_j$ starts in $x\cA_J$ and ends in $y\cA_J$. It is not difficult to see, using Proposition~\ref{prop:unfold2}, that the contribution of $p_j$ to $\cQ_{J,\sv}(p)$ is $\psi_{J,\sv}(T_{x^{-1}y}')$, where $T_w'$ denotes the standard basis of $\HJaff$ (see Section~\ref{sec:JaffineHecke}). We have already seen that (1) implies (5), and hence there is a bound $\deg\psi_{J,\sv}(T_{w}')\leq N'$ for all $w\in\WJaff$. Thus 
$$
\deg\cQ_{J,\sv}(p)\leq \sum_{i\in\{0\}\cup\Ifin}L(s_i)f_i(p)+(\ell(\sw_0)+1)N'
$$
(with the first sum coming from the folds) hence (6).
\end{proof}

\begin{remark}
On extending scalars to $\CC$ and specialising $\sq\to q$ and $\zeta_i\to z_i$ for $i\in I\backslash J$ with $q>1$ and $z_i\in\CC$ with $|z_i|=1$, Casselman's criteria for temperedness (see \cite[Lemma~2.20]{Opd:04}) along with the equivalence of (1) and (2) in Theorem~\ref{thm:bound} shows that $(\pi_{J,\sv},M_{J,\sv})$ is bounded if and only if the specialised representation is tempered.
\end{remark}

By Theorem~\ref{thm:bound} determining boundedness of $(\pi_{J,\sv},M_{J,\sv})$ is equivalent to determining boundedness of the associated $1$-dimensional representation $\psi_{J,\sv}$ of $\HJaff$. The latter is a much simpler task. Since it suffices to consider each irreducible component of $\HJaff$, it is sufficient to determine the bounded $1$-dimensional representations of an affine Hecke algebra $\Hext(L)$ with irreducible root system~$\Phi$. Since $\Phi$ is irreducible the degrees of freedom in choosing a weighted $I$-parameter system is equal to the number of root lengths in~$\Phi$. If $\Phi$ is simply laced then $\sv_{\alpha}=\sv$ for all $\alpha\in\Phi$ and we write $b=L(s_i)$ for any $i\in I$. If $\Phi$ is reduced with two distinct root lengths we write $\sv_{\mathrm{sh}}=\sv_{\alpha}$ for any short root $\alpha$, and $\sv_{\mathrm{lo}}=\sv_{\alpha}$ for any long root $\alpha$, and we write $a=L(s_j)$ is $\alpha_j$ is short, and $b=L(s_j)$ if $\alpha_j$ is long. Then $\sv_{\mathrm{sh}}\in\{\sq^a,-\sq^{-a}\}$ and $\sv_{\mathrm{lo}}\in\{\sq^b,-\sq^{-b}\}$.

\begin{prop}\label{prop:classifybounded}
The bounded $1$-dimensional representations of $\Hext(L)$ are the maps $\psi_{I,\sv}$ where $\sv=(\sv_{\alpha})_{\alpha\in\Phi}$ is a weighted $I$-parameter system appearing in the list below.
\begin{compactenum}[$(1)$]
\item If $\Phi$ is simply laced then $\sv_{\alpha}=-\sq^{-b}$ for all $\alpha\in\Phi$.
\item If $\Phi$ is reduced and not simply laced then the possible values of $(\sv_{\mathrm{sh}},\sv_{\mathrm{lo}})$ are as follows, with the stated constraints on $a,b$:
$$
\begin{array}{|c||c|c|c|c|}\hline
(\sv_{\mathrm{sh}},\sv_{\mathrm{lo}})&\sB_n&\sC_n&\sF_4&\sG_2\\
\hline
(-\sq^{-a},-\sq^{-b})&a,b\geq 1&a,b\geq 1 &a,b\geq 1&a,b\geq 1\\
\hline
(\sq^{a},-\sq^{-b})&a/b\leq n-1&a/b\leq 1/(n-1)&a/b\leq 6/5&a/b\leq 3/2\\
\hline
(-\sq^{-a},\sq^{b})&a/b\geq 2(n-1)&a/b\geq 2/(n-1)&a/b\geq 5/3&a/b\geq 2\\
\hline
\end{array}
$$
\item If $\Phi$ is of type $\sBC_n$ then the possible values of $(\sv_{\alpha_1},\sv_{\alpha_n}\sv_{2\alpha_n},\sv_{2\alpha_n})$ are as follows, with the stated constraints on $a=L(s_n)$, $b=L(s_1)$, and $c=L(s_0)$ (and Convention~\ref{conv:sym} in force):
$$
\begin{array}{|c||c|}\hline
(\sv_{\alpha_1},\sv_{\alpha_n}\sv_{2\alpha_n},\sv_{2\alpha_n})&\sBC_n\\
\hline
(-\sq^{-b},-\sq^{-a},-\sq^{-c})&a,b,c\geq 1\\
\hline
(-\sq^{-b},-\sq^{-a},\sq^{c})&a,b,c\geq 1\\
\hline
(\sq^{b},-\sq^{-a},-\sq^{-c})&a/b+c/b\geq 2(n-1)\\
\hline
(\sq^{b},-\sq^{-a},\sq^{c})&a/b-c/b\geq 2(n-1)\\
\hline
(-\sq^{-b},\sq^{a},-\sq^{-c})&a/b-c/b\leq n-1\\
\hline
(-\sq^{-b},\sq^{a},\sq^{c})&a/b+c/b\leq n-1\\
\hline
\end{array}
$$
\end{compactenum}
\end{prop}

\begin{proof}
By Theorem~\ref{thm:bound} it is sufficient to determine the weighted $I$-parameter systems $\sv$ with $\deg\sv^{\lambda}\leq 0$ for all $\lambda\in P^+$, and this in turn is equivalent to $\deg\sv^{\omega_i}\leq 0$ for all $i\in I$. Write $\rho=\rho_I$ and $\rho'=\rho_I'$ (see Section~\ref{sec:parabolics}). If $\Phi$ is simply laced then $\sv_{\alpha}=\sv$ is constant for all $\alpha\in\Phi$, and hence $\sv^{\omega_i}=\sv^{\langle \omega_i,2\rho\rangle}$ for all $i\in I$. Since $\sv\in\{\sq^b,-\sq^{-b}\}$ and $\langle\omega_i,2\rho\rangle>0$ the result follows in this case. 

Suppose now that $\Phi$ is reduced with two distinct root lengths. Then
$
\sv^{\omega_i}=\sv_{\mathrm{sh}}^{\langle \omega_i,2\rho'\rangle}\sv_{\mathrm{lo}}^{\langle \omega_i,2\rho\rangle}.
$ Thus if $(\sv_{\mathrm{sh}},\sv_{\mathrm{lo}})=(\sq^{a},-\sq^{-b})$ we require $a/b\leq \langle\omega_i,2\rho\rangle/\langle\omega_i,2\rho'\rangle$ for all $1\leq i\leq n$, and if $(\sv_{\mathrm{sh}},\sv_{\mathrm{lo}})=(-\sq^{-a},\sq^{b})$ we require $a/b\geq \langle\omega_i,2\rho\rangle/\langle\omega_i,2\rho'\rangle$ for all $1\leq i\leq n$.

The result now follows by considering each case. If $\Phi$ is of type $\sB_n$ then in Bourbaki conventions \cite{Bou:02} we have $2\rho=2(n-1)e_1+2(n-2)e_2+\cdots+2e_{n-1}$ and $2\rho'=e_1+e_2+\cdots+e_n$. Since $\omega_i=e_1+\cdots+e_i$ we have $\langle\omega_i,2\rho'\rangle=i$ and $\langle\omega_i,2\rho\rangle=i(2n-i-1)$ for $1\leq i\leq n$. Thus if $(\sv_{\mathrm{sh}},\sv_{\mathrm{lo}})=(\sq^{a},-\sq^{-b})$ we have $a/b\leq 2n-i-1$ for all $1\leq i\leq n$, and so $a/b\leq n-1$. If $(\sv_{\mathrm{sh}},\sv_{\mathrm{lo}})=(-\sq^{-a},\sq^{b})$ then $a/b\geq 2n-i-1$ for all $1\leq i\leq n$, and so $a/b\geq 2(n-1)$. 

The remaining reduced cases are similar. In type $\sC_n$ we have $2\rho=2(e_1+\cdots+e_n)$, $2\rho'=2(n-1)e_1+2(n-2)e_2+\cdots+2e_{n-1}$, $\omega_i=e_1+\cdots+e_i$ for $1\leq i<n$, and $\omega_n=\frac{1}{2}(e_1+\cdots+e_n)$. In type $\sF_4$ we have $2\rho'=5e_1+e_2+e_3+e_4$, $2\rho=6e_1+4e_2+2e_3$, $\omega_1=e_1+e_2$, $\omega_2=2e_1+e_2+e_3$, $\omega_3=3e_1+e_2+e_3+e_4$, and $\omega_4=2e_1$. In type $\sG_2$ we have $2\rho'=-e_2+e_3$, $2\rho=-e_1-e_2+2e_3$, $\omega_1=-e_2+e_3$, and $\omega_2=\frac{1}{3}(-e_1-e_2+2e_3)$.

Finally, consider the $\sBC_n$ case. We have
$$
\sv^{\lambda}=\sv_{\alpha_1}^{\langle\lambda,2(n-1)e_1+2(n-2)e_2+\cdots+2e_{n-1}\rangle}\sv_{\alpha_n}^{\langle \lambda, e_1+\cdots+e_n\rangle}\sv_{2\alpha_n}^{2\langle \lambda, e_1+\cdots+e_n\rangle}.
$$
Since $\omega_i=e_1+\cdots+e_i$ we require 
$
\deg(\sv_{\alpha_1}^{i(2n-i-1)}\sv_{\alpha_n}^i\sv_{2\alpha_n}^{2i})\leq 0$ for all $1\leq i\leq n$, and the result follows by considering the possibilities $\sv_{\alpha_1}\in\{\sq^b,-\sq^{-b}\}$, $\sv_{\alpha_n}\sv_{2\alpha_n}\in\{\sq_a,-\sq^{-a}\}$, and $\sv_{2\alpha_n}\in\{\sq^c,-\sq^{-c}\}$.
\end{proof}

\begin{example}\label{ex:F4} Combining Theorem~\ref{thm:bound} and Proposition~\ref{prop:classifybounded} gives a very explicit classification of the bounded modules $M_{J,\sv}$. For example, in type $\sF_4$ the bounded representations are listed below. We encode the representations by the Dynkin diagram of type $\sF_4$ in which the nodes~$j\in J$ are encircled. An encircled node~$j$ is white if $\psi_{J,\sv}(T_j)=-\sq^{-L(s_j)}$ and black if $\psi_{J,\sv}(T_j)=\sq^{L(s_j)}$. Let $L(s_i)=a$ if $\alpha_i$ is short and $L(s_i)=b$ if $\alpha_i$ is long (thus $L(s_1)=L(s_2)=b$ and $L(s_3)=L(s_4)=a$). In some cases there are constraints on $a/b$ for the representation to be bounded, and these constraints are indicated under the diagram. 

$$
\begin{array}{|c|c|c|c|c|c|}\hline
\begin{tikzpicture}[scale=0.5,baseline=-0.5ex]
\node [inner sep=0.8pt,outer sep=0.8pt] at (0,0) (1) {$\circ$};
\node [inner sep=0.8pt,outer sep=0.8pt] at (1,0) (2) {$\circ$};
\node [inner sep=0.8pt,outer sep=0.8pt] at (2,0) (3) {$\circ$};
\node [inner sep=0.8pt,outer sep=0.8pt] at (3,0) (4) {$\circ$};
\node at (0,0.6) {};
\draw (0.12,0.01)--(0.88,0.01);
\draw (2.12,0.01)--(2.88,0.01);
\draw (1.1,0.1)--(1.9,0.1);
\draw (1.1,-0.075)--(1.9,-0.075);
\draw ({1.5-0.125},0.25) -- (1.5+0.08,0) -- (1.5-0.125,-0.25);
\end{tikzpicture}&
\begin{tikzpicture}[scale=0.5,baseline=-0.5ex]
\node [inner sep=0.8pt,outer sep=0.8pt] at (0,0) (1) {$\circ$};
\node [inner sep=0.8pt,outer sep=0.8pt] at (1,0) (2) {$\circ$};
\node [inner sep=0.8pt,outer sep=0.8pt] at (2,0) (3) {$\circ$};
\node [inner sep=0.8pt,outer sep=0.8pt] at (3,0) (4) {$\circ$};
\node at (0,0.6) {};
\draw (0.12,0.01)--(0.88,0.01);
\draw (2.12,0.01)--(2.88,0.01);
\draw (1.1,0.1)--(1.9,0.1);
\draw (1.1,-0.075)--(1.9,-0.075);
\draw ({1.5-0.125},0.25) -- (1.5+0.08,0) -- (1.5-0.125,-0.25);
\draw [line width=0.5pt,line cap=round,rounded corners] (1.north west)  rectangle (1.south east);
\end{tikzpicture}&
\begin{tikzpicture}[scale=0.5,baseline=-0.5ex]
\node [inner sep=0.8pt,outer sep=0.8pt] at (0,0) (1) {$\circ$};
\node [inner sep=0.8pt,outer sep=0.8pt] at (1,0) (2) {$\circ$};
\node [inner sep=0.8pt,outer sep=0.8pt] at (2,0) (3) {$\circ$};
\node [inner sep=0.8pt,outer sep=0.8pt] at (3,0) (4) {$\circ$};
\node at (0,0.6) {};
\draw (0.12,0.01)--(0.88,0.01);
\draw (2.12,0.01)--(2.88,0.01);
\draw (1.1,0.1)--(1.9,0.1);
\draw (1.1,-0.075)--(1.9,-0.075);
\draw ({1.5-0.125},0.25) -- (1.5+0.08,0) -- (1.5-0.125,-0.25);
\draw [line width=0.5pt,line cap=round,rounded corners] (2.north west)  rectangle (2.south east);
\end{tikzpicture}&
\begin{tikzpicture}[scale=0.5,baseline=-0.5ex]
\node [inner sep=0.8pt,outer sep=0.8pt] at (0,0) (1) {$\circ$};
\node [inner sep=0.8pt,outer sep=0.8pt] at (1,0) (2) {$\circ$};
\node [inner sep=0.8pt,outer sep=0.8pt] at (2,0) (3) {$\circ$};
\node [inner sep=0.8pt,outer sep=0.8pt] at (3,0) (4) {$\circ$};
\node at (0,0.6) {};
\draw (0.12,0.01)--(0.88,0.01);
\draw (2.12,0.01)--(2.88,0.01);
\draw (1.1,0.1)--(1.9,0.1);
\draw (1.1,-0.075)--(1.9,-0.075);
\draw ({1.5-0.125},0.25) -- (1.5+0.08,0) -- (1.5-0.125,-0.25);
\draw [line width=0.5pt,line cap=round,rounded corners] (3.north west)  rectangle (3.south east);
\end{tikzpicture}&
\begin{tikzpicture}[scale=0.5,baseline=-0.5ex]
\node [inner sep=0.8pt,outer sep=0.8pt] at (0,0) (1) {$\circ$};
\node [inner sep=0.8pt,outer sep=0.8pt] at (1,0) (2) {$\circ$};
\node [inner sep=0.8pt,outer sep=0.8pt] at (2,0) (3) {$\circ$};
\node [inner sep=0.8pt,outer sep=0.8pt] at (3,0) (4) {$\circ$};
\node at (0,0.6) {};
\draw (0.12,0.01)--(0.88,0.01);
\draw (2.12,0.01)--(2.88,0.01);
\draw (1.1,0.1)--(1.9,0.1);
\draw (1.1,-0.075)--(1.9,-0.075);
\draw ({1.5-0.125},0.25) -- (1.5+0.08,0) -- (1.5-0.125,-0.25);
\draw [line width=0.5pt,line cap=round,rounded corners] (4.north west)  rectangle (4.south east);
\end{tikzpicture}&
\begin{tikzpicture}[scale=0.5,baseline=-0.5ex]
\node [inner sep=0.8pt,outer sep=0.8pt] at (0,0) (1) {$\circ$};
\node [inner sep=0.8pt,outer sep=0.8pt] at (1,0) (2) {$\circ$};
\node [inner sep=0.8pt,outer sep=0.8pt] at (2,0) (3) {$\circ$};
\node [inner sep=0.8pt,outer sep=0.8pt] at (3,0) (4) {$\circ$};
\node at (0,0.6) {};
\draw (0.12,0.01)--(0.88,0.01);
\draw (2.12,0.01)--(2.88,0.01);
\draw (1.1,0.1)--(1.9,0.1);
\draw (1.1,-0.075)--(1.9,-0.075);
\draw ({1.5-0.125},0.25) -- (1.5+0.08,0) -- (1.5-0.125,-0.25);
\draw [line width=0.5pt,line cap=round,rounded corners] (1.north west)  rectangle (2.south east);
\end{tikzpicture}\\
\hline
\begin{tikzpicture}[scale=0.5,baseline=-0.5ex]
\node [inner sep=0.8pt,outer sep=0.8pt] at (0,0) (1) {$\circ$};
\node [inner sep=0.8pt,outer sep=0.8pt] at (1,0) (2) {$\circ$};
\node [inner sep=0.8pt,outer sep=0.8pt] at (2,0) (3) {$\circ$};
\node [inner sep=0.8pt,outer sep=0.8pt] at (3,0) (4) {$\circ$};
\node at (0,0.6) {};
\draw (0.12,0.01)--(0.88,0.01);
\draw (2.12,0.01)--(2.88,0.01);
\draw (1.1,0.1)--(1.9,0.1);
\draw (1.1,-0.075)--(1.9,-0.075);
\draw ({1.5-0.125},0.25) -- (1.5+0.08,0) -- (1.5-0.125,-0.25);
\draw [line width=0.5pt,line cap=round,rounded corners] (2.north west)  rectangle (3.south east);
\end{tikzpicture}&
\begin{tikzpicture}[scale=0.5,baseline=-0.5ex]
\node [inner sep=0.8pt,outer sep=0.8pt] at (0,0) (1) {$\circ$};
\node [inner sep=0.8pt,outer sep=0.8pt] at (1,0) (2) {$\circ$};
\node [inner sep=0.8pt,outer sep=0.8pt] at (2,0) (3) {$\circ$};
\node [inner sep=0.8pt,outer sep=0.8pt] at (3,0) (4) {$\circ$};
\node at (0,0.6) {};
\draw (0.12,0.01)--(0.88,0.01);
\draw (2.12,0.01)--(2.88,0.01);
\draw (1.1,0.1)--(1.9,0.1);
\draw (1.1,-0.075)--(1.9,-0.075);
\draw ({1.5-0.125},0.25) -- (1.5+0.08,0) -- (1.5-0.125,-0.25);
\draw [line width=0.5pt,line cap=round,rounded corners] (3.north west)  rectangle (4.south east);
\end{tikzpicture}&
\begin{tikzpicture}[scale=0.5,baseline=-0.5ex]
\node [inner sep=0.8pt,outer sep=0.8pt] at (0,0) (1) {$\circ$};
\node [inner sep=0.8pt,outer sep=0.8pt] at (1,0) (2) {$\circ$};
\node [inner sep=0.8pt,outer sep=0.8pt] at (2,0) (3) {$\circ$};
\node [inner sep=0.8pt,outer sep=0.8pt] at (3,0) (4) {$\circ$};
\node at (0,0.6) {};
\draw (0.12,0.01)--(0.88,0.01);
\draw (2.12,0.01)--(2.88,0.01);
\draw (1.1,0.1)--(1.9,0.1);
\draw (1.1,-0.075)--(1.9,-0.075);
\draw ({1.5-0.125},0.25) -- (1.5+0.08,0) -- (1.5-0.125,-0.25);
\draw [line width=0.5pt,line cap=round,rounded corners] (1.north west)  rectangle (1.south east);
\draw [line width=0.5pt,line cap=round,rounded corners] (3.north west)  rectangle (3.south east);
\end{tikzpicture}&
\begin{tikzpicture}[scale=0.5,baseline=-0.5ex]
\node [inner sep=0.8pt,outer sep=0.8pt] at (0,0) (1) {$\circ$};
\node [inner sep=0.8pt,outer sep=0.8pt] at (1,0) (2) {$\circ$};
\node [inner sep=0.8pt,outer sep=0.8pt] at (2,0) (3) {$\circ$};
\node [inner sep=0.8pt,outer sep=0.8pt] at (3,0) (4) {$\circ$};
\node at (0,0.6) {};
\draw (0.12,0.01)--(0.88,0.01);
\draw (2.12,0.01)--(2.88,0.01);
\draw (1.1,0.1)--(1.9,0.1);
\draw (1.1,-0.075)--(1.9,-0.075);
\draw ({1.5-0.125},0.25) -- (1.5+0.08,0) -- (1.5-0.125,-0.25);
\draw [line width=0.5pt,line cap=round,rounded corners] (1.north west)  rectangle (1.south east);
\draw [line width=0.5pt,line cap=round,rounded corners] (4.north west)  rectangle (4.south east);
\end{tikzpicture}&
\begin{tikzpicture}[scale=0.5,baseline=-0.5ex]
\node [inner sep=0.8pt,outer sep=0.8pt] at (0,0) (1) {$\circ$};
\node [inner sep=0.8pt,outer sep=0.8pt] at (1,0) (2) {$\circ$};
\node [inner sep=0.8pt,outer sep=0.8pt] at (2,0) (3) {$\circ$};
\node [inner sep=0.8pt,outer sep=0.8pt] at (3,0) (4) {$\circ$};
\node at (0,0.6) {};
\draw (0.12,0.01)--(0.88,0.01);
\draw (2.12,0.01)--(2.88,0.01);
\draw (1.1,0.1)--(1.9,0.1);
\draw (1.1,-0.075)--(1.9,-0.075);
\draw ({1.5-0.125},0.25) -- (1.5+0.08,0) -- (1.5-0.125,-0.25);
\draw [line width=0.5pt,line cap=round,rounded corners] (2.north west)  rectangle (2.south east);
\draw [line width=0.5pt,line cap=round,rounded corners] (4.north west)  rectangle (4.south east);
\end{tikzpicture}&
\begin{tikzpicture}[scale=0.5,baseline=-0.5ex]
\node [inner sep=0.8pt,outer sep=0.8pt] at (0,0) (1) {$\circ$};
\node [inner sep=0.8pt,outer sep=0.8pt] at (1,0) (2) {$\circ$};
\node [inner sep=0.8pt,outer sep=0.8pt] at (2,0) (3) {$\circ$};
\node [inner sep=0.8pt,outer sep=0.8pt] at (3,0) (4) {$\circ$};
\node at (0,0.6) {};
\draw (0.12,0.01)--(0.88,0.01);
\draw (2.12,0.01)--(2.88,0.01);
\draw (1.1,0.1)--(1.9,0.1);
\draw (1.1,-0.075)--(1.9,-0.075);
\draw ({1.5-0.125},0.25) -- (1.5+0.08,0) -- (1.5-0.125,-0.25);
\draw [line width=0.5pt,line cap=round,rounded corners] (1.north west)  rectangle (3.south east);
\end{tikzpicture}\\
\hline
\begin{tikzpicture}[scale=0.5,baseline=-0.5ex]
\node [inner sep=0.8pt,outer sep=0.8pt] at (0,0) (1) {$\circ$};
\node [inner sep=0.8pt,outer sep=0.8pt] at (1,0) (2) {$\circ$};
\node [inner sep=0.8pt,outer sep=0.8pt] at (2,0) (3) {$\circ$};
\node [inner sep=0.8pt,outer sep=0.8pt] at (3,0) (4) {$\circ$};
\node at (0,0.6) {};
\draw (0.12,0.01)--(0.88,0.01);
\draw (2.12,0.01)--(2.88,0.01);
\draw (1.1,0.1)--(1.9,0.1);
\draw (1.1,-0.075)--(1.9,-0.075);
\draw ({1.5-0.125},0.25) -- (1.5+0.08,0) -- (1.5-0.125,-0.25);
\draw [line width=0.5pt,line cap=round,rounded corners] (2.north west)  rectangle (4.south east);
\end{tikzpicture}&
\begin{tikzpicture}[scale=0.5,baseline=-0.5ex]
\node [inner sep=0.8pt,outer sep=0.8pt] at (0,0) (1) {$\circ$};
\node [inner sep=0.8pt,outer sep=0.8pt] at (1,0) (2) {$\circ$};
\node [inner sep=0.8pt,outer sep=0.8pt] at (2,0) (3) {$\circ$};
\node [inner sep=0.8pt,outer sep=0.8pt] at (3,0) (4) {$\circ$};
\node at (0,0.6) {};
\draw (0.12,0.01)--(0.88,0.01);
\draw (2.12,0.01)--(2.88,0.01);
\draw (1.1,0.1)--(1.9,0.1);
\draw (1.1,-0.075)--(1.9,-0.075);
\draw ({1.5-0.125},0.25) -- (1.5+0.08,0) -- (1.5-0.125,-0.25);
\draw [line width=0.5pt,line cap=round,rounded corners] (1.north west)  rectangle (1.south east);
\draw [line width=0.5pt,line cap=round,rounded corners] (3.north west)  rectangle (4.south east);
\end{tikzpicture}&
\begin{tikzpicture}[scale=0.5,baseline=-0.5ex]
\node [inner sep=0.8pt,outer sep=0.8pt] at (0,0) (1) {$\circ$};
\node [inner sep=0.8pt,outer sep=0.8pt] at (1,0) (2) {$\circ$};
\node [inner sep=0.8pt,outer sep=0.8pt] at (2,0) (3) {$\circ$};
\node [inner sep=0.8pt,outer sep=0.8pt] at (3,0) (4) {$\circ$};
\node at (0,0.6) {};
\draw (0.12,0.01)--(0.88,0.01);
\draw (2.12,0.01)--(2.88,0.01);
\draw (1.1,0.1)--(1.9,0.1);
\draw (1.1,-0.075)--(1.9,-0.075);
\draw ({1.5-0.125},0.25) -- (1.5+0.08,0) -- (1.5-0.125,-0.25);
\draw [line width=0.5pt,line cap=round,rounded corners] (1.north west)  rectangle (2.south east);
\draw [line width=0.5pt,line cap=round,rounded corners] (4.north west)  rectangle (4.south east);
\end{tikzpicture}&
\begin{tikzpicture}[scale=0.5,baseline=-0.5ex]
\node [inner sep=0.8pt,outer sep=0.8pt] at (0,0) (1) {$\circ$};
\node [inner sep=0.8pt,outer sep=0.8pt] at (1,0) (2) {$\circ$};
\node [inner sep=0.8pt,outer sep=0.8pt] at (2,0) (3) {$\circ$};
\node [inner sep=0.8pt,outer sep=0.8pt] at (3,0) (4) {$\circ$};
\node at (0,0.6) {};
\draw (0.12,0.01)--(0.88,0.01);
\draw (2.12,0.01)--(2.88,0.01);
\draw (1.1,0.1)--(1.9,0.1);
\draw (1.1,-0.075)--(1.9,-0.075);
\draw ({1.5-0.125},0.25) -- (1.5+0.08,0) -- (1.5-0.125,-0.25);
\draw [line width=0.5pt,line cap=round,rounded corners] (1.north west)  rectangle (4.south east);
\end{tikzpicture}&
\begin{tikzpicture}[scale=0.5,baseline=-0.5ex]
\node [inner sep=0.8pt,outer sep=0.8pt] at (0,0) (1) {$\circ$};
\node [inner sep=0.8pt,outer sep=0.8pt] at (1,0) (2) {$\circ$};
\node [inner sep=0.8pt,outer sep=0.8pt] at (2,0) (3) {$\bullet$};
\node [inner sep=0.8pt,outer sep=0.8pt] at (3,0) (4) {$\circ$};
\node at (0,0.6) {};
\node at (1.5,-0.9) {\scriptsize{$a/b\leq 2$}};
\draw (0.12,0.01)--(0.88,0.01);
\draw (2.12,0.01)--(2.88,0.01);
\draw (1.1,0.1)--(1.9,0.1);
\draw (1.1,-0.075)--(1.9,-0.075);
\draw ({1.5-0.125},0.25) -- (1.5+0.08,0) -- (1.5-0.125,-0.25);
\draw [line width=0.5pt,line cap=round,rounded corners] (1.north west)  rectangle (3.south east);
\end{tikzpicture}&
\begin{tikzpicture}[scale=0.5,baseline=-0.5ex]
\node [inner sep=0.8pt,outer sep=0.8pt] at (0,0) (1) {$\bullet$};
\node [inner sep=0.8pt,outer sep=0.8pt] at (1,0) (2) {$\bullet$};
\node [inner sep=0.8pt,outer sep=0.8pt] at (2,0) (3) {$\circ$};
\node [inner sep=0.8pt,outer sep=0.8pt] at (3,0) (4) {$\circ$};
\node at (0,0.6) {};
\node at (1.5,-0.9) {\scriptsize{$a/b\geq 4$}};
\draw (0.12,0.01)--(0.88,0.01);
\draw (2.12,0.01)--(2.88,0.01);
\draw (1.1,0.1)--(1.9,0.1);
\draw (1.1,-0.075)--(1.9,-0.075);
\draw ({1.5-0.125},0.25) -- (1.5+0.08,0) -- (1.5-0.125,-0.25);
\draw [line width=0.5pt,line cap=round,rounded corners] (1.north west)  rectangle (3.south east);
\end{tikzpicture}\\
\hline
\begin{tikzpicture}[scale=0.5,baseline=-0.5ex]
\node [inner sep=0.8pt,outer sep=0.8pt] at (0,0) (1) {$\circ$};
\node [inner sep=0.8pt,outer sep=0.8pt] at (1,0) (2) {$\bullet$};
\node [inner sep=0.8pt,outer sep=0.8pt] at (2,0) (3) {$\circ$};
\node [inner sep=0.8pt,outer sep=0.8pt] at (3,0) (4) {$\circ$};
\node at (0,0.6) {};
\node at (1.5,-0.9) {\scriptsize{$a/b\geq 1$}};
\draw (0.12,0.01)--(0.88,0.01);
\draw (2.12,0.01)--(2.88,0.01);
\draw (1.1,0.1)--(1.9,0.1);
\draw (1.1,-0.075)--(1.9,-0.075);
\draw ({1.5-0.125},0.25) -- (1.5+0.08,0) -- (1.5-0.125,-0.25);
\draw [line width=0.5pt,line cap=round,rounded corners] (2.north west)  rectangle (4.south east);
\end{tikzpicture}&
\begin{tikzpicture}[scale=0.5,baseline=-0.5ex]
\node [inner sep=0.8pt,outer sep=0.8pt] at (0,0) (1) {$\circ$};
\node [inner sep=0.8pt,outer sep=0.8pt] at (1,0) (2) {$\circ$};
\node [inner sep=0.8pt,outer sep=0.8pt] at (2,0) (3) {$\bullet$};
\node [inner sep=0.8pt,outer sep=0.8pt] at (3,0) (4) {$\bullet$};
\node at (0,0.6) {};
\node at (1.5,-0.9) {\scriptsize{$a/b\leq 1/2$}};
\draw (0.12,0.01)--(0.88,0.01);
\draw (2.12,0.01)--(2.88,0.01);
\draw (1.1,0.1)--(1.9,0.1);
\draw (1.1,-0.075)--(1.9,-0.075);
\draw ({1.5-0.125},0.25) -- (1.5+0.08,0) -- (1.5-0.125,-0.25);
\draw [line width=0.5pt,line cap=round,rounded corners] (2.north west)  rectangle (4.south east);
\end{tikzpicture}&
\begin{tikzpicture}[scale=0.5,baseline=-0.5ex]
\node [inner sep=0.8pt,outer sep=0.8pt] at (0,0) (1) {$\circ$};
\node [inner sep=0.8pt,outer sep=0.8pt] at (1,0) (2) {$\circ$};
\node [inner sep=0.8pt,outer sep=0.8pt] at (2,0) (3) {$\bullet$};
\node [inner sep=0.8pt,outer sep=0.8pt] at (3,0) (4) {$\circ$};
\node at (0,0.6) {};
\node at (1.5,-0.9) {\scriptsize{$a/b\geq 2$}};
\draw (0.12,0.01)--(0.88,0.01);
\draw (2.12,0.01)--(2.88,0.01);
\draw (1.1,0.1)--(1.9,0.1);
\draw (1.1,-0.075)--(1.9,-0.075);
\draw ({1.5-0.125},0.25) -- (1.5+0.08,0) -- (1.5-0.125,-0.25);
\draw [line width=0.5pt,line cap=round,rounded corners] (2.north west)  rectangle (3.south east);
\end{tikzpicture}&
\begin{tikzpicture}[scale=0.5,baseline=-0.5ex]
\node [inner sep=0.8pt,outer sep=0.8pt] at (0,0) (1) {$\circ$};
\node [inner sep=0.8pt,outer sep=0.8pt] at (1,0) (2) {$\bullet$};
\node [inner sep=0.8pt,outer sep=0.8pt] at (2,0) (3) {$\circ$};
\node [inner sep=0.8pt,outer sep=0.8pt] at (3,0) (4) {$\circ$};
\node at (0,0.6) {};
\node at (1.5,-0.9) {\scriptsize{$a/b\leq 1$}};
\draw (0.12,0.01)--(0.88,0.01);
\draw (2.12,0.01)--(2.88,0.01);
\draw (1.1,0.1)--(1.9,0.1);
\draw (1.1,-0.075)--(1.9,-0.075);
\draw ({1.5-0.125},0.25) -- (1.5+0.08,0) -- (1.5-0.125,-0.25);
\draw [line width=0.5pt,line cap=round,rounded corners] (2.north west)  rectangle (3.south east);
\end{tikzpicture}&
\begin{tikzpicture}[scale=0.5,baseline=-0.5ex]
\node [inner sep=0.8pt,outer sep=0.8pt] at (0,0) (1) {$\bullet$};
\node [inner sep=0.8pt,outer sep=0.8pt] at (1,0) (2) {$\bullet$};
\node [inner sep=0.8pt,outer sep=0.8pt] at (2,0) (3) {$\circ$};
\node [inner sep=0.8pt,outer sep=0.8pt] at (3,0) (4) {$\circ$};
\node at (0,0.6) {};
\node at (1.5,-0.9) {\scriptsize{$a/b\geq 5/3$}};
\draw (0.12,0.01)--(0.88,0.01);
\draw (2.12,0.01)--(2.88,0.01);
\draw (1.1,0.1)--(1.9,0.1);
\draw (1.1,-0.075)--(1.9,-0.075);
\draw ({1.5-0.125},0.25) -- (1.5+0.08,0) -- (1.5-0.125,-0.25);
\draw [line width=0.5pt,line cap=round,rounded corners] (1.north west)  rectangle (4.south east);
\end{tikzpicture}&
\begin{tikzpicture}[scale=0.5,baseline=-0.5ex]
\node [inner sep=0.8pt,outer sep=0.8pt] at (0,0) (1) {$\circ$};
\node [inner sep=0.8pt,outer sep=0.8pt] at (1,0) (2) {$\circ$};
\node [inner sep=0.8pt,outer sep=0.8pt] at (2,0) (3) {$\bullet$};
\node [inner sep=0.8pt,outer sep=0.8pt] at (3,0) (4) {$\bullet$};
\node at (0,0.6) {};
\node at (1.5,-0.9) {\scriptsize{$a/b\leq 6/5$}};
\draw (0.12,0.01)--(0.88,0.01);
\draw (2.12,0.01)--(2.88,0.01);
\draw (1.1,0.1)--(1.9,0.1);
\draw (1.1,-0.075)--(1.9,-0.075);
\draw ({1.5-0.125},0.25) -- (1.5+0.08,0) -- (1.5-0.125,-0.25);
\draw [line width=0.5pt,line cap=round,rounded corners] (1.north west)  rectangle (4.south east);
\end{tikzpicture}\\
\hline
\end{array}
$$
\end{example}

\subsection{The bound $\ba_{J,\sv}$}

In this section we use the path formula (Theorem~\ref{thm:mainpath1}) to give an upper bound for the bound $\ba_{J,\hat{\sv}}$ for the representations associated to the simplest weighted $J$-parameter system $\hat{\sv}=(\sv_{\alpha})_{\alpha\in \Phi_J}$ given by
\begin{align*}
\sv_{\alpha_j}&=-\sq^{-L(s_j)}&&\text{whenever $j\in J$ with $2\alpha_j\notin\Phi_J$}\\
\sv_{\alpha_n}\sv_{2\alpha_n}&=-\sq^{-L(s_n)}&&\text{if $\Phi_J$ is not reduced}\\
\sv_{2\alpha_n}&=-\sq^{-L(s_0)}&&\text{if $\Phi_J$ is not reduced}.
\end{align*}

With Convention~\ref{conv:sym} in force we have the following result.

\begin{thm}\label{thm:negbound}
The bound of $(\pi_{J,\hat{\sv}},M_{J,\hat{\sv}},\sB_{J,\hat{\sv}})$ satisfies $\ba_{J,\hat{\sv}}\leq L(\sw_0)$.
\end{thm}

\begin{proof}
By Theorem~\ref{thm:mainpath1} we have, for $w\in\Wext$ and $u,v\in\sB_{J,\hat{\sv}}$,
\begin{align}\label{eq:ineq}
\deg[\pi_{J,\hat{\sv}}(T_w)]_{u,v}\leq \max\{\deg \cQ_{J,\hat{\sv}}(p)\mid p\in \cP_J(\vec{w},u)\text{ with }\theta^J(p)=v\},
\end{align}
and by Proposition~\ref{prop:unfold2} we have 
\begin{align}\label{eq:secondform}
\cQ_{J,\hat{\sv}}(p)=\cQ(p_J)\prod_{\alpha-k\delta\in\widetilde{\Phi}^+}\sv_{\alpha-k\delta}^{c_{\alpha,k}(p_J)}
\end{align}
for all positively $J$-folded paths~$p$. It is well known that $\deg\cQ(p')\leq L(\sw_0)$ for all positively folded alcove paths $p'$ (see \cite[Lemma~7.7]{Lus:85} or \cite[Lemma~6.2]{GP:19}). Since the second term in~(\ref{eq:secondform}) has degree bounded above by~$0$ we have $\deg\cQ_{J,\hat{\sv}}(p)\leq L(\sw_0)$.
\end{proof}

\subsection{Conjectures}\label{sec:conj}

The precise value of the bound $\ba_{J,\sv}$ appears to be a very subtle statistic, and we conjecture (based on the analysis in \cite{GP:19,GP:19b} and the examples below) that it is intimately connected to Lusztig's $\ba$-function~\cite{Lus:03}, Macdonald's $c$-function~\cite{Mac:71}, and Opdam's Plancherel Theorem~\cite{Opd:04}. Again Convention~\ref{conv:sym} is assumed to be in force. 
Firstly, we believe that the upper bound $L(\sw_0)$ for $\ba_{J,\hat{\sv}}$ given in Theorem~\ref{thm:negbound} applies more generally. 

\begin{conjecture}\label{conj:upperbound}
 If $(\pi_{J,\sv},M_{J,\sv},\sB_{J,\sv})$ is bounded then the bound $\ba_{J,\sv}$ satisfies $\ba_{J,\sv}\leq L(\sw_0)$ with equality if and only if $J=\emptyset$.
 \end{conjecture}
 
In fact, we shall state a considerably stronger conjecture giving a formula for $\ba_{J,\sv}$. For $\alpha\in\Phi$ we define $\sq_{\alpha}\in\ZZ[\sq,\sq^{-1}]$ as follows. If $\Phi$ is reduced, let
$$
\sq_{\alpha}=\sq^{L(s_i)}\quad\text{if $\alpha\in W_0\alpha_i$},
$$
and if $\Phi$ is not reduced let
$$
\sq_{\alpha}=\begin{cases}
\sq^{L(s_i)}&\text{if $\alpha\in W_0\alpha_i$ with $i\neq n$}\\
\sq^{L(s_n)-L(s_0)}&\text{if $\alpha\in W_0\alpha_n$}\\
\sq^{L(s_0)}&\text{if $\alpha$ is long.}
\end{cases}
$$

\begin{conjecture}\label{conj:strongconjecture}
If $(\pi_{J,\sv},M_{J,\sv},\sB_{J,\sv})$ is bounded then the bound $\ba_{J,\sv}$ is given by
\begin{align*}
\ba_{J,\sv}=L(\sw_0)-\frac{1}{2}\deg {\prod_{\alpha\in\Phi}}'\frac{1-\sq_{\alpha/2}^{-1}\sv^{\alpha^{\vee}}}{1-\sq_{\alpha/2}^{-1}\sq_{\alpha}^{-2}\sv^{\alpha^{\vee}}}
\end{align*}
where ${\prod}'$ indicates that any factors in the numerator or denominator that are~$0$ are omitted. 
\end{conjecture}

We make the following conjecture, linking bounded representations to \textit{Lusztig's $\ba$ function} (see \cite{Lus:03} for the definition of the $\ba$-function). In particular, note that Conjecture~\ref{conj:cells} combined with Conjecture~\ref{conj:strongconjecture} give a conjectural formula for the value of Lusztig's $\ba$-function at elements $w\in\Waff$ that are recognised by some bounded representation $(\pi_{J,\sv},M_{J,\sv},\sB_{J,\sv})$. 

\begin{conjecture}\label{conj:cells}
Let $\Gamma_{J,\sv}$ be the cell recognised by the bounded representation~$(\pi_{J,\sv},M_{J,\sv},\sB_{J,\sv})$. Then
\begin{compactenum}[$(1)$]
\item Lusztig's $\ba$-function satisfies $\ba(w)=\ba_{J,\sv}$ for all $w\in\Gamma_{J,\sv}\cap\Waff$.
\item The set $\Gamma_{J,\sv}\cap \Waff$ is contained in a two sided Kazhdan-Lusztig cell of the weighted Coxeter group $(\Waff,L)$. 
\end{compactenum}
\end{conjecture}

\begin{remark}
It is not necessarily true that $\Gamma_{J,\sv}$ \textit{equals} a two sided Kazhdan-Lusztig cell. For example, in Example~\ref{ex:G21}(2) and (4) the set $\Gamma_{J,\sv}$ is strictly contained in a two sided cell (see \cite[Figure~2]{GP:19} for the cell decomposition of $\tilde{\sG}_2$). 
\end{remark}

\newpage

\begin{example}
Consider the case $\Phi=\sF_4$, with notation as in Example~\ref{ex:F4}. Writing $r=a/b$, the conjectural bounds $\ba_{\pi}$ (from Conjecture~\ref{conj:strongconjecture}) for a selection of the bounded representations of $\tilde{\sF}_4$ are as follows:
\begin{compactenum}[$(1)$]
\item \begin{tikzpicture}[scale=0.5,baseline=-0.5ex]
\node [inner sep=0.8pt,outer sep=0.8pt] at (0,0) (1) {$\circ$};
\node [inner sep=0.8pt,outer sep=0.8pt] at (1,0) (2) {$\circ$};
\node [inner sep=0.8pt,outer sep=0.8pt] at (2,0) (3) {$\bullet$};
\node [inner sep=0.8pt,outer sep=0.8pt] at (3,0) (4) {$\circ$};
\node at (0,0.6) {};
\draw (0.12,0.01)--(0.88,0.01);
\draw (2.12,0.01)--(2.88,0.01);
\draw (1.1,0.1)--(1.9,0.1);
\draw (1.1,-0.075)--(1.9,-0.075);
\draw ({1.5-0.125},0.25) -- (1.5+0.08,0) -- (1.5-0.125,-0.25);
\draw [line width=0.5pt,line cap=round,rounded corners] (1.north west)  rectangle (3.south east);
\end{tikzpicture} $\ba_{\pi}=2a+2b,5a,6a-b,11a-7b$ for $r\in (0,2/3],[2/3,1],[1,6/5],[6/5,2]$. 
\item \begin{tikzpicture}[scale=0.5,baseline=-0.5ex]
\node [inner sep=0.8pt,outer sep=0.8pt] at (0,0) (1) {$\bullet$};
\node [inner sep=0.8pt,outer sep=0.8pt] at (1,0) (2) {$\bullet$};
\node [inner sep=0.8pt,outer sep=0.8pt] at (2,0) (3) {$\circ$};
\node [inner sep=0.8pt,outer sep=0.8pt] at (3,0) (4) {$\circ$};
\node at (0,0.6) {};
\draw (0.12,0.01)--(0.88,0.01);
\draw (2.12,0.01)--(2.88,0.01);
\draw (1.1,0.1)--(1.9,0.1);
\draw (1.1,-0.075)--(1.9,-0.075);
\draw ({1.5-0.125},0.25) -- (1.5+0.08,0) -- (1.5-0.125,-0.25);
\draw [line width=0.5pt,line cap=round,rounded corners] (1.north west)  rectangle (3.south east);
\end{tikzpicture} $\ba_{\pi}=4a+12b$ for $r\in [4,\infty)$. 
\item \begin{tikzpicture}[scale=0.5,baseline=-0.5ex]
\node [inner sep=0.8pt,outer sep=0.8pt] at (0,0) (1) {$\circ$};
\node [inner sep=0.8pt,outer sep=0.8pt] at (1,0) (2) {$\bullet$};
\node [inner sep=0.8pt,outer sep=0.8pt] at (2,0) (3) {$\circ$};
\node [inner sep=0.8pt,outer sep=0.8pt] at (3,0) (4) {$\circ$};
\node at (0,0.6) {};
\draw (0.12,0.01)--(0.88,0.01);
\draw (2.12,0.01)--(2.88,0.01);
\draw (1.1,0.1)--(1.9,0.1);
\draw (1.1,-0.075)--(1.9,-0.075);
\draw ({1.5-0.125},0.25) -- (1.5+0.08,0) -- (1.5-0.125,-0.25);
\draw [line width=0.5pt,line cap=round,rounded corners] (2.north west)  rectangle (4.south east);
\end{tikzpicture} $\ba_{\pi}=-2a+11b,a+6b,2a+3b$ for $r\in [1,5/3],[5/3,3],[3,\infty)$. 
\item \begin{tikzpicture}[scale=0.5,baseline=-0.5ex]
\node [inner sep=0.8pt,outer sep=0.8pt] at (0,0) (1) {$\circ$};
\node [inner sep=0.8pt,outer sep=0.8pt] at (1,0) (2) {$\circ$};
\node [inner sep=0.8pt,outer sep=0.8pt] at (2,0) (3) {$\bullet$};
\node [inner sep=0.8pt,outer sep=0.8pt] at (3,0) (4) {$\bullet$};
\node at (0,0.6) {};
\draw (0.12,0.01)--(0.88,0.01);
\draw (2.12,0.01)--(2.88,0.01);
\draw (1.1,0.1)--(1.9,0.1);
\draw (1.1,-0.075)--(1.9,-0.075);
\draw ({1.5-0.125},0.25) -- (1.5+0.08,0) -- (1.5-0.125,-0.25);
\draw [line width=0.5pt,line cap=round,rounded corners] (2.north west)  rectangle (4.south east);
\end{tikzpicture} $\ba_{\pi}=6a+4b,9a+3b$ for $r\in (0,1/3],[1/3,1/2]$. 
\item \begin{tikzpicture}[scale=0.5,baseline=-0.5ex]
\node [inner sep=0.8pt,outer sep=0.8pt] at (0,0) (1) {$\circ$};
\node [inner sep=0.8pt,outer sep=0.8pt] at (1,0) (2) {$\circ$};
\node [inner sep=0.8pt,outer sep=0.8pt] at (2,0) (3) {$\bullet$};
\node [inner sep=0.8pt,outer sep=0.8pt] at (3,0) (4) {$\circ$};
\node at (0,0.6) {};
\draw (0.12,0.01)--(0.88,0.01);
\draw (2.12,0.01)--(2.88,0.01);
\draw (1.1,0.1)--(1.9,0.1);
\draw (1.1,-0.075)--(1.9,-0.075);
\draw ({1.5-0.125},0.25) -- (1.5+0.08,0) -- (1.5-0.125,-0.25);
\draw [line width=0.5pt,line cap=round,rounded corners] (2.north west)  rectangle (3.south east);
\end{tikzpicture} $\ba_{\pi}=4a+12b,6a+4b$ for $r\in [2,4],[4,\infty)$. 
\item \begin{tikzpicture}[scale=0.5,baseline=-0.5ex]
\node [inner sep=0.8pt,outer sep=0.8pt] at (0,0) (1) {$\circ$};
\node [inner sep=0.8pt,outer sep=0.8pt] at (1,0) (2) {$\bullet$};
\node [inner sep=0.8pt,outer sep=0.8pt] at (2,0) (3) {$\circ$};
\node [inner sep=0.8pt,outer sep=0.8pt] at (3,0) (4) {$\circ$};
\node at (0,0.6) {};
\draw (0.12,0.01)--(0.88,0.01);
\draw (2.12,0.01)--(2.88,0.01);
\draw (1.1,0.1)--(1.9,0.1);
\draw (1.1,-0.075)--(1.9,-0.075);
\draw ({1.5-0.125},0.25) -- (1.5+0.08,0) -- (1.5-0.125,-0.25);
\draw [line width=0.5pt,line cap=round,rounded corners] (2.north west)  rectangle (3.south east);
\end{tikzpicture} $\ba_{\pi}=3a+6b,11a+2b$ for $r\in(0,1/2], [1/2,1]$. 
\item \begin{tikzpicture}[scale=0.5,baseline=-0.5ex]
\node [inner sep=0.8pt,outer sep=0.8pt] at (0,0) (1) {$\bullet$};
\node [inner sep=0.8pt,outer sep=0.8pt] at (1,0) (2) {$\bullet$};
\node [inner sep=0.8pt,outer sep=0.8pt] at (2,0) (3) {$\circ$};
\node [inner sep=0.8pt,outer sep=0.8pt] at (3,0) (4) {$\circ$};
\node at (0,0.6) {};
\draw (0.12,0.01)--(0.88,0.01);
\draw (2.12,0.01)--(2.88,0.01);
\draw (1.1,0.1)--(1.9,0.1);
\draw (1.1,-0.075)--(1.9,-0.075);
\draw ({1.5-0.125},0.25) -- (1.5+0.08,0) -- (1.5-0.125,-0.25);
\draw [line width=0.5pt,line cap=round,rounded corners] (1.north west)  rectangle (4.south east);
\end{tikzpicture} $\ba_{\pi}=-2a+11b,-a+9b,6b$ for $r\in [5/3,2],[2,3],[3,\infty)$. 
\item \begin{tikzpicture}[scale=0.5,baseline=-0.5ex]
\node [inner sep=0.8pt,outer sep=0.8pt] at (0,0) (1) {$\circ$};
\node [inner sep=0.8pt,outer sep=0.8pt] at (1,0) (2) {$\circ$};
\node [inner sep=0.8pt,outer sep=0.8pt] at (2,0) (3) {$\bullet$};
\node [inner sep=0.8pt,outer sep=0.8pt] at (3,0) (4) {$\bullet$};
\node at (0,0.6) {};
\draw (0.12,0.01)--(0.88,0.01);
\draw (2.12,0.01)--(2.88,0.01);
\draw (1.1,0.1)--(1.9,0.1);
\draw (1.1,-0.075)--(1.9,-0.075);
\draw ({1.5-0.125},0.25) -- (1.5+0.08,0) -- (1.5-0.125,-0.25);
\draw [line width=0.5pt,line cap=round,rounded corners] (1.north west)  rectangle (4.south east);
\end{tikzpicture} $\ba_{\pi}=3a,5a-b,11a-7b$ for $r\in (0,1/2],[1/2,1],[1,6/5]$.
\end{compactenum}\smallskip

\noindent Thus Conjecture~\ref{conj:cells} predicts the existence of elements of the affine Weyl group of type $\tilde{\sF}_4$ with the above $\ba$-function values in the respective parameter ranges. 
\end{example}

\begin{prop}
Conjecture~\ref{conj:strongconjecture} implies Conjecture~\ref{conj:upperbound}. 
\end{prop}

\begin{proof}
We have
$$
{\prod_{\alpha\in\Phi}}'\frac{1-\sq_{\alpha/2}^{-1}\sv^{\alpha^{\vee}}}{1-\sq_{\alpha/2}^{-1}\sq_{\alpha}^{-2}\sv^{\alpha^{\vee}}}={\prod_{\alpha\in\Phi^+}}'\frac{(1-\sq_{\alpha/2}^{-1}\sv^{\alpha^{\vee}})(1-\sq_{\alpha/2}^{-1}\sv^{-\alpha^{\vee}})}{(1-\sq_{\alpha/2}^{-1}\sq_{\alpha}^{-2}\sv^{\alpha^{\vee}})(1-\sq_{\alpha/2}^{-1}\sq_{\alpha}^{-2}\sv^{-\alpha^{\vee}})}.
$$
Suppose that $\alpha\in\Phi_0\cap\Phi_1$ (that is, $\alpha/2,2\alpha\notin\Phi$). Then $\sq_{\alpha}=\sq^a$ for some $a>0$, and writing $\sv^{\alpha^{\vee}}=\sq^{k}$ for some $k\in\ZZ$ gives
\begin{align*}
\frac{(1-\sq_{\alpha/2}^{-1}\sv^{\alpha^{\vee}})(1-\sq_{\alpha/2}^{-1}\sv^{-\alpha^{\vee}})}{(1-\sq_{\alpha/2}^{-1}\sq_{\alpha}^{-2}\sv^{\alpha^{\vee}})(1-\sq_{\alpha/2}^{-1}\sq_{\alpha}^{-2}\sv^{-\alpha^{\vee}})}&=
\frac{(1-\sq^k)(1-\sq^{-k})}{(1-\sq^{k-2a})(1-\sq^{k-2a})}
\end{align*}
Recalling the convention that any factors that are identically $0$ are removed, it follows that 
$$
\deg\frac{(1-\sq_{\alpha/2}^{-1}\sv^{\alpha^{\vee}})(1-\sq_{\alpha/2}^{-1}\sv^{-\alpha^{\vee}})}{(1-\sq_{\alpha/2}^{-1}\sq_{\alpha}^{-2}\sv^{\alpha^{\vee}})(1-\sq_{\alpha/2}^{-1}\sq_{\alpha}^{-2}\sv^{-\alpha^{\vee}})}=\begin{cases}
|k|&\text{if $0\leq |k|\leq 2a$}\\
2a&\text{if $2a\leq |k|$}.
\end{cases}
$$

If $\alpha\notin\Phi_0\cap\Phi_1$ (this only occurs in the non-reduced case) then we pair the four terms in the product related to $\alpha$ and $\alpha/2$ (if $\alpha\in\Phi_1$) or $2\alpha$ (if $\alpha\in\Phi_0$). We may thus assume that $\alpha\in\Phi_0$ (and so $\sq_{\alpha/2}=1$, $\sq_{\alpha}=\sq^{L(s_n)-L(s_0)}$, and $\sq_{2\alpha}=\sq^{L(s_0)}$), and the four terms combine to give
\begin{align*}
C_{\alpha}=&\frac{(1-\sv^{\alpha^{\vee}})(1-\sv^{-\alpha^{\vee}})(1-\sq_{\alpha}^{-1}\sv^{\alpha^{\vee}/2})(1-\sq_{\alpha}^{-1}\sv^{-\alpha^{\vee}/2})}{(1-\sq_{\alpha}^{-2}\sv^{\alpha^{\vee}})(1-\sq_{\alpha}^{-2}\sv^{-\alpha^{\vee}})(1-\sq_{\alpha}^{-1}\sq_{2\alpha}^{-2}\sv^{\alpha^{\vee}/2})(1-\sq_{\alpha}^{-1}\sq_{2\alpha}^{-2}\sv^{-\alpha^{\vee}/2})}\\
&=\frac{(1-\sv^{\alpha^{\vee}})(1-\sv^{-\alpha^{\vee}})}{(1+\sq^{c-a}\sv^{\alpha^{\vee}/2})(1+\sq^{c-a}\sv^{-\alpha^{\vee}/2})(1-\sq^{-a-c}\sv^{\alpha^{\vee}/2})(1-\sq^{-a-c}\sv^{-\alpha^{\vee}/2})},
\end{align*}
where $a=L(s_n)$ and $c=L(s_0)$. Recall that by Convention~\ref{conv:sym} we have $a-c>0$. Writing $\sv^{\alpha^{\vee}}=\sq^{2k}$ for some $k\in\ZZ$ then 
$$
\deg C_{\alpha}=\begin{cases}
2|k|&\text{if $0\leq |k|\leq a-c$}\\
|k|+a-c&\text{if $a-c\leq |k|\leq a+c$}\\
2a&\text{if $2a\leq |k|$}.
\end{cases}
$$

In summary, once the terms in the product are suitably grouped together we may write the product as ${\prod_{\alpha\in\Phi_0^+}}'C_{\alpha}$ where $\deg C_{\alpha}\geq 0$ with $\deg C_{\alpha}=0$ if and only if $\sv^{\alpha^{\vee}}=1$. Thus 
$$
L(\sw_0)-\frac{1}{2}\deg {\prod_{\alpha\in\Phi}}'\frac{1-\sq_{\alpha/2}^{-1}\sv^{\alpha^{\vee}}}{1-\sq_{\alpha/2}^{-1}\sq_{\alpha}^{-2}\sv^{\alpha^{\vee}}}\leq L(\sw_0)
$$
with equality if and only if $\sv^{\alpha^{\vee}}=1$ for all $\alpha\in\Phi_0^+$, which in turn forces $J=\emptyset$.
\end{proof}

\begin{thm}\label{thm:verify1}
Conjectures~\ref{conj:strongconjecture} and~\ref{conj:cells} hold in the following cases.
\begin{compactenum}[$(1)$]
\item All weighted affine Hecke algebras in the case $J=\emptyset$.
\item All weighted affine Hecke algebras of dimension~$1$ or $2$ (rank $2$ or $3$). 
\end{compactenum}
\end{thm}

\begin{proof}
(1) follows from Example~\ref{ex:lowestcell} and \cite[Theorem~4.6]{Gui:08}, which shows that the set $\Gamma$ from Example~\ref{ex:lowestcell} is the lowest two sided Kazhdan-Lusztig cell of $(W,L)$, and it is well known that $\ba(w)=L(\sw_0)$ for all elements of this cell.

(2) The cases $\Phi\in\{\sA_1,\sBC_1,\sA_2\}$ are easy exercises and are omitted (see \cite[Figure~1]{Lus:85} and \cite[\S7]{Lus:03} for the decomposition of $\sA_2$, $\sA_1$, and $\sBC_1$ into cells). Conjecture~\ref{conj:cells} for $\Phi\in\{\sC_2,\sBC_2,\sG_2\}$ follows from the results of \cite{GP:19,GP:19b}. Thus it only remains to verify Conjecture~\ref{conj:strongconjecture} in the cases $\Phi\in\{\sC_2,\sBC_2,\sG_2\}$. 

Consider the case $\Phi=\sG_2$. The case $J=\emptyset$ is covered by~(1). Suppose that $J=\{1\}$. The only bounded representation with $J=\{1\}$ has $\sv_{\alpha_1}=-\sq^{-a}$, where we use the conventions of Example~\ref{ex:G21} ($L(s_1)=a$ and $L(s_2)=L(s_0)=b$). Thus we have $\sv^{\alpha_1^{\vee}}=\sq^{-2a}$ and $\sv^{\alpha_2^{\vee}}=-\sq^a$. Since $\Phi^{\vee}=\pm\{\alpha_1^{\vee},\alpha_1^{\vee}+3\alpha_2^{\vee},2\alpha_1^{\vee}+3\alpha_2^{\vee},\alpha_2^{\vee},\alpha_1^{\vee}+\alpha_2^{\vee},\alpha_1^{\vee}+2\alpha_2^{\vee}\}$ we have
\begin{align*}
{\prod_{\alpha\in\Phi}}'\frac{1-\sq_{\alpha/2}^{-1}\sv^{\alpha^{\vee}}}{1-\sq_{\alpha/2}^{-1}\sq_{\alpha}^{-2}\sv^{\alpha^{\vee}}}=\frac{(1-\sq^{-2a})(1-\sq^{2a})(1+\sq^a)^4(1+\sq^{-a})^2}{(1-\sq^{-4a})(1+\sq^{-3a})^2(1-\sq^{-2b})^2(1+\sq^{a-2b})^2(1+\sq^{-a-2b})^2}
\end{align*}
Thus, since $L(\sw_0)=3a+3b$, we have
$$
L(\sw_0)-\frac{1}{2}\deg{\prod_{\alpha\in\Phi}}'\frac{1-\sq_{\alpha/2}^{-1}\sv^{\alpha^{\vee}}}{1-\sq_{\alpha/2}^{-1}\sq_{\alpha}^{-2}\sv^{\alpha^{\vee}}}=\begin{cases}
3b&\text{if $a-2b<0$}\\
a+b&\text{if $a-2b\geq 0$}.
\end{cases}
$$
By \cite[Theorem~7.10]{GP:19} these are equal to the bound of the induced representation $\pi_{\{1\},\sv}$, confirming the Conjecture~\ref{conj:strongconjecture} in this case. 

The case $J=\{2\}$ is similar (using again the values of the bounds from \cite[Theorem~7.10]{GP:19}) and we omit the details. 

Consider now $\Phi=\sG_2$ and $J=I=\{1,2\}$. By Proposition~\ref{prop:classifybounded} there are three bounded $I$-parameter systems. The case $(\sv_{\alpha_1},\sv_{\alpha_2})=(-\sq^{-a},-\sq^{-b})$ is clearly bounded by $0$ and recognises the trivial Kazhdan-Lusztig cell $\{e\}$. Consider the case $(\sv_{\alpha_1},\sv_{\alpha_2})=(\sq^a,-\sq^{-b})$ (which is bounded if and only if $a/b\leq 3/2$). We have $\sv^{\alpha_1^{\vee}}=\sq^{2a}$ and $\sv^{\alpha_2^{\vee}}=\sq^{-2b}$, and hence 
\begin{align*}
&{\prod_{\alpha\in\Phi}}'\frac{1-\sq_{\alpha/2}^{-1}\sv^{\alpha^{\vee}}}{1-\sq_{\alpha/2}^{-1}\sq_{\alpha}^{-2}\sv^{\alpha^{\vee}}}\\
&\quad=\frac{(1-\sq^{2a})(1-\sq^{-2a+6b})(1-\sq^{4a-6b})(1-\sq^{-2b})(1-\sq^{2b})(1-\sq^{2a-2b})(1-\sq^{-2a+4b})}{(1-\sq^{-4a})(1-\sq^{-6b})(1-\sq^{-6a+6b})(1-\sq^{-4b})(1-\sq^{2a-6b})}
\end{align*}
(with the convention on removing zero products if required). It follows that 
$$
L(\sw_0)-\frac{1}{2}\deg{\prod_{\alpha\in\Phi}}'\frac{1-\sq_{\alpha/2}^{-1}\sv^{\alpha^{\vee}}}{1-\sq_{\alpha/2}^{-1}\sq_{\alpha}^{-2}\sv^{\alpha^{\vee}}}=\begin{cases}
3a-2b&\text{if $1\leq a/b\leq 3/2$}\\
a&\text{if $a/b\leq 1$},
\end{cases}
$$
agreeing with the bounds given in Example~\ref{ex:G21}. Moreover, the cell recognised by this representation is given in Example~\ref{ex:G21}, and by \cite[Figure~2]{GP:19} these cells are contained in Kazhdan-Lusztig cells for the relevant parameter values. 

The case $(\sv_{\alpha_1},\sv_{\alpha_2})=(-\sq^{-a},\sq^b)$ is similar, and we omit the details. Moreover, the analysis for the cases $\Phi=\sC_2$ and $\sBC_2$ is similar, using~\cite[Theorems~6.15, 6.21, 6.22]{GP:19b}, and we again omit the calculations.
\end{proof}

\begin{remark}\label{rem:origins}
We provide some comments on the origin of Conjecture~\ref{conj:strongconjecture}. Recall that for a finite dimensional weighted Hecke algebra $\cH_0$ the ``canonical trace'' $\mathrm{Tr}(\sum a_w T_w)=a_e$ on $\cH_0$ decomposes as a sum of irreducible characters as $\mathrm{Tr}=\sum m_{\lambda}\chi_{\lambda}$ where the elements $m_{\lambda}$ are rational functions in~$\sq$ known as the ``generic degrees'' of $\cH_0$ (see \cite[Chapter 11]{GP:00}). There is a connection between the generic degrees and Lusztig's $\ba$-function, stated roughly that if $w$ is in the cell ``corresponding'' to $\lambda$ then $\deg m_{\lambda}=2\ba(w)$ (see \cite{Geck:11}). 

In the affine case, the decomposition of the canonical trace takes the form of an integral over tempered representations of the affine Hecke algebra, and the generic degrees are replaced by the ``Plancherel measure''. The \textit{Macdonald $c$-function} is (see Macdonald~\cite{Mac:71})
$$
c(X)=\prod_{\alpha\in\Phi^+}\frac{1-\sq_{\alpha/2}^{-1}\sq_{\alpha}^{-2}X^{-\alpha^{\vee}}}{1-\sq_{\alpha/2}^{-1}X^{-\alpha^{\vee}}}.
$$
By Opdam's work \cite[Theorem~3.25]{Opd:04} the reciprocal of the term $\psi_{J,\sv}(\sq^{2L(\sw_0)}c(X)c(X^{-1}))'$ (where the prime indicates that any factors that are $0$ on evaluation by $\psi_{J,\sv}$ are to be omitted) appears as the mass of the character $\chi_{J,\sv}$ of $\pi_{J,\sv}$ in the Plancherel formula for $\Hext(L)$ (once the parameters $\sq$ and $\zeta_i$ are specialised appropriately and scalars are extended, see \cite[\S9.4]{GP:19} for further discussion). Thus since
$$
L(\sw_0)-\frac{1}{2}\deg{\prod_{\alpha\in\Phi}}'\frac{1-\sq_{\alpha/2}^{-1}\sv^{\alpha^{\vee}}}{1-\sq_{\alpha/2}^{-1}\sq_{\alpha}^{-2}\sv^{\alpha^{\vee}}}=\frac{1}{2}\deg\psi_{J,\sv}\bigg(\sq^{2L(\sw_0)}c(X)c(X^{-1})\bigg)'
$$
our conjecture can be seen as an affine analogue of the finite dimensional situation, giving conjectural connections between Kazhdan-Lusztig Theory and Opdam's Plancherel Theorem. 
\end{remark}

\section{The case $\Phi=\sA_n$ with $J=\{1,2,\ldots,n-1\}$}\label{sec:An}

In this section we illustrate the theory in the case $\Phi=\sA_n$ with $J=\{1,2,\ldots,n-1\}$. We will apply Theorem~\ref{thm:mainpath2} to prove Conjectures~\ref{conj:strongconjecture} and~\ref{conj:cells} in this case. 

The only $J$-parameter system for which $\pi_{J,\sv}$ is bounded is $\sv=(-\sq^{-1})_{\alpha\in\Phi_J}$ (see Theorem~\ref{thm:bound} and Proposition~\ref{prop:classifybounded}), and so the symbol $\sv$ will be suppressed in the notation. Thus, for example, we shall write $\pi_J=\pi_{J,\sv}$ and $\varpi_J=\varpi_{J,\sv}$. 

We have $\sR[\zeta_J]=\sR[\zeta]$ where $\zeta=\zeta_J^{\omega_1}$, since $\zeta_J^{\omega_i}=\zeta^i$ for $1\leq i\leq n$. Since $\sy_{\omega_1}=\sw_{\{2,3,\ldots,n-1\}}\sw_{\{1,2,\ldots,n-1\}}=s_1\cdots s_{n-1}$ we have
$$
\st_{\omega_1}=t_{\omega_1}\sy_{\omega_1}=(s_0s_ns_{n-1}\cdots s_2\sigma)(s_1\cdots s_{n-1})=s_0\sigma,
$$
where $\sigma\in\Sigma$ is given by $\sigma s_i\sigma^{-1}=s_{i+1}$ (with indices read cyclically modulo $n+1$, and so in particular $s_{n+1}=s_0$; recall the definition of $\Sigma$ from Section~\ref{sec:affineWeyl}). 

The following choice of fundamental domain leads to a matrix representation with very symmetric matrices.

\goodbreak

\begin{lemma}
The set $\Sigma$ is a fundamental domain for the action of $\TJ$ on $\WJ$.
\end{lemma}

\begin{proof}
We have $W^J=\{e,s_n,s_ns_{n-1},\ldots,s_ns_{n-1}\cdots s_2s_1\}$, and for $1\leq i\leq n$ we compute
\begin{align*}
\st_{\omega_1}^is_ns_{n-1}\cdots s_{n-i+1}=\sigma^i,
\end{align*}
hence the result
\end{proof}

Fix the following order on the basis:
\begin{align}\label{eq:orderedbasis}
(\varpi_J(X_{\sigma^{-1}}),\varpi_J(X_{\sigma^{-2}}),\ldots,\varpi_J(X_{\sigma^{-n}}),\varpi_J(X_{\sigma^{-n-1}})).
\end{align}
For $1\leq i\leq n+1$ we have $\varpi_J(X_{\sigma^{-i}})\cdot T_w=\varpi_J(X_e)\cdot T_{\sigma^{-i}(w)}X_{\sigma^{-i}}$, and it follows that 
\begin{align}\label{eq:permute}
[\pi_J(T_w)]_{i,j}=[\pi_J(T_{\sigma^{-i}(w)})]_{n+1,\sigma^{-i}(j)}\quad\text{for $1\leq i,j\leq n+1$}.
\end{align}

\begin{cor}
The matrices for $\pi_J(T_i)$, $0\leq i\leq n$, with respect to the ordered basis~(\ref{eq:orderedbasis}) are 
\begin{align*}
\pi_J(T_i)&=\begin{bmatrix}
-\sq^{-1}\mathsf{I}_{i-1}&0&0&0\\
0&0&\zeta&0\\
0&\zeta^{-1}&\sq-\sq^{-1}&0\\
0&0&0&-\sq^{-1}\mathsf{I}_{n-i}
\end{bmatrix}&\pi_J(T_0)&=\begin{bmatrix}
\sq-\sq^{-1}&0&\zeta^{-1}\\
0&-\sq^{-1}\mathsf{I}_{n-1}&0\\
\zeta&0&0
\end{bmatrix}
\end{align*}
where $\mathsf{I}_k$ is the $k\times k$ identity matrix.
\end{cor}

\begin{proof} 
This follows directly from Theorem~\ref{thm:mainpath2}. 
\end{proof}

\begin{defn}
Recall that $\Waff$ is the non-extended affine Weyl group. Let
$$
\Gamma_{!}=\{w\in W\backslash\{e\}\mid \text{$w$ has a unique reduced expression}\}.
$$
\end{defn}

By \cite[Chapter~12]{Bon:17} for any Coxeter group $(W,S)$ with weight function $L=\ell$ (equal parameters) the set $\Gamma_!$ forms a two-sided Kazhdan-Lusztig cell, and the right cells in $\Gamma_!$ are the sets $\{w\in\Gamma_!\mid D_L(w)=s\}$ for $s\in S$ (with $D_L$ the left descent set). In type $\tilde{\sA}_n$ the elements of $\Gamma_!$ are precisely the nontrivial elements with a reduced expression with no subwords $s_is_j$ with $m_{ij}=2$ or $s_is_js_i$ with $m_{ij}=3$ (by Tits' solution to the Word Problem).

\begin{remark} For irreducible affine Coxeter groups $\Waff$ the set $\Gamma_!$ is infinite if and only if $\Waff$ is of type $\tilde{\sA}_n$ or $\tilde{\sC}_n$ (see \cite[Proposition~12.1.14]{Bon:17}). We expect that the techniques of this section could be applied to the $\tilde{\sC}_n$ case, with the relevant set in that case being $J=\{2,3,\ldots,n\}$.
\end{remark}

If $x\in\Gamma_!$ then the unique reduced expression of $x$ is of the form $s_is_{i+1}s_{i+2}\cdots$ or $s_is_{i-1}s_{i-2}\cdots$ for some $0\leq i\leq n$ (recall the indices are read cyclically). The final generator that appears in the reduced expression will be important to keep track of, and so we denote
\begin{align*}
x^{\uparrow}_{i,j,\ell}&=s_is_{i+1}s_{i+2}\cdots s_j \\
x^{\downarrow}_{i,j,\ell}&=s_is_{i-1}s_{i-2}\cdots s_j
\end{align*}
where $\ell=\ell(x^{\uparrow}_{i,j,\ell})=\ell(x^{\downarrow}_{i,j,\ell})$  (note that many cycles are allowed in these expressions, and that there is a compatibility condition between $i,j,\ell$, however this condition will not play a role). Sometimes either the length or the final generator will not be important to keep track of, and in this case we will abbreviate $x^{\uparrow}_{i,j,\ell}$ to $x^{\uparrow}_{i,j}$ or $x^{\uparrow}_i$, and similarly for $x^{\downarrow}_{i,j,\ell}$. Thus, for example, $x^{\downarrow}_5$ represents a word of the form $s_5s_4s_3s_2s_1s_0s_ns_{n-1}\cdots$ with no constraint on the length or final generator.

\goodbreak

\begin{lemma}\label{lem:pathsinGamma}
Let $x\in \Gamma_!$ and let $p$ be a positively $J$-folded alcove path of type $x$. Then $\deg \cQ_J(p)\leq 1$ with equality if and only if $p$ has exactly one fold and no bounces. Moreover if $1\leq i\leq n+1$ then the paths $p$ of type $x\in\Gamma_!$ starting at $\sigma^{-i}$ with $\deg\cQ_J(p)=1$ are the paths
\begin{compactenum}[$(1)$]
\item $p=(i-1)i(i+1)(i+2)\cdots n12\cdots$, where the first step is a fold and all remaining steps are crossings, and 
\item $p=(i-1)(i-2)\cdots 210n(j+1)j$ for some $j$, where the final step is a fold and all other steps are crossings. 
\end{compactenum}
\end{lemma}

\begin{proof}
We argue by induction on the length of the path, with the result clear for paths of length~$1$. It is sufficient to consider paths starting at $e$, because starting at $\sigma^{-i}$ for some $i$ is equivalent to starting $e$ and performing a cyclic shift of the generators of the expression~$x$. Thus we consider paths of $x^{\uparrow}_i$ and $x^{\downarrow}_i$ starting at $e$ for each choice of $0\leq i\leq n$. 

In the following analysis we start with a straight path $p=i_1i_2\cdots$ (we omit the $s$'s to lighten the notation) and construct from it a positively $J$-folded path, starting from left to right and deciding if the steps are folds, bounces, or crossings. We will use the following coding to represent a partially positively $J$-folded alcove path: $\hat{i}$ (or $i^{\vee}$) for a fold (hence degree contribution $+1$ to $\cQ_J(p)$), $\check{i}$ (or $i^{\vee}$) for a bounce (hence degree contribution $-1$ to $\cQ_J(p)$), and $\overline{i}$ for a crossing (hence zero degree contribution). Parts of the path that are grey are understood to remain ``straight'', and are yet to be determined. For example, given that $p$ starts at $e$, the expression $p=\overline{n}(n-1)^{\vee}\overline{n}\check{1}\overline{0}\bla{1234}$ represents a path where the first $5$ steps have been positively $J$-folded (a crossing, followed by a fold, a crossing, a bounce, and a crossing), and the remaining $4$ steps have not yet been determined. An expression like $[123]^{\vee}$ indicates that all three steps are bounces. 

Consider first paths $p$ of type $x=x^{\uparrow}_i$ starting at $e$. If $i=0$ then $p=\overline{0}\,\overline{1}\,\overline{2}\cdots $ consists entirely of positive crossings staying in $\cA_J$, and hence is positively $J$-folded with $\deg\cQ_J(p)=0$. 

If $i=n$ then either $p=\hat{n}\bla{012}\cdots$ or $p=\overline{n}\,\bla{012}\cdots$. In the first case the remainder of the path consists entirely of positive crossings staying in $\cA_J$, and so $\deg\cQ_J(p)=1$ and the path is as in case (1) of the statement of the lemma. We must show that in the second case the degree is bounded by $0$. In the second case the next $n-1$ steps are forced bounces and the following step is a positive crossing, giving $p=\overline{n}\,[012\cdots(n-2)]^{\vee}(\overline{n-1})\bla{n01}\cdots$. The remainder of the path is a path $p'$ of type $x^{\uparrow}_{n-1}$ starting at $s_n$. Note that $\st_{\omega_1}s_n=s_0\sigma s_n=\sigma$, and so by the action of $\TJ$ on paths (see Lemma~\ref{lem:preserveQ}) the path $\st_{\omega_1}\cdot p'$ starts at $\sigma$ and $\deg\cQ_J(p')=\deg\cQ_J(\st_{\omega_1}\cdot p')\leq 1$ (by the induction hypothesis). Since $\cQ_J(p)=(-\sq^{-1})^{n-1}\cQ_J(p')$ we have $\deg\cQ_J(p)\leq 2-n\leq 0$ as required. 

If $i\in J$ then the first $n-i$ steps are forced bounces, and so $p=[i\cdots(n-1)]^{\vee}\bla{n012}\cdots$. The remainder of the path is a path $p'$ of type $\bla{n012}\cdots$ starting at $e$, and by the $i=n$ case in the previous paragraph we have $\deg\cQ_J(p)\leq i-n+1\leq 0$. 

Now consider paths $p$ of type $x=x^{\downarrow}_i$ starting at $e$. If $i=0$ then $p=\overline{0}\bla{n(n-1)}\cdots$. The next $n-1$ steps are necessarily bounces, giving $p=\overline{0}[n(n-1)\cdots 2]^{\vee}\bla{10n(n-1)}\cdots$. The remainder of the path is a path $p'$ of type $x^{\downarrow}_1$ starting at $s_0$. Since $\st_{\omega_1}^{-1}s_0=\sigma^{-1}$ the path $\st_{\omega_1}^{-1}\cdot p'$ starts at $\sigma^{-1}$ and by induction $\deg\cQ_J(p')=\deg\cQ_J(\st_{\omega_1}^{-1}\cdot p')\leq 1$. Then $\deg\cQ_J(p)=-n+1+\deg\cQ_J(p')\leq 2-n\leq 0$. 

If $i=n$ then $p=\bla{n(n-1)(n-2)}\cdots$ has each step negative and stays in~$\cA_J$. Thus we either have no folds, and so $p=\overline{n(n-1)(n-2)\cdots}$ has $\deg\cQ_J(p)=0$, or we fold at some step. Thus suppose that $p=\overline{n(n-1)\cdots (j+1)}\hat{j}\bla{(j-1)}\cdots \bla{10n(n-1)}\cdots$ (possibly with multiple cycles before the fold). Either the path terminates immediately after the fold (and so $\deg\cQ_J(p)=1$ and $p$ is as in case (2) of the statement of the lemma), or there is a sequence of forced bounces. In fact the following $n-1$ steps (or up to the end of the path, whichever is first) are forced to be bounces. To see this, one checks that $s_ns_{n-1}\cdots s_{j+1}\alpha\in\Phi_J$ for $\alpha\in\{\alpha_{j-1},\ldots,\alpha_1,\varphi,\alpha_n,\ldots,\alpha_{j+1}\}$. Thus  
\begin{center}
$p=\overline{n(n-1)\cdots(j+1)}\hat{j}[(j-1)\cdots 10n(n-1)\cdots(j+1)]^{\vee}\bla{j(j-1)}\cdots$.
\end{center}
The remainder of the path is a path $p'$ of type $x^{\downarrow}_j$ starting at $s_n\cdots s_{j+1}$, and since $\st_{\omega_1}^{n-j}s_n\cdots s_{j+1}=\sigma^{n-j}$ we use Lemma~\ref{lem:preserveQ} and the induction hypothesis to give 
$$
\deg\cQ_J(p)=1-n+\deg\cQ_J(\st_{\omega_1}^{n-j}\cdot p')\leq 2-n\leq 0.
$$

Finally, if $i\in J$ then $p=[i(i-1)\cdots 1]^{\vee}\bla{0n(n-1)}\cdots$ and we return to the $i=0$ case, concluding the proof.
\end{proof}

Let $\Mat_{n+1}(\ZZ[\sq^{-1}])$ denote the ring of $(n+1)\times (n+1)$ matrices with entries in $\ZZ[\sq^{-1}]$. For $1\leq i,j\leq n+1$ let $E_{i,j}$ be the $(n+1)\times (n+1)$ matrix with $1$ in the $(i,j)$-th place, and zeros elsewhere.

\begin{cor}\label{cor:leading1}
We have
\begin{align*}
\pi_J(T_{x^{\uparrow}_{i,j,\ell}})&\in\sq\zeta^{\ell-1}E_{i+1,j+1}+\Mat_{n+1}(\ZZ[\sq^{-1}])\\
\pi_J(T_{x^{\downarrow}_{i,j,\ell}})&\in\sq\zeta^{-\ell+1}E_{i+1,j+1}+\Mat_{n+1}(\ZZ[\sq^{-1}])
\end{align*}
\end{cor}

\begin{proof}
This follows from Theorem~\ref{thm:mainpath2} and Lemma~\ref{lem:pathsinGamma} and the fact that in (1) of Lemma~\ref{lem:pathsinGamma} we have $\wt(p,\Sigma)^J=(\ell(p)-1)\overline{\omega}_n$, and in (2) we have $\wt(p,\Sigma)^J=(-\ell(p)+1)\overline{\omega}_n$.
\end{proof}

\begin{lemma}\label{lem:rows}
Let $0\leq i,j\leq n$ with $m_{ij}\in\{2,3\}$ and let $\sw_{ij}$ be the longest element of the parabolic subgroup $\langle s_i,s_j\rangle$. Let $y\in\Gamma_!\cup\{e\}$ and suppose that $\ell(\sw_{ij}y)=\ell(\sw_{ij})+\ell(y)$. Then
\begin{compactenum}[$(1)$]
\item $\pi_J(T_{\sw_{ij}y})\in\Mat_{n+1}(\ZZ[\sq^{-1}])$, and 
\item if $\deg[\pi_J(T_{\sw_{ij}y})]_{k,l}= 0$ then $k\in \{i+1,j+1\}$. 
\end{compactenum}
\end{lemma}

\begin{proof}
By~(\ref{eq:permute}) it is sufficient to prove the result for $k=n+1$. That is, we must show that if $0\leq i,j\leq n$ with $m_{ij}\in\{2,3\}$, and if $y\in\Gamma_!\cup\{e\}$ with $\ell(\sw_{ij}y)=\ell(\sw_{ij})+\ell(y)$, then $\deg[\pi_J(T_{\sw_{ij}y})]_{n+1,l}\leq 0$ for all $1\leq l\leq n+1$, and if $\deg[\pi_J(T_{\sw_{ij}y})]_{n+1,l}=0$ then either $i=n$ or $j=n$. By Theorem~\ref{thm:mainpath2} we have
$$
[\pi_J(T_{\sw_{ij}y})]_{n+1,l}=\sum_p\cQ_J(p)\zeta_J^{\wt(p,\Sigma)}
$$
where the sum is over paths $p\in\cP_J(\sw_{ij}y,e)$ with $\theta(p,\Sigma)=\sigma^{-l}$. To fix notation, we will choose the reduced expression $\sw_{ij}=s_is_j$ or $\sw_{ij}=s_is_js_i$ with $j<i$. It will turn out to be sufficient to bound the degree of $\cQ_J(p)$ for each path $p\in\cP_J(\sw_{ij}y,e)$ with the exception of case (4) below, where cancellations come into play (meaning that the path-by-path degree is higher than the degree of the associated matrix entry, and more care is required). In the analysis below we adopt the same coding as in the proof of Lemma~\ref{lem:pathsinGamma}.

Let $p\in\cP_J(\sw_{ij}y,e)$. Suppose first that $m_{ij}=3$. 
\begin{compactenum}[$(1)$]
\item Suppose that $1\leq j<i\leq n-1$. Then $p=\check{i}\check{j}\check{i}\bla{y}$. The remaining path is of type $y\in\Gamma_!$, and hence has degree bounded by~$1$ by Lemma~\ref{lem:pathsinGamma}, and hence $\deg\cQ_J(p)\leq -2$. 
\item Suppose that $(j,i)=(0,1)$. Then $p=\check{1}\overline{01}\bla{y}$.  Let $s_r$ be the first generator in the unique reduced expression for $y$, and write $y=s_ry'$. Since $\sw_{ij}y$ is reduced we have $r\notin\{0,1\}$. If $r\neq 2$ or $r\neq n$ then $p=\check{1}\overline{01}\check{r}\bla{y'}$ and $\deg\cQ_J(p)\leq -1$ by Lemma~\ref{lem:pathsinGamma}.
If $r=2$ then $p=\check{1}\overline{012}\bla{y'}$ and either $\bla{y'}=\bla{34}\cdots$ or $\bla{y'}=\bla{10}\cdots$. In the first case 
all remaining steps are positive crossings staying in $\cA_J$, and so $p=\check{1}\overline{0123\cdots}$ and $\deg\cQ_J(p)= -1$. In the second case we must have $p=\check{1}\overline{012}\check{1}\cdots$ and $\deg\cQ_J(p)\leq -1$ by Lemma~\ref{lem:pathsinGamma}. The case $r=0$ is similar. 
\item Suppose that $(j,i)=(0,n)$. Then either $p=\overline{n}\,\check{0}\,\overline{n}\bla{y}$ or $p=\hat{n}\overline{0}\check{n}\bla{y}$. In the first case Lemma~\ref{lem:pathsinGamma} gives $\deg\cQ_J(p)\leq 0$, and so consider the second case. Let $s_r$ be the first generator in the unique reduced expression for $y$, and write $y=s_ry'$. Since $\sw_{ij}y$ is reduced we have $r\notin\{0,n\}$. If $2\leq r\leq n-1$ then $s_{\varphi}\alpha_r=\alpha_r\in\Phi_J$ and so $s_0s_r\notin\cA_J$, giving $p=\hat{n}\overline{0}\check{n}\check{r}\bla{y'}$ and so again $\deg\cQ_J(p)\leq 0$. If $r=1$ then $p=\hat{n}\overline{0}\check{n}\overline{1}\bla{y'}$. Either $\bla{y'}=\bla{234}\cdots$ or $\bla{y'}=\bla{0n(n-1)}\cdots$. In the first case all remaining steps are positive crossings staying in $\cA_J$, and so $p=\hat{n}\overline{0}\check{n}\overline{1234\cdots}$ with $\deg\cQ_J(p)=0$. In the second case since $s_0s_1s_0\notin \cA_J$ we have $p=\hat{n}\overline{0}\check{n}\overline{1}\check{0}\bla{n(n-1)}\cdots$ and Lemma~\ref{lem:pathsinGamma} gives $\deg\cQ_J(p)\leq 0$ as required. 
\item Suppose that $(j,i)=(n-1,n)$. The possibilities are $p=\overline{n(n-1)}\check{n}\bla{y}$, $p=\hat{n}(n-1)^{\vee}\overline{n}\bla{y}$, $p=\overline{n}(n-1)^{\wedge}\overline{n}\bla{y}$, or $p=\hat{n}(n-1)^{\vee}\hat{n}\bla{y}$. In the first case Lemma~\ref{lem:pathsinGamma} gives $\deg\cQ_J(p)\leq 0$. In the second case, with $y=s_ry'$ as above, since $r\notin \{n,n-1\}$ the next step is forced to be a bounce, giving $p=\hat{n}(n-1)^{\vee}\overline{n}\check{r}\bla{y'}$ with $\deg\cQ_J(p)\leq 0$. Thus consider the third and fourth cases. Note that the initial segments $\overline{n}(n-1)^{\wedge}\overline{n}$ and $\hat{n}(n-1)^{\vee}\hat{n}$ both end at $e$. Thus for each $p'\in\cP_J(y,e)$ we have a pair of paths $p_1,p_2\in \cP_J(\sw_{n-1,n}y,e)$ given by appending each initial segment to the beginning of $p'$. Then 
$$
\cQ_J(p_1)+\cQ_J(p_2)=(\sq-\sq^{-1})\cQ_J(p')-\sq^{-1}(\sq-\sq^{-1})^2\cQ_J(p')=(\sq^{-1}+\sq^{-3})\cQ_J(p').
$$
By Lemma~\ref{lem:pathsinGamma} we have $\deg\cQ_J(p')\leq 1$, and so the combined contribution to the matrix entry $[\pi_J(T_{\sw_{n-1,n}y})]_{n+1,l}$ (where $\sigma^{-l}=\theta(p',\Sigma)$) has degree at most $0$, as required. 
\end{compactenum}
\smallskip

\noindent Now suppose that $m_{ij}=2$.
\begin{compactenum}
\item[$(5)$] Suppose that $1\leq j<i\leq n-1$. Then $p=\check{i}\check{j}\bla{y}$, and by Lemma~\ref{lem:pathsinGamma} we have $\deg\cQ_J(p)\leq -1$. 
\item[$(6)$] Suppose that $j=0$ and $1\leq i\leq n-1$. Then $p=\check{i}\overline{0}\bla{y}$. Let $s_r$ be the first generator in the unique reduced expression for $y$. Then $r\neq 0,i$ (as $\sw_{ij}y$ is reduced). If $r\in\{1,2,\ldots,n\}\backslash\{1,i\}$ then writing $y=s_ry'$ we have $p=\check{i}\overline{0}\check{r}\bla{y'}$, and as in the previous case the degree is bounded by~$-1$. So suppose that $r=1$. Then $p=\check{i}\overline{01}\bla{y'}$. Either $\bla{y'}=\bla{234}\cdots$ or $\bla{y'}=\bla{0n(n-1)}\cdots$. In the first case all remaining steps are positive crossings staying in $\cA_J$, and so we have $p=\check{i}\overline{0123\cdots}$ with $\deg\cQ_J(p)=-1$. In the second case, since $s_{\varphi}s_{\alpha_1}\varphi\in \Phi_J$, we have $p=\check{i}\overline{01}\check{0}\bla{n(n-1)}\cdots$. By Lemma~\ref{lem:pathsinGamma} the degree contribution of the remainder of the path is bounded by $1$, hence $\deg\cQ_J(p)\leq -1$. 
\item[$(7)$] Suppose that $i=n$. Then $j\in J\backslash\{n-1\}$ (because the cases $j=0,n-1$ are impossible as $m_{ij}=2$). Thus $s_{\alpha_n}\alpha_j\in\Phi_J$ and so $p=\hat{n}\check{j}\bla{y}$ or $p=\overline{n}\check{j}\bla{y}$. In the second case Lemma~\ref{lem:pathsinGamma} gives $\deg\cQ_J(p)\leq 0$ as required. Thus consider the first case. As before, let $s_r$ be the first generator in the unique reduced expression for $y$, and write $y=s_ry'$. Since $\sw_{ij}y$ is reduced we have $r\notin\{j,n\}$. If $r\in J$ then $p=\hat{n}\check{j}\check{r}\bla{y'}$, and so again $\deg\cQ_J(p)\leq 0$. If $r=0$ then $\bla{y'}=\bla{123}\cdots$ or $\bla{y}=\bla{n(n-1)(n-2)}\cdots$ and so $p=\hat{n}\check{j}\overline{0}\bla{123}\cdots$ or $p=\hat{n}\check{j}\overline{0}\bla{n(n-1)(n-2)}\cdots$. In the first case the remainder of the path consists of positive crossings staying in $\cA_J$, and hence $\deg\cQ_J(p)=0$. In the second case we have $s_0s_n\notin\cA_J$ and hence $p=\hat{n}\check{j}\overline{0}\check{n}\bla{(n-1)(n-2)}\cdots$, and again by Lemma~\ref{lem:pathsinGamma} we have $\deg\cQ_J(p)\leq 0$. 
\end{compactenum}
Hence the result.
\end{proof}

\begin{cor}\label{cor:join} Let $x\in\Gamma_!$ and $y\in\Gamma_!\cup\{e\}$ and $0\leq i<j\leq n$. Let $\sw_{ij}$ be the longest element of $W_{ij}=\langle s_i,s_j\rangle$, and suppose that $\ell(x\sw_{ij}y)=\ell(x)+\ell(\sw_{ij})+\ell(y)$. Then $\pi_J(T_w)\in\Mat_{n+1}(\ZZ[\sq^{-1}])$. 
\end{cor}

\begin{proof}
We have
$$
\pi_J(T_{x\sw_{ij}y})=\pi_J(T_x)\pi_J(T_{\sw_{ij}y}).
$$
Let $s_r$ be the last generator in the unique reduced expression for~$x$. By Corollary~\ref{cor:leading1} we have $\pi_J(T_x)\in \sq\zeta^aE_{k,r+1}+\Mat_{n+1}(\ZZ[\sq^{-1}])$ for some $1\leq k\leq n+1$ and $a\in\ZZ$. Moreover, by Lemma~\ref{lem:rows} $\pi_J(T_{\sw_{ij}y})$ satisfies
$$
\pi_J(T_{\sw_{ij}y})\in \sum_{l=1}^{n+1}(f_lE_{i+1,l}+g_lE_{j+1,l})+\sq^{-1}\Mat_{n+1}(\ZZ[\sq^{-1}]),
$$
for some $f_l,g_l\in\ZZ[\zeta,\zeta^{-1}]$. Since $\ell(x\sw_{ij}y)=\ell(x)+\ell(y)+\ell(\sw_{ij})$ we have $r\neq i,j$, and hence $\pi_J(T_x)\pi_J(T_{\sw_{ij}y})\in\Mat_{n+1}(\ZZ[\sq^{-1}])$ as required. 
\end{proof}

\goodbreak

\begin{cor}\label{cor:outofGamma}
If $w\in\Waff\backslash\Gamma_!$ then $\deg\pi(T_w)\leq 0$. 
\end{cor}

\begin{proof}
The result is clear if $w=e$. If $w\neq e$ choose any reduced expression for $w$, and write this expression in the form $w=y_1\sw_1 y_2\sw_2y_3\cdots y_m\sw_m y_{m+1}$ where $y_k\in \Gamma_!\cup\{e\}$ and $\sw_k$ is either of the form $sts$ with $m_{st}=3$ or $st$ with $m_{st}=2$. We have
$$
\pi_J(T_w)=\pi_J(T_{y_1\sw_1y_2})\pi_J(T_{\sw_2})\pi_J(T_{y_3\sw_3 y_4})\cdots,
$$
and applying Lemma~\ref{lem:rows}(1) and Corollary~\ref{cor:join} to see that each matrix is in $\Mat_{n+1}(\ZZ[\sq^{-1}])$, hence the result. 
\end{proof}

The following theorem shows, in particular, that Conjectures~\ref{conj:strongconjecture} and~\ref{conj:cells} hold for the case $\Phi=\sA_n$ and $J=\{1,2,\ldots,n-1\}$. Recall the definition of leading matrices from Definition~\ref{defn:cellandleading}.

\begin{thm}
Let $\Phi=\sA_n$ and $J=\{1,2,\ldots,n-1\}$. Let $\sF$ be any fundamental domain for the action of $\TJ$ on $\WJ$, and let $\Pi=(\pi_J,M_J,\sB_{\sF})$. 
\begin{compactenum}[$(1)$]
\item The matrix representation $\Pi$ is bounded with bound~$\ba_{\Pi}=1$.
\item We have $\Gamma_{\Pi}=\{w\sigma\mid w\in\Gamma_!,\,\sigma\in\Sigma\}$.
\item The leading matrices $\{\fc(w)\mid w\in\Gamma_{\Pi}\cap\Waff\}$ form a basis of the $\ZZ$-module $\Mat_{n+1}(\ZZ[\zeta,\zeta^{-1}])$.
\end{compactenum}
\end{thm}

\begin{proof}
Corollaries~\ref{cor:leading1} and~\ref{cor:outofGamma} show for $w\in\Waff$ the degree of the matrix entries of $\pi_J(T_w)$ is bounded by $1$, and that this bound is attained in the matrix $\pi_J(T_w)$ if and only if $w\in\Gamma_!$. Since $\pi_J(T_{\sigma})$ is a monomial matrix with entries in $\ZZ[\zeta,\zeta^{-1}]$ both (1) and (2) follow. The statement (3) follows from Corollary~\ref{cor:leading1}. 
\end{proof}

%
%
%
%
%
\bibliographystyle{plain}

\begin{thebibliography}{10}

\bibitem{AB:08}
P.~{A}bramenko and K.~Brown.
\newblock {\em Buildings: Theory and Applications}, volume 248.
\newblock Graduate Texts in Mathematics, Springer, 2008.

\bibitem{Bon:17}
C.~{B}onnaf\'e.
\newblock {\em Kazhdan-Lusztig cells with unequal parameters}, volume~24 of
  {\em Algebra and Applications}.
\newblock Springer, 2017.

\bibitem{Bou:02}
N.~{B}ourbaki.
\newblock {\em Lie groups and {L}ie algebras, Chapters 4--6}.
\newblock Elements of Mathematics. Springer-Verlag, 2002.

\bibitem{BH:93}
B.~{B}rink and R.~{H}owlett.
\newblock A finiteness property and an automatic structure for {C}oxeter
  groups.
\newblock {\em Math. Ann.}, 296:179--190, 1993.

\bibitem{CGLP:24}
N.~{C}hapelier-{L}aget, J.~{G}uilhot, E.~{L}ittle, and J.~{P}arkinson.
\newblock {T}he asymptotic {P}lancherel formula and {L}usztig's asymptotic algebra for $\tilde{\mathsf{A}}_n$.
\newblock {\em Preprint}, 2024.

\bibitem{Deo:87}
V.~{D}eodhar.
\newblock {O}n some geometric aspects of {B}ruhat orderings {II}: {T}he
  parabolic analogue of {K}azhdan-{L}usztig polynomials.
\newblock {\em J. Algebra}, 111:483--506, 1987.

\bibitem{Deo:90}
V.~{D}eodhar.
\newblock {A} combinatorial setting for questions in {K}azhdan-{L}usztig
  theory.
\newblock {\em Geom. Dedicata}, 36:95--119, 1990.

\bibitem{Deo:91}
V.~{D}eodhar.
\newblock {D}uality in parabolic set up for questions in {K}azhdan-{L}usztig
  theory.
\newblock {\em J. Algebra}, 142:201--209, 1991.

\bibitem{Deo:97}
V.~{D}eodhar.
\newblock ${J}$-chains and multichains, duality of {H}ecke modules, and
  formulas for parabolic {K}azhdan-{L}usztig polynomials.
\newblock {\em J. Algebra}, 190:214--225, 1997.

\bibitem{EW:14}
B.~{E}lias and G.~{W}illiamson.
\newblock The {H}odge theory of {S}oergel bimodules.
\newblock {\em Ann. of Math.}, 180(2):1089--1136, 2014.

\bibitem{GL:05}
S.~{G}aussent and P.~{L}ittelmann.
\newblock {LS} galleries, the path model, and {MV} cycles.
\newblock {\em {D}uke {M}ath. {J}.} 127:35--88, 2005.


\bibitem{Geck:02}
M.~{G}eck.
\newblock {C}onstructible characters, leading coefficients and left cells for
  finite {C}oxeter groups with unequal parameters.
\newblock {\em {R}epresent. {T}heory}, 6:1--30, 2002.

\bibitem{Geck:11}
M.~{G}eck.
\newblock {O}n {I}wahori-{H}ecke algebras with unequal parameters and
  {L}usztig's {I}somorphism {T}heorem.
\newblock {\em {P}ure {A}ppl. {M}ath.}, 7:587--620, 2011.

\bibitem{GP:00}
M.~{G}eck and G.~{P}feiffer.
\newblock {\em Characters of finite {C}oxeter groups and {I}wahori-{H}ecke
  algebras}, volume~21 of {\em London Mathematical Society Monographs}.
\newblock Oxford University Press, 2000.

\bibitem{Goe:07}
U.~{G}\"ortz.
\newblock Alcove walks and nearby cycles on affine flag manifolds.
\newblock {\em J. Algebraic Combin.}, 26:415--430, 2007.

\bibitem{Gui:08}
J.~{G}uilhot.
\newblock {O}n the lowest two-sided cell in affine {W}eyl groups.
\newblock {\em {R}epresent. {T}heory}, 12:327--345, 2008.

\bibitem{GP:19}
J.~{G}uilhot and J.~{P}arkinson.
\newblock A proof of {L}usztig's conjectures for affine type ${G}_2$ with
  arbitrary parameters.
\newblock {\em Proc. London Math Soc.}, 2019.


\bibitem{GP:19b}
J.~{G}uilhot and J.~{P}arkinson.
\newblock Balanced representations, the asymptotic {P}lancherel formula, and
  {L}usztig's conjectures for $\tilde{C}_2$.
\newblock {\em Algebraic Combinatorics}, 2:969--1031, 2019.



\bibitem{HNW:16}
C.~{H}ohlweg, P.~{N}adeau, and N.~{W}illiams.
\newblock Automata, reduced words and {G}arside shadows in {C}oxeter groups.
\newblock {\em J. Algebra}, 457:431--456, 2016.

\bibitem{KL:79}
D.~A. Kazhdan and G.~Lusztig.
\newblock Representations of {C}oxeter groups and {H}ecke algebras.
\newblock {\em Invent. Math}, 53:165--184, 1979.


\bibitem{Lit:94}
P.~Littelmann.
\newblock {A} {L}ittlewood-{R}ichardson rule for symmetrizable {K}ac-{M}oody algebras.
\newblock {\em {I}nvent. {M}ath.}, 116:329--346, 1994.

\bibitem{Lit:95}
P.~Littelmann.
\newblock {P}aths and root operators in representation theory.
\newblock {\em {A}nnals of {M}athematics}, 142:499--525. 

\bibitem{Lus:85}
G.~Lusztig.
\newblock {C}ells in affine {W}eyl groups.
\newblock In {\em {A}lgebraic groups and related topics}, volume~6 of {\em Adv.
  Stud. Pure Math.}, pages 255--287, 1985.

\bibitem{Lus:89}
G.~Lusztig.
\newblock {A}ffine {H}ecke algebras and their graded version.
\newblock {\em Journal of the American Mathematical Society}, 2(3):599--693,
  1989.

\bibitem{Lus:97}
G.~Lusztig.
\newblock {P}eriodic ${W}$-graphs.
\newblock {\em {R}epresent. {T}heory}, 1:207--279, 1997.

\bibitem{Lus:03}
G.~Lusztig.
\newblock {\em {H}ecke algebras with unequal parameters}.
\newblock CRM Monograph Series. Amer. Math. Soc., Providence, RI, 2003.




\bibitem{Mac:71}
I.~G. {M}acdonald.
\newblock {\em Spherical functions on a group of $p$-adic type}, volume~2 of
  {\em Publications of the Ramanujan Institute}.
\newblock Ramanujan Institute, Centre for Advanced Studies in Mathematics,
  University of Madras, 1971.

\bibitem{Mac:03}
I.~G. {M}acdonald.
\newblock {\em Affine Hecke algebras and orthogonal polynomials}, volume 157.
\newblock CUP, 2003.


\bibitem{MNST:22}
E.~Mili\'cevi\'c, Y.~Naqvi, P.~Schwer, and A.~Thomas.
\newblock {A} gallery model for affine flag varieties via chimney retractions.
\newblock {\em {T}ransformation {G}roups}, 29:773--821, 2022.


\bibitem{MST:19}
E.~Mili\'cevi\'c, P.~Schwer, and A.~Thomas.
\newblock {D}imensions of affine {D}eligne-{L}usztig varieties: {A} new approach via labelled folded alcove walks.
\newblock {\em {M}em. {A}mer. {M}ath. {S}oc.}, 261, 2019.

\bibitem{NR:03}
K.~{N}elsen and A.~{R}am.
\newblock {K}ostka-{F}oulkes polynomials and {M}acdonald spherical functions.
\newblock In {\em {S}urveys in {C}ombinatorics}, volume 307 of {\em {L}ondon
  {M}ath. {S}oc. {L}ect. {N}otes}, pages 325--370. Cambridge University Press,
  2003.

\bibitem{Opd:04}
{E}. {O}pdam.
\newblock {O}n the spectral decomposition of affine {H}ecke algebras.
\newblock {\em J. Inst. Math. Jussieu}, 3:531--648, 2004.

\bibitem{PRS:09}
J.~{P}arkinson, A.~{R}am, and C.~{S}chwer.
\newblock Combinatorics in affine flag varieties.
\newblock {\em J. Algebra}, 321:3469--3493, 2009.

\bibitem{Ram:03}
A.~{R}am.
\newblock {A}ffine {H}ecke algebras and generalized standard {Y}oung tableaux.
\newblock {\em J. Alg.}, 260:367--415, 2003.

\bibitem{Ram:06}
A.~{R}am.
\newblock Alcove walks, {H}ecke algebras, spherical functions, crystals and
  column strict tableaux.
\newblock {\em Pure and Applied Mathematics Quarterly (special issue in honor
  of Robert MacPherson)}, 2(4):963--1013, 2006.

\bibitem{Rou:01}
G.~Rousseau.
\newblock {E}xercises m\'etriques immobiliers.
\newblock {\em {I}ndag. {M}ath.} 12:383--405, 2001.


\bibitem{Sch:06}
C.~Schwer.
\newblock {G}alleries, {H}all-{L}ittlewood polynomials, and structure constants of the spherical {H}ecke algebra. 
\newblock {\em {I}nt. {M}ath. {R}es. {N}ot.} 2006.


\end{thebibliography}

\bigskip
  \footnotesize

  \noindent J\'er\'emie Guilhot, \textsc{Institut Denis Poisson, Universit\'e de Tours,
    37200 Tours, France}\par\nopagebreak
 \noindent \textit{E-mail address:} \texttt{Jeremie.Guilhot@lmpt.univ-tours.fr}
 \smallskip
 
  \noindent Eloise Little, \textsc{School of Mathematics and Statistics, University of Sydney, NSW 2006, Australia}\par\nopagebreak
 \noindent \textit{E-mail address:} \texttt{E.Little@maths.usyd.edu.au}
\smallskip

 \noindent James Parkinson, \textsc{School of Mathematics and Statistics, University of Sydney, NSW 2006, Australia}\par\nopagebreak
 \noindent \textit{E-mail address:} \texttt{jamesp@maths.usyd.edu.au}

\end{document}